\documentclass[11pt,a4paper]{article}
\usepackage{macros}
\usepackage[margin=0.75in]{geometry}

\newif\ifappendix
\appendixfalse 
\ifappendix
    \usepackage{xr}
    \title{Electronic Companion: Nash Equilibria, 
    Regularization and Computation in OT-Based DRO}
\else
  \title{Nash Equilibria, 
Regularization and Computation in\\ Optimal Transport-Based Distributionally Robust Optimization}
\fi

\author[$\dagger$]{Soroosh Shafiee}
\author[$\ddagger$]{Liviu Aolaritei}
\author[$\dagger\dagger$]{Florian D\"{o}rfler}
\author[$\ddagger\ddagger$]{Daniel Kuhn}
\affil[$\dagger$]{Operations Research and Information Engineering, Cornell University, Ithaca, USA 
\texttt{shafiee@cornell.edu}}
\affil[$\ddagger$]{Department of Electrical Engineering and Computer Sciences, UC Berkeley, Berkeley, USA 
\texttt{liviu.aolaritei@berkeley.edu}}
\affil[$\dagger\dagger$]{Automatic Control Lab, ETH Z\"urich, Z\"urich, Switzerland \protect\\ \texttt{dorfler@ethz.ch}}
\affil[$\ddagger\ddagger$]{Risk Analytics and Optimization Chair, EPFL Lausanne, Switzerland  \protect\\
\texttt{daniel.kuhn@epfl.ch}}
\date{}

\begin{document}
\maketitle

\ifappendix
    \setcounterref{equation}{eq:robust:portfolio}
    \setcounterref{lemma}{lem:solvability-dual}
    \setcounterref{corollary}{corollary:Wasserstein-p:regularization}
    \setcounterref{assumption}{assumption:derivatives}
    \setcounterref{proposition}{proposition:portfolio}
    \setcounterref{theorem}{theorem:nonconvex:duality}
    \appendix

\section{Computation of Nash Equilibria Revisited}
\label{appendix:nash:ml}

All results of Section~\ref{section:nash:computation} generalize to decision problems where~$Z = (X, Y)$ consists of two random vectors~$X\in\R^{d_x}$ and~$Y\in\R^{d_y}$ with $d_x+d_y=d$ and where the loss function is piecewise concave in~$X$ but may display an arbitrary dependence on~$Y$---provided that the adversary cannot move probability mass along the $Y$-space. Propositions~\ref{proposition:primal} and~\ref{proposition:dual} as well as Theorem~\ref{theorem:nash:computation} immediately carry over to this generalized setting with obvious minor modifications at the expense of a higher notational overhead that could obfuscate the proof ideas. For this reason, said generalizations of Propositions~\ref{proposition:primal} and~\ref{proposition:dual} and of Theorem~\ref{theorem:nash:computation}, which are needed for the numerical experiments in Section~\ref{sec:nash-experiments}, are stated below without proofs. These generalizations rely on the following updated version of Assumption~\ref{assumption:nash:convexity}.

\begin{assumption}[Convexity conditions revisited]~
    \label{assumption:nash:convexity:ml}
    \begin{enumerate}[label=(\roman*)]
        \item \label{assumption:nash:cost:ml} The cost of moving probability mass from $(x,y)$ to $(\hat x, \hat y)$ is given by~$c (x,\hat x)$ if $y = \hat y$ and amounts to~$+\infty$ if $y \neq \hat y$. We also assume that~$c (x,\hat x)$ is lower semicontinuous in $(x,\hat x)$ and convex in~$x$.
        \item \label{assumption:nash:loss:ml} The loss function is representable as a pointwise maximum of finitely many saddle functions, that is, we have $\ell(\theta, x, y) = \max_{i \in [I]} \ell_{i}(\theta, x, y)$ for some $I \in \N$, where $\ell_{i}(\theta, x, y)$ is proper, convex and lower semicontinuous in $\theta$, while $-\ell_{i}(\theta, x, y)$ is proper, convex and lower semicontinuous in~$x$. 
        \item
        \label{assumption:nash:set:X:ml} The support set of the random vector $X$ is representable as $\cX = \{ x \in \R^{d_x} : f_k(x) \leq 0~\forall k \in [K] \}$ for some $K \in \N$, where each function $f_k(x)$ is proper, convex and lower semicontinuous. 
        \item
        \label{assumption:nash:set:Theta:ml}
        The feasible set is representable as $\Theta = \{ \theta \in \R^m: g_l(\theta) \leq 0~\forall l \in [L] \}$ for some $L \in \N$, where each function $g_l(\theta)$ is proper, convex and lower semicontinuous. 
    \end{enumerate}
\end{assumption}

We also maintain Assumption~\ref{assumption:nash:reference}, whereby the reference distribution $\hat \P$ is discrete, and we use~$\hat z_j=(\hat x_j,\hat y_j)$ for $j\in[J]$ to denote its atoms. Using this convention, we update Assumption~\ref{assumption:slater} as follows.

\begin{assumption}[Slater conditions]~
    \label{assumption:slater:ml}
    \begin{enumerate}[label=(\roman*)]
        \item \label{assumption:slater:X:ml} 
        For every $j \in [J]$, $\hat x_j \in \rint(\dom(c(\cdot, \hat x_j)))$ is a Slater point for the support set $\cX$. 
        \item \label{assumption:slater:Theta:ml}
        The feasible set $\Theta$ admits a Slater point.
    \end{enumerate}
\end{assumption}

We are now ready to generalize Proposition~\ref{proposition:primal}.

\begin{proposition}[Primal reformulation revisited]
    \label{prop:primal-reformulation-revisited}
    If Assumptions~\ref{assumption:nash:reference}, \ref{assumption:nash:convexity:ml} and~\ref{assumption:slater:ml}\,\ref{assumption:slater:X:ml} hold and if~$\eps > 0$, then the primal DRO problem~\eqref{eq:dro} has the same infimum as the finite convex program
    \begin{align}
    \label{eq:primal-revisited}
        \begin{array}{cl@{\,}l}
            \inf & \DS \lambda \eps + \sum_{j \in [J]} p_j s_j \\[1.5ex]
            \st & \theta \in \Theta, \, \lambda, \tau_{ijk} \in \R_+, \, s_j \in \R, \, \zeta^{\ell}_{ij}, \zeta^{c}_{ij}, \zeta^{f}_{ijk} \in \R^{d_x} & \forall i \in [I],\, j \in [J], \,k \in [K] \\[2ex]
            & \DS (- \ell_i)^{*2}(\theta, \zeta^{\ell}_{ij}, \hat y_j) + \lambda c^{*1}( \zeta^{c}_{ij} / \lambda, \hat x_j) + \sum_{k \in [K]} \tau_{ijk} f_k^*( \zeta^{f}_{ijk} / \tau_{ijk}) \leq s_j & \forall i \in [I],\, j \in [J] \\[1.5ex]
            & \DS \zeta^{\ell}_{ij} + \zeta^{c}_{ij} + \sum_{k \in [K]} \zeta^{f}_{ijk} = 0 & \forall i \in [I],\, j \in [J],
        \end{array}
    \end{align}
    where $(-\ell_i)^{*2}(\theta,\zeta,y)$ denotes the conjugate of $-\ell_i(\theta, x,y)$ with respect to~$x$ for fixed~$\theta$ and $y$, and where $c^{*1}(\zeta,\hat x)$ denotes the conjugate of $c(x,\hat x)$ with respect to~$x$ for fixed~$\hat x$.
\end{proposition}

Proposition~\ref{prop:dual-reformulation-revisited}  below generalizes Proposition~\ref{proposition:dual} from the main text.

\begin{proposition}[Dual reformulation revisited]
    \label{prop:dual-reformulation-revisited}
    If Assumptions~\ref{assumption:loss}\,\ref{assumption:loss:Fatou}, \ref{assumption:regularity}\,\ref{assumption:regularity:infty}, \ref{assumption:nash:reference}, \ref{assumption:nash:regularity}, \ref{assumption:nash:convexity:ml} and \ref{assumption:slater:ml} hold and~$\eps > 0$, then the dual DRO problem~\eqref{eq:dual:dro} has the same supremum as the finite convex program
    \begin{align}
    \label{eq:dual-revisited}
    \begin{array}{cl@{\,}l}
        \max & \DS - \sum_{i \in [I]} \sum_{j \in [J]} q_{ij} \ell_i^{*1}(\alpha_{ij}, \hat x_j + \xi_{ij} / q_{ij}, \hat y_j) - \sum_{l \in [L]} \nu_l& g^*_l(\beta_l / \nu_l) \\[1ex]
        \st & \DS q_{ij}, \nu_l \in \R_+, ~ \xi_{ij} \in \R^{d_x}, ~ \alpha_{ij}, \beta_l \in \R^m & \forall i \in [I], \, j \in [J], \, l \in [L] \\[1ex]
        & q_{ij} f_k(\hat x_j + {\xi_{ij}}/{q_{ij}}) \leq 0 & \forall i \in [I], \, j \in [J], \, k \in [K] \\[1ex]
        & \DS \sum_{i \in [I]} q_{ij} = p_j & \forall j \in [J] \\[1ex]
        & \DS \sum_{i \in [I]} \sum_{j \in [J]} \alpha_{ij} + \sum_{l \in [L]} \beta_l = 0 \\[1ex]
        &\DS \sum_{i \in [I]} \sum_{j \in [J]} q_{ij} \, c \big( \hat x_j+ {\xi_{ij}}/{q_{ij}}, \hat x_j) \leq \eps,
    \end{array}
    \end{align}
    where $\ell_i^{*1}(\alpha, x, y)$ denotes the conjugate of $\ell_i(\theta, x, y)$ with respect to~$\theta$ for fixed~$x$ and~$y$.
\end{proposition}

Given Propositions~\ref{prop:primal-reformulation-revisited} and~\ref{prop:dual-reformulation-revisited}, we are finally ready to generalize Theorem~\ref{theorem:nash:computation}.

\begin{theorem}[Computing Nash equilibria revisited]
    \label{theorem:nash:computation:ml}
    If Assumptions~\ref{assumption:loss}\,\ref{assumption:loss:Fatou}, \ref{assumption:regularity}\,\ref{assumption:regularity:infty}, \ref{assumption:nash:reference}, \ref{assumption:nash:regularity2}, \ref{assumption:nash:convexity:ml} and \ref{assumption:slater:ml} hold and $\eps>0$, then the optima of the primal and dual DRO problems~\eqref{eq:dro} and~\eqref{eq:dual:dro} match, and the following~hold.
    \begin{enumerate}[label=(\roman*)]
        \item \label{theorem:nash:computation:theta:ml}
        If $(\theta^\star, \lambda^\star, \{ s_j^\star \}_j, \{ \zeta^{\ell \star}_{ij}, \zeta^{h \star}_{ij} \}_{i,j}, \{ \tau_{ijk}^\star, \zeta^{f \star}_{ijk} \}_{i,j,k} )$ solves~\eqref{eq:primal-revisited}, then~$\theta^\star$ solves the primal DRO problem~\eqref{eq:dro}.
        \item \label{theorem:nash:computation:Q:ml}
        If $(\{\beta_l^\star\}_l, \{\nu_l^\star\}_l, \{q_{ij}^\star\}_{i,j}, \{\xi_{ij}^\star\}_{i,j}, \{\alpha_{ij}^\star\}_{i,j})$ solves~\eqref{eq:dual-revisited} with $\cI_j^\infty=\emptyset$ for every~$j\in[J]$, then the discrete distribution $\Q^\star = \sum_{j \in [J]} \sum_{i \in \cI_j^+} q_{ij}^\star \delta_{(\hat x_j+{\xi_{ij}^\star}/{q_{ij}^\star}, \hat y_j) }$ solves the dual DRO problem~\eqref{eq:dual:dro}, where
        \begin{align*}
            \cI_j^+ = \left\{ i \in [I] : q_{ij}^\star > 0 \right\}
            \quad \text{and} \quad
            \cI_j^\infty = \left\{ i \in [I] : q_{ij}^\star = 0,\, \xi_{ij}^\star \neq 0 \right\}.
        \end{align*}
    \end{enumerate}
\end{theorem}

\section{On the Solution of Distributionally Robust Learning Models over Linear Hypotheses with Quadratic Transportation Costs}
\label{section:convexity}

Consider a distributionally robust learning model over linear hypotheses of the form~\eqref{eq:dro:linear}, and assume that the ambiguity set is defined in terms of the transportation cost function $c(z, \hat z) = \| z - \hat z \|^2$, where $\|\cdot\|$ is an arbitrary norm on~$\R^d$. By Proposition~\ref{proposition:strong:duality}, the inner maximization problem in~\eqref{eq:dro:linear} over the probability distributions has the same optimal value as the univariate minimization problem
\begin{align}
    \label{eq:min:lambda}
    \inf_{\lambda \geq 0}\; \lambda \eps + \E_{\hat Z \sim \hat{\P}} \left[ \ell_{c}(\theta, \lambda, \hat Z) \right],
\end{align}
and by Theorem~\ref{theorem:nonconvex:duality}\,\ref{theorem:norm}, the $c$-transform satisfies 
\begin{align}
    \label{eq:Lc:2}
    \ell_c(\theta, \lambda, \hat z) =  \sup_{\gamma \in \R} \; L \big( \inner{\theta}{\hat z} + \gamma \| \theta \|_* \big) - \lambda \hspace{0.1ex} \gamma^2.
\end{align}
In the following we will establish easily checkable sufficient conditions for problem~\eqref{eq:Lc:2} to be convex. To this end, we first recall some key definitions from convex analysis. A differentiable loss function $L$ with derivative $L'$ is called $M$-smooth for some $M\geq 0$ if $| L'(s_1) - L'(s_2)| \leq M | s_1-s_2|$ for all $s_1$ and $s_2$ in the domain of $L$. 
The quadratic growth rate of a loss function $L$ is defined as $G = \inf \{ g \in \R_+: \sup_{s \in \R} L(s) - g s^2 < \infty \}$. The quadratic growth rate~$G$ of an $M$-smooth function satisfies~$G \leq M / 2$.

We first provide easily checkable sufficient conditions for problem~\eqref{eq:Lc:2} to be convex and solvable.

\begin{proposition}[Convexity and solvability of problem~\eqref{eq:Lc:2}]
\label{proposition:gamma}
    If the univariate loss function $L$ is convex, differentiable and $M$-smooth with quadratic growth rate $G$, then the following hold. 
    \begin{enumerate}[label=(\roman*)]
        \item \label{proposition:structural:unbounded}
        If $\lambda < G \| \theta \|_*^2$, then problem~\eqref{eq:Lc:2} is unbounded.
        \item \label{proposition:structural:finite}
        If $\lambda > G \| \theta \|_*^2$, then problem~\eqref{eq:Lc:2} is solvable, and any maximizer~$\gamma^\star$  satisfies $|\gamma^\star| \geq \frac{\| \theta \|_*}{2 \lambda} |L'(\inner{\theta}{\hat z})|$.
        \item \label{proposition:structural:unique}
        If $\lambda > M \| \theta \|_*^2 / 2$, then problem~\eqref{eq:Lc:2} is convex and admits a unique maximizer.
    \end{enumerate}
\end{proposition}
\begin{proof}
    As for assertion~\ref{proposition:structural:unbounded}, assume that $\lambda<G\|\theta\|^2_*$. Thus, there exist two strictly positive constants $g$ and $\delta$ such that $\lambda<g\|\theta\|^2_*$ and $g+\delta<G$. The definition of the $c$-transform in~\eqref{eq:Lc:2} then implies that
    \begin{align*}
        \ell_c(\theta, \lambda,\hat z)&= \sup_{s\in\R} L(s)-\lambda\left(\frac{s-\inner{\theta}{\hat z}}{\|\theta\|_*}\right)^2\geq \sup_{s\in\R} L(s)-g (s-\inner{\theta}{\hat z})^2\\
        & = \sup_{s\in\R} L(s)-(g+\delta)s^2 + g \left(s^2-(s-\inner{\theta}{\hat z})^2\right)+\delta s^2\\
        &\geq \sup_{s\in\R} L(s)-(g+\delta)s^2 + \inf_{s\in\R} \delta s^2 - g \inner{\theta}{\hat z}^2 + 2 g s \inner{\theta}{\hat z}  = +\infty,
    \end{align*}
    where the first equality exploits the variable substitution $s\gets \inner{\theta}{\hat z} + \gamma \| \theta \|_*$, and the first inequality holds because $\lambda<g\|\theta\|^2_*$. The second equality as well as the second inequality are elementary. The last equality, finally, holds because the supremum in the last line of the above expression evaluates to $+\infty$ (as $g+\delta< G$) and because the infimum is finite (as $\delta>0$). Thus, the claim follows.
    
    As for assertion~\ref{proposition:structural:finite}, assume that $\lambda>G\|\theta\|^2_*$. Thus, there exist two strictly positive constants $g$ and $\delta$ such that $\lambda>g\|\theta\|^2_*$ and $g-\delta>G$. The definition of the $c$-transform in~\eqref{eq:Lc:2} then implies that
        \begin{align*}
        \ell_c(\theta, \lambda,\hat z)
        & \leq \sup_{s\in\R} L(s)-(g-\delta)s^2 + g \left(s^2-(s-\inner{\theta}{\hat z})^2\right)-\delta s^2\\
        &\leq \sup_{s\in\R} L(s)-(g-\delta)s^2 + \sup_{s\in\R} -\delta s^2 - g \inner{\theta}{\hat z}^2 + 2 g s \inner{\theta}{\hat z}  < +\infty,
    \end{align*}
    where the first inequality follows from our assumption that $\lambda>g\|\theta\|^2_*$ and from elementary rearrangements familiar from the proof of assertion~\ref{proposition:structural:unbounded}, while the second inqueality exploits the sub-additivity of the supremum. The two suprema in the last line, finally, are both finite because $g-\delta> G$ and~$\delta>0$, respectively. Thus, $\ell_c(\theta, \lambda,\hat z)$ is indeed finite. By slightly adapting the above arguments, one can further show that the objective function of the maximization problem in~\eqref{eq:Lc:2} is continuous and bounded above by a function with compact superlevel sets, which implies that this problem admits indeed a maximizer~$\gamma^\star$. Details are omitted for brevity. The stipulated lower bound on~$|\gamma^\star|$ is trivial if $L'(\inner{\theta}{\hat z}) = 0$. From now on we may thus assume without loss of generality that $L'(\inner{\theta}{\hat z})>0$ (the case $L'(\inner{\theta}{\hat z})<0$ can be handled similarly). Clearly, $\gamma^\star$ must satisfy the first-order optimality condition
    \begin{align}
        \label{eq:optimality:gamma}
        L'(\inner{\theta}{\hat z} + \gamma^\star \| \theta \|_*)\| \theta \|_*  - 2 \lambda \gamma^\star = 0.
    \end{align}
    In the following we distinguish two cases depending on sign of $\gamma^\star$. First, if $\gamma^\star \geq 0$, then we have
    \begin{align*}
        |\gamma^\star | =\gamma^\star = \textstyle \frac{\| \theta \|_*}{2\lambda}L'(\inner{\theta}{\hat z} + \gamma^\star \| \theta \|_*) \geq \frac{\| \theta \|_*}{2\lambda} L'(\inner{\theta}{\hat z}),
    \end{align*}
     where the second equality follows from~\eqref{eq:optimality:gamma}, and the inequality exploits the convexity of~$L$, which ensures that $L'$ is non-decreasing. If~$\gamma^\star < 0$, on the other hand, then we have
     \begin{align*}
        \ell_c(\theta, \lambda , z)
        &= L \big( \inner{\theta}{\hat z} + \gamma^\star \| \theta \|_* \big) - \lambda \hspace{0.1ex} (\gamma^\star)^2\\
        &= L(\inner{\theta}{\hat z}) + \int_0^{\gamma^\star} \left[L'(\inner{\theta}{\hat z} + \gamma \| \theta \|_*) \| \theta \|_*- 2 \lambda \gamma\right] \mathrm{d}\gamma \\
        &\leq L(\inner{\theta}{\hat z}) + \int_0^{\gamma^\star} \left[L'(\inner{\theta}{\hat z} + \gamma^\star \| \theta \|_*) \| \theta \|_*- 2 \lambda \gamma\right] \mathrm{d}\gamma \\
        &= L(\inner{\theta}{\hat z}) + \int_0^{\gamma^\star} \left[ 2 \lambda \gamma^\star - 2 \lambda \gamma \right] \mathrm{d}\gamma
        = L(\inner{\theta}{\hat z}) + \lambda (\gamma^\star)^2,
    \end{align*}
    where the first two equalities follow from the optimality of~$\gamma^\star$ and the fundamental theorem of calculus, respectively, the inequality holds due to the monotonicity of $L'$ and the assumption that $\gamma^\star<0$, and the third equality follows from~\eqref{eq:optimality:gamma}. By the first-order convexity condition of~$L$, we further obtain
    \begin{align*}
        \ell_c(\theta, \lambda , \hat z) \geq \sup_{\gamma \in \R} ~ L(\inner{\theta}{\hat z}) + L'(\inner{\theta}{\hat z}) \gamma \| \theta \|_* - \lambda \gamma^2
        = \textstyle L(\inner{\theta}{\hat z}) + \frac{1}{4 \lambda} [L'(\inner{\theta}{\hat z}) \| \theta \|_*]^2.
    \end{align*}
    Combining the upper and lower bounds on~$\ell_c(\theta, \lambda , \hat z)$ derived above then yields
    \begin{align*}
        L(\inner{\theta}{\hat z}) +\lambda (\gamma^\star)^2 \geq \ell_c(\theta, \lambda , \hat z) \geq \textstyle L(\inner{\theta}{\hat z}) + \frac{1}{4 \lambda} [L'(\inner{\theta}{\hat z}) \| \theta \|_*]^2 \implies | \gamma^\star | \geq \frac{\| \theta \|_*}{2 \lambda} |L'(\inner{\theta}{\hat z})|.
    \end{align*}
    This completes the proof of assertion~\ref{proposition:structural:finite}.
    
    As for assertion~\ref{proposition:structural:unique}, assume that $\lambda > M \| \theta \|_*^2 / 2$. Since $L$ is $M$-smooth by assumption, one readily verifies that the objective function $L ( \inner{\theta}{\hat z} + \gamma \| \theta \|_* ) - \lambda \gamma^2$ of problem~\eqref{eq:Lc:2} is strongly concave in $\gamma$ for any $\lambda > M \| \theta \|_*^2 / 2$. Therefore, problem~\eqref{eq:Lc:2} must have a unique maximizer.
\end{proof}

The next proposition provides a priori bounds on the minimizers of problem~\eqref{eq:min:lambda}.

\begin{proposition}[Bounds on the minimizers of~\eqref{eq:min:lambda}]
    \label{proposition:lambda}
    If $L$ is convex, upper semicontinuous, differentiable and $M$-smooth with quadratic growth rate $G$, $L(\inner{\theta}{\hat Z})$ is $\hat{\P}$-integrable, $c(z, \hat z) = \| z - \hat z \|^2$ for some norm $\|\cdot\|$, and $\eps > 0$, then~\eqref{eq:min:lambda} has a minimizer 
    \[
        \lambda^\star \in \Big[ \sqrt{\E_{\hat Z \sim \hat{\P}} \big[ |L'(\inner{\theta}{\hat Z})|^2 \big] \| \theta \|_*^2 / (4 \eps)}, ~ \sqrt{\E_{\hat Z \sim \hat{\P}} \big[ |L'(\inner{\theta}{\hat Z})|^2 \big] \| \theta \|_*^2 / \eps} + M \| \theta \|_*^2 / 2 \Big].
    \]
\end{proposition}
\begin{proof}
    Let $\lambda^\star$ be any minimizer of problem~\eqref{eq:min:lambda}, which is guaranteed to exist for the following reasons. First, $\ell_c(\theta, \lambda,\hat Z)$ is convex and lower semicontinuous in $\lambda$. In addition, lower semi-continuity is preserved by taking expectations with respect to $\hat{\P}$ thanks to Fatou's lemma, which can be applied since $\ell_c(\theta, \lambda,\hat Z)$ is bounded below by the $\hat{\P}$-integrable random variable $L(\inner{\theta}{\hat Z})$. Finally, the objective function of problem~\eqref{eq:min:lambda} is bounded below by $\lambda \eps + \E_{\hat Z \sim \hat{\P}} [L(\inner{\theta}{\hat Z}]$, which grows linearly with~$\lambda$. Thus, the minimum of problem~\eqref{eq:min:lambda} must be attained at some point $\lambda^\star\in [0,+\infty)$. Consequently, we have
    \begin{align*}
        \lambda^\star \eps + \E_{\hat Z \sim \hat{\P}} [L(\inner{\theta}{\hat Z}] \leq \lambda^\star \eps + \E_{\hat Z \sim \hat{\P}} [\ell_c(\theta, \lambda^\star , \hat Z)] = \inf_{\lambda \geq 0} \lambda \eps + \E_{\hat Z \sim \hat{\P}} [\ell_c(\theta, \lambda , \hat Z)].
    \end{align*}
    By rearranging the above inequality and recalling the definition of~$\ell_c$, we then obtain
    \begin{align*}
        \lambda^\star \eps 
        &\leq \inf_{\lambda \geq 0} \lambda \eps + \E_{\hat Z \sim \hat{\P}} \left[ \sup_{\gamma \in \R}  ~ L(\inner{\theta}{\hat Z} + \gamma \| \theta \|_*) - L(\inner{\theta}{\hat Z}) - \lambda \gamma^2 \right] \\
        &\leq \inf_{\lambda \geq 0} \lambda \eps + \E_{\hat Z \sim \hat{\P}} \left[ \sup_{\gamma \in \R} ~  \frac{M}{2} \gamma^2 \| \theta \|_*^2 - \gamma \| \theta \|_* L'(\inner{\theta}{\hat Z}) - \lambda \gamma^2 \right] \\
        &\leq \inf_{\lambda > {M \| \theta \|_*^2}/{2}} \lambda \eps + \frac{\| \theta \|_*^2}{4 \lambda - 2 M \| \theta \|_*^2} \E_{\hat Z \sim \hat{\P}} \left[ |L'(\inner{\theta}{\hat Z})|^2 \right] \\
        &= \frac{\eps M \| \theta \|_*^2}{2} + \sqrt{\eps} \| \theta \|_* \E_{\hat Z \sim \hat{\P}} \left[ |L'(\inner{\theta}{\hat Z})|^2 \right]^{\frac{1}{2}},
    \end{align*}
    where the second inequality exploits Taylor's theorem and the $M$-smoothness of $L$, the third inequality holds because we restrict the feasible set of a minimization problem and because the inner maximization problem is solved by $\gamma^\star = -L'(\inner{\theta}{\hat Z) \|\theta\|_* / (2 \lambda - M \| \theta \|_*^2)}$ whenever $\lambda > M \| \theta \|_*^2 / 2$, and the last equality follows from the analytical solution of the minimization problem over~$\lambda$. The desired upper bound on~$\lambda^\star$ is then obtained by dividing both sides of the above inequality by~$\eps$.

    Note that the desired lower bound on $\lambda^\star$ trivially holds if~$\theta=0$. From now on we may thus assume without loss of generality that $\theta\neq 0$. From Proposition~\ref{proposition:gamma} we know that $\lambda^\star \geq G \|\theta\|_*^2$. Below we assume without much loss of generality that $\lambda^\star > G \|\theta\|_*^2$. Indeed, the case $\lambda^\star = G \|\theta\|_*^2$ can be handled similarly to Case~$2$ in the proof of \cite[Lemma~4]{blanchet2018optimal}. In addition, we temporarily denote the objective function of problem~\eqref{eq:min:lambda} by~$f(\lambda)$. The first-order optimality condition of problem~\eqref{eq:min:lambda} then implies that the right derivative of~$f$ at~$\lambda^\star$, which we denote by $\partial_\lambda f ((\lambda^\star)^+)$, must be non-negative, that is,
    \begin{align}
    \label{eq:right-derivative}
        0\leq \partial_\lambda f ((\lambda^\star)^+) = \eps + \E_{\hat Z \sim \hat{\P}} \left[ \partial_\lambda \ell_c(\theta,(\lambda^\star)^+,\hat Z) \right],
    \end{align}
    where the right derivative and the expectation may be interchanged thanks to~\cite[Proposition~2.1]{bertsekas1973stochastic}. Next, we can use the envelope theorem~\citep[Corollary~4]{milgrom2002envelope} to demonstrate that
    \begin{align}
    \label{eq:env:thm}
        \partial_\lambda \ell_c(\theta,(\lambda^\star)^+,\hat z) = \max_{\gamma^\star \in \Gamma^\star(\lambda^\star,\hat z)} -(\gamma^\star)^2
    \end{align}
    with $\Gamma^\star(\lambda^\star, \hat z) = \argmax \{ L(\inner{\theta}{\hat z} + \gamma \| \theta \|_*) - \lambda^\star \gamma^2 : \gamma \in \R \}$. The envelope theorem applies because $L \big( \inner{\theta}{\hat z} + \gamma \| \theta \|_* \big) - \lambda \hspace{0.1ex} \gamma^2$ is upper semicontinuous in $\gamma$, $\partial_\lambda (L \big( \inner{\theta}{\hat z} + \gamma \| \theta \|_* \big) - \lambda \hspace{0.1ex} \gamma^2) = - \gamma^2$ is trivially continuous in~$\gamma$ and~$\lambda$, and $\Gamma^\star(\lambda^\star, \hat z)$ is compact. To see that~$\Gamma^\star(\lambda^\star, \hat z)$ is compact, we need to show that it is closed as well as bounded. Closedness follows immediately from the upper semi-continuity of~$L$. To show boundedness, select any~$\gamma^\star \in \Gamma^\star(\lambda^\star, \hat z)$. Since $L$ has quadratic growth rate~$G$ and since $\lambda^\star > G \|\theta\|_*^2$ by assumption, there exist $\delta>0$ and $C>0$ such that $L(u)\leq (G+\delta)u^2 + C$ and $\lambda^\star \geq (G+2 \delta)\|\theta\|_*^2$. In addition, since $L$ is convex, we have $L(\inner{\theta}{\hat z}) \geq L(0) + L'(0)\inner{\theta}{\hat z}$. Substituting these three estimates into the inequality $\ell_c(\theta,\lambda^\star,\hat z)= L(\inner{\theta}{\hat z}+\gamma^\star\|\theta\|_*)-\lambda(\gamma^\star)^2\geq L(\inner{\theta}{\hat z})$, we obtain
    \begin{align*}
        ( G+\delta)(\inner{\theta}{\hat z} + \gamma^\star \|\theta\|_*)^2 + C - (G + 2\delta)(\gamma^\star)^2 \|\theta\|_*^2 \geq L(0) + L'(0)\inner{\theta}{\hat z},
    \end{align*}
    which is equivalent to
    \begin{align*}
        \left( \gamma^\star \|\theta\|_* - \frac{G+\delta}{\delta}\inner{\theta}{\hat z} \right)^2 \leq \frac{C - L(0) - L'(0)\inner{\theta}{\hat z}}{\delta} + \left( \frac{G+\delta}{\delta} + \frac{(G+\delta)^2}{\delta^2} \right) (\inner{\theta}{\hat z})^2.
    \end{align*}
    Recalling that $\theta\neq 0$, we may conclude that $\gamma^\star$ belongs to the sublevel set of a strictly convex quadratic function and is therefore bounded. As $\gamma^\star \in \Gamma^\star(\lambda^\star, \hat z)$ was chosen freely, $\Gamma^\star(\lambda^\star, \hat z)$ is indeed bounded. 
        
    We are now ready to construct the desired lower bound on $\lambda^\star$. Proposition~\ref{proposition:gamma}\,\ref{proposition:structural:finite} implies that every maximizer $\gamma^\star\in\Gamma^\star (\lambda^\star, \hat z)$ satisfies $| \gamma^\star | \geq \frac{\| \theta \|_*}{2 \lambda^\star} |L'(\inner{\theta}{\hat z})|$, and thus we have
    \begin{align*}
        \frac{\| \theta \|_*^2}{4 (\lambda^\star)^2}\, \E_{\hat Z \sim \hat{\P}} \left[ |L'(\inner{\theta}{\hat Z})|^2 \right] 
        \leq \E_{\hat Z \sim \hat{\P}} \left[ \min_{\gamma^\star \in \Gamma^\star(\lambda^\star,\hat Z)} (\gamma^\star)^2 \right] = - \E_{\hat Z \sim \hat{\P}} \left[ \partial_\lambda \ell_c(\theta,(\lambda^\star)^+,\hat Z) \right]
        \leq \eps,
    \end{align*}
    where the equality and the second inequality follow from~\eqref{eq:env:thm} and~\eqref{eq:right-derivative}, respectively. The desired lower bound on $\lambda^\star$ finally follows from elementary rearrangements of the above inequality.
\end{proof}

\begin{corollary}
    \label{corollary:lambda}
    Suppose that all assumptions of Proposition~\ref{proposition:lambda} hold, and define the constants
    \[
        \underline{L} = \inf_{\theta \in \Theta} \E_{\hat Z \sim \hat{\P}}[|L'(\inner{\theta}{\hat Z}|^2]\quad \text{and} \quad R = \sup_{\theta \in \Theta} \| \theta \|_*.
    \]
    If $0<\eps < \underline{L} / (R M)^2$, then problem~\eqref{eq:min:lambda} is solved by some $\lambda^\star > M \| \theta \|_*^2 / 2$.
\end{corollary}
\begin{proof}
    By Proposition~\ref{proposition:lambda}, problem~\eqref{eq:min:lambda} has a minimizer~$\lambda^\star$ that satisfies
    \begin{align*}
        \lambda^\star \geq \sqrt{\E_{\hat Z \sim \hat{\P}} \big[ |L'(\inner{\theta}{\hat Z})|^2 \big]} \frac{\| \theta \|_*}{2 \sqrt{\eps}} > \sqrt{\E_{\hat Z \sim \hat{\P}} \big[ |L'(\inner{\theta}{\hat Z})|^2 \big]}\frac{RM \| \theta \|_* }{2 \sqrt{\underline{L}}} \geq \frac{RM \| \theta \|_*}{2} \geq \frac{M \| \theta \|_*^2}{2},
    \end{align*}
    where the second inequality holds because $\eps < \underline{L} / (R M)^2$, the third inequality follows from the definition of $\overline{L}$, and the last inequality follows from the definition of $R$. Thus the claim follows.
\end{proof}

Corollary~\ref{corollary:lambda} implies that if $0 < \eps < \underline{L} / (R M)^2$, then problem~\eqref{eq:min:lambda} is equivalent to
\begin{align}
    \label{eq:efficient-SGD}
    \inf_{\lambda > M \| \theta \|_*^2 / 2}\; \lambda \eps + \E_{\hat Z \sim \hat{\P}} \left[ \ell_{c}(\theta, \lambda, \hat Z) \right].
\end{align}
By Proposition~\ref{proposition:gamma}\,\ref{proposition:structural:unique}, evaluating the $c$-transform $\ell_{c}(\theta, \lambda, \hat z)$ at any feasible~$\lambda$ thus amounts to solving a univariate convex maximization problem. Thanks to the envelope theorem~\citep[Theorem~2.16]{de2000mathematical}, the solution of this maximization problem also provides a subgradient of $\ell_{c}(\theta, \lambda, \hat z)$. Thus, problem~\eqref{eq:efficient-SGD} is amenable to stochastic subgradient descent algorithms, where all subgradients can be computed by solving univariate convex optimization problems. By extension, the same conclusions apply to the distributionally robust learning model~\eqref{eq:dro:linear} if we replace the inner maximization problem with~\eqref{eq:efficient-SGD}.

\section{Additional Proofs}
\begin{proof}[Proof of Proposition~\ref{prop:non-uniqueness}]
    The conic duality theorem~\citep[Theorem~1.4.2]{ben2001lectures} ensures that the minimization problem~\eqref{eq:reg:svm}, which trivially admits a Slater point, and the maximization problem~\eqref{eq:dual:svm:light} are strong duals. Thus, problem~\eqref{eq:dual:svm:light} is solvable and has the same optimal value as problem~\eqref{eq:reg:svm}. Theorem~\ref{theorem:nash:computation:ml} further implies that problems~\eqref{eq:dual:svm:light} and~\eqref{eq:dual:svm} share the same optimal values, too. In the following we use any maximizer~$(\{\tilde q^\star_j\}_{j})$ of~\eqref{eq:dual:svm:light} to construct a family of optimal solutions~$(\{q_{ij}^\star(\alpha), \xi^\star_{ij}(\alpha)\}_{i,j})$ for problem~\eqref{eq:dual:svm} parametrized by~$\alpha\in A$. Specifically, we set $q_{1j}^\star(\alpha) = {1}/{J} - \tilde q_{j}^\star$, $q_{2j}^\star(\alpha) = \tilde q_{j}^\star$ and $\xi_{1j}^\star(\alpha) = 0$ for all $j \in [J]$, $\xi_{2j}^\star(\alpha) = 0$ for all $j \in [J]$ with $\tilde q_j^\star = 0$, and $\xi_{2j}^\star(\alpha) = \alpha_j \hat{y}_j \tilde \xi$  for all $j \in [J]$ with  $\tilde q_j^\star > 0$.
    By construction, $(\{q_{ij}^\star(\alpha), \xi^\star_{ij}(\alpha)\}_{i,j})$ is feasible in~\eqref{eq:dual:svm}, and its objective function value in~\eqref{eq:dual:svm} matches the optimal value of problem~\eqref{eq:dual:svm:light}, which equals the optimal value of problem~\eqref{eq:dual:svm}. Hence, $(\{q_{ij}^\star(\alpha), \xi^\star_{ij}(\alpha)\}_{i,j})$ solves~\eqref{eq:dual:svm} for any~$\alpha \in A$. In addition, one readily verifies that~$\cI_j^\infty(\alpha) = \{ i \in [I] : q_{ij}^\star(\alpha) = 0,\, \xi_{ij}^\star (\alpha)\neq 0 \}=\emptyset$ for all $j\in[J]$ and~$\alpha \in A$. This implies that Assumption~\ref{assumption:nash:regularity2} is satisfied irrespective of~$\alpha\in\cA$. By Theorem~\ref{theorem:nash:computation:ml}\,\ref{theorem:nash:computation:Q:ml} and the definitions of~$\cJ_0$ and~$\cJ_+$, the discrete distribution~$\Q^\star(\alpha)$ therefore represents a Nash strategy of nature for every~$\alpha \in A$. This observation completes the proof.
\end{proof}

\begin{proof}[Proof of Proposition~\ref{proposition:portfolio}]
    Assume first that~$\lambda < 1 / \| \theta \|_0$. Then, we have
    \begin{align*}
        \ell_c(\theta, \lambda, \hat z) 
        &\geq \lim_{\gamma \to -\infty} ~ 1 + \log(-\gamma)+ \sum_{i \in [d]} \lambda h(\gamma \theta_i \hat z_i/\lambda)  \\ 
        &= \lim_{\gamma \to -\infty} ~ 1 + \log(-\gamma)+ \sum_{i \in [d]:\,\theta_i>0} -\lambda - \lambda \log(-\gamma \theta_i \hat z_i / \lambda) \\
        &= \lim_{\gamma \to -\infty} ~ ( 1 - \lambda\, \| \theta \|_0 ) (1 + \log(-\gamma)) - \sum_{i \in [d]:\,\theta_i>0} \lambda \log(\theta_i \hat z_i / \lambda)=+\infty,
    \end{align*}
    where the first equality follows from the definition of~$h$, and the third equality holds because $\lambda < 1 / \| \theta \|_0$. 
    
    Next, assume that $\lambda > 1 / \| \theta \|_0$. Using a similar reasoning as above, one can then show that the objective function of problem~\eqref{eq:robust:portfolio} drops to~$-\infty$ as $\gamma$ either tends to~$-\infty$ or to~$0$. Therefore, problem~\eqref{eq:robust:portfolio} has a maximizer. In addition, as the objective function of problem~\eqref{eq:robust:portfolio} is strictly concave and differentiable, this maximizer is unique and fully determined by the first-order optimality condition
    \begin{align}
        \label{eq:optimality:condition}
        \frac{1}{\gamma} + \sum_{i \in \cD(\gamma)} u_i - \sum_{i \in [d] \backslash \cD(\gamma)} \frac{\lambda}{\gamma} = 0 \quad\iff\quad \gamma=\frac{(d - | \cD(\gamma) |) \lambda - 1}{\sum_{i \in \cD(\gamma)} u_i},
    \end{align}
    where $\cD(\gamma) = \{ i \in [d]: \gamma u_i \geq -\lambda \}$ and $u_i=\theta_i\hat z_i$ for every $i\in[d]$. To show that $\gamma^\star$ solves~\eqref{eq:optimality:condition} and~\eqref{eq:robust:portfolio}, it thus suffices to prove that $\cD(\gamma^\star)=\{\sigma(i):i\in[k]\}$. To this end, we first verify that the critical index $k$ is well-defined as the maximal element of a non-empty finite set. This is indeed the case because
    \begin{align*}
        (1 - (d-1) \lambda)u_{\sigma(1)} \leq \left(1 - \frac{d-1}{d} \right)u_{\sigma(1)}= \frac{u_{\sigma(1)}}{d}  \leq \lambda u_{\sigma(1)},
    \end{align*}
    where both inequalities follow from our assumption that~$\lambda > 1 / \| \theta \|_0 > 1/d$ and from the non-negativity of~$u_{\sigma(1)}$. We may thus conclude that $k\geq 1$. Next, we use induction to show that $\sigma(i)\in\cD(\gamma^\star)$ for every $i\in [k]$. As for the base step corresponding to~$i=k$, we may use the definition of~$k$ to find
    \begin{align*}
        (1 - (d-k) \lambda) u_{\sigma(k)} \leq \lambda \sum_{j \in [k]} u_{\sigma(j)} \quad\iff \quad \gamma^\star u_{\sigma(k)} \geq -\lambda,
    \end{align*}
    where the equivalence exploits the definition of~$\gamma^\star$ in the proposition statement. Hence,  $\sigma(k)\in\cD(\gamma^\star)$. As for the induction step, assume that $\sigma(i)\in\cD(\gamma^\star)$ for some $i \in [k]$ with $i>1$. Thus, we have
    \begin{align*}
        (1 - (d-i) \lambda) u_{\sigma(i)} \leq \lambda \sum_{j \in [i]} u_{\sigma(j)} 
        \quad \implies \quad
        (1 - (d-(i-1)) \lambda) u_{\sigma(i)} \leq \lambda \sum_{j \in [i-1]} u_{\sigma(j)}.
    \end{align*}
    Since the permutation~$\sigma$ sorts the parameters $\{u_i\}_{i\in[d]}$ in ascending order, the last inequality implies 
    \begin{align*}
        &(1 - (d-(i-1)) \lambda) u_{\sigma(i-1)} \leq (1 - (d-(i-1)) \lambda) u_{\sigma(i)} 
        \leq \lambda \sum_{j \in [i-1]} u_{\sigma(j)}.
    \end{align*}
    This in turn allows us to conclude that
    \begin{align*}
        (1 - (d-k) \lambda) u_{\sigma(i-1)} \leq \lambda \sum_{j \in [k]} u_{\sigma(j)} \quad \implies \quad \gamma^\star u_{\sigma(i-1)} \geq -\lambda,
    \end{align*}
    where the implication follows again from the definition of~$\gamma^\star$. Hence, $\sigma(i-1)\in\cD(\gamma^\star)$, which completes the induction step. We have thus shown that $\sigma(i) \in \cD(\gamma^\star)$ for all $i \leq k$. In addition, it is clear from the definition of the critical index $k$ that any $i > k$ satisfies
    \begin{align*}
        &(1 - (d-i) \lambda) u_{\sigma(i)} > \lambda \sum_{j \in [i]} u_{\sigma(j)} \quad \implies \quad (1 - (d-k) \lambda) u_{\sigma(i)} > \lambda \sum_{j \in [k]} u_{\sigma(j)} \quad \implies \quad \gamma^\star u_{\sigma(i)} < -\lambda,
    \end{align*}
    and thus $\sigma(i) \in [d] \backslash \cD(\gamma^\star)$. In summary, we have thus shown that  $\cD(\gamma^\star)=\{\sigma(i):i\in[k]\}$, which confirms via the optimality condition~\eqref{eq:optimality:condition} that $\gamma^\star$ solves indeed the maximization problem~\eqref{eq:robust:portfolio}.

    Finally, assume that~$\lambda = 1 / \| \theta \|_0$. By using similar arguments as in the last part of the proof, one can show that $\gamma^\star$ remains optimal in~\eqref{eq:robust:portfolio} but is no longer unique. Details are omitted for brevity.  
\end{proof}

\section{Technical Background Results}

\begin{lemma}
\label{lemma:conv:envelope}
    If $p\geq 1$ and  $L:\R \to \R$ is an upper semicontinuous function satisfying the growth condition $L(s) \leq C (1 + |s|^p)$ for some $C \geq 0$, then the $p$-th envelope $L_p(s, \lambda) = \sup_{s' \in \R}\; L(s')-\lambda | s-s' |^p$ converges to $L(s)$ as $\lambda$ grows, that is, $\lim_{\lambda \to +\infty} L_p(s,\lambda) = L(s)$ for all $s \in \R$.
\end{lemma}
\begin{proof}
    By the definition of $L_p$, for any $\lambda>0$ there exists $s_\lambda\in\R$ with
    \begin{align*}
        L_p(s,\lambda) \leq L(s_\lambda) - \lambda |s-s_\lambda|^p + \frac{1}{\lambda}.
    \end{align*}
    In the following we prove that $s_\lambda$ converges to $s$ as $\lambda$ tends to infinity. To this end, note first that
    \begin{align*}
        L(s) &\leq L_p(s,\lambda) \leq L(s_\lambda) - \lambda |s-s_\lambda|^p + \frac{1}{\lambda} \leq C(1+|s_\lambda|^p) - \lambda |s-s_\lambda|^p + \frac{1}{\lambda}\\ 
        &\leq C \left(1 + (|s| + |s-s_\lambda|)^p \right)- \lambda |s-s_\lambda|^p + \frac{1}{\lambda} \leq C(1+2^{p-1} |s|^p) - (\lambda - C 2^{p-1})|s-s_\lambda|^p + \frac{1}{\lambda},
    \end{align*}
    where the first and the second inequalities exploit the definitions of $L_p$ and $s_\lambda$, respectively, whereas the third and the fourth inequalities follow from the growth condition on $L$ and the triangle inequality, respectively. The last inequality, finally, holds because $p\geq 1$, which implies via Jensen's inequality that
    \[
        \textstyle (|s| + |s-s_\lambda|)^p = 2^{p} \left(\frac{1}{2}|s| + \frac{1}{2} |s-s_\lambda|\right)^p \leq 2^{p-1} \left(|s|^p + |s-s_\lambda|^p\right).
    \]
    Rearranging terms thus yields
    \begin{align*}
        |s-s_\lambda|^p \leq \left( C(1+2^{p-1} |s|^p) - L(s) + \frac{1}{\lambda} \right)/ (\lambda - C 2^{p-1}),
    \end{align*}
    which ensures that $s_\lambda$ converges to $s$ as $\lambda$ grows. Recalling the definition of $L_p(x,\lambda)$, we finally obtain
    \begin{align*}
        L(s) \leq \lim_{\lambda \to + \infty} L_p(s,\lambda) \leq \limsup_{\lambda \to +\infty} L(s_\lambda) - \lambda |s-s_\lambda|^p + \frac{1}{\lambda} \leq \limsup_{\lambda \to +\infty} L(s_\lambda) \leq L(s),
    \end{align*}
    where the second and the third inequalities follow from the definition of $s_\lambda$ and the upper semicontinuity of $L$, respectively. As both sides of the above expression coincide with $L(s)$, we conclude that all inequalities are in fact equalities. This proves that $\lim_{\lambda \to +\infty} L_p(s,\lambda) = L(s)$ for all $s \in \R$.
\end{proof}
    \bibliographystyle{myabbrvnat} 
    \bibliography{bibfile.bib}
\else
  \begin{abstract}
    We study optimal transport-based distributionally robust optimization problems where a fictitious adversary, often envisioned as nature, can choose the distribution of the uncertain problem parameters by reshaping a prescribed reference distribution at a finite transportation cost. In this framework, we show that robustification is intimately related to various forms of variation and Lipschitz regularization even if the transportation cost function fails to be (some power of) a metric. We also derive conditions for the existence and the computability of a Nash equilibrium between the decision-maker and nature, and we demonstrate numerically that nature's Nash strategy can be viewed as a distribution that is supported on remarkably deceptive adversarial samples. Finally, we identify practically relevant classes of optimal transport-based distributionally robust optimization problems that can be addressed with efficient gradient descent algorithms even if the loss function or the transportation cost function are nonconvex (but not both at the same time). 
\end{abstract}

\section{Introduction}

Because of their relevance for empirical risk minimization in machine learning, stochastic optimization methods are becoming increasingly popular beyond their traditional application domains in operations research and economics \citep{shapiro2014lectures}. In the wake of the ongoing data revolution and the rapid emergence of ever more complex decision problems, there is also a growing need for stochastic optimization models outputting reliable decisions that are insensitive to input misspecification and easy to compute. 

A (static) stochastic optimization problem aims to minimize the expected value~$\E_{Z\sim \P} [\ell(\theta, Z)]$ of an uncertainty-affected loss function~$\ell: \R^m \times \R^d \to (-\infty, \infty]$ across all feasible decisions~$\theta \in \Theta \subseteq \R^m$, where~$Z$ is the random vector of all uncertain problem parameters that is governed by some probability distribution~$\P$ supported on $\cZ \subseteq \R^d$. To exclude trivialities, we assume throughout the paper that the feasible set~$\Theta$ and the support set~$\cZ$ are non-empty and closed. Despite their simplicity, static stochastic optimization problems are ubiquitous in statistics and machine learning, among many other application domains. However, their practical deployment is plagued by a fundamental challenge, namely that the distribution~$\P$ governing~$Z$ is rarely accessible to the decision-maker. In a data-driven decision situation, for instance, $\P$ is only indirectly observable through a set of independent training samples. In this case one may use standard methods from statistics to construct a parametric or non-parametric reference distribution~$\hat{\P}$ from the data. However, $\hat \P$ invariably differs from~$\P$ due to inevitable statistical errors, and optimizing in view of~$\hat{\P}$ instead of~$\P$ may lead to decisions that display a poor performance on test data. For example, in machine learning it is well known that deep neural networks trained in view of the empirical distribution of the training data can easily be fooled by adversarial examples, that is, test samples subject to seemingly negligible noise that cause the neural network to make a wrong prediction. Even worse, the decision problem at hand could suffer from a distribution shift, that is, the training data may originate from a distribution other than~$\P$, under which decisions are evaluated.

A promising strategy to mitigate the detrimental effects of estimation errors in the reference distribution~$\hat{\P}$ would be to minimize the worst-case expected loss with respect to all distributions in some neighborhood of~$\hat{\P}$. There is ample evidence that, for many natural choices of the neighborhood of~$\hat{\P}$, this distributionally robust approach leads to tractable optimization models and provides a simple means to derive powerful generalization bounds \citep{esfahani2018data, lam2019recovering, blanchet2019confidence, van2020data, duchi2020learning, gao2023finite, duchi2021statistics}. Specialized distributionally robust approaches may even enable generalization in the face of domain shifts \citep{farnia2016minimax, volpi2018generalizing,  lee2018minimax, duchi2020learning} or may make the training of deep neural networks more resilient against adversarial attacks~\citep{sinha2018certifying, wang2019convergence, tu2019theoretical, kwon2020principled}. 

More formally, distributionally robust optimization (DRO) captures the uncertainty about the unknown distribution~$\P$ through an ambiguity set~$\B_{\eps}(\hat{\P})$, that is, an $\eps$-neighborhood of the reference distribution $\hat{\P}$ with respect to a distance function on the space $\cP(\cZ)$ of all distributions supported on~$\cZ$. Using this ambiguity set, the DRO problem of interest is formulated as the minimax problem
\begin{align}
    \label{eq:dro}
    \inf_{\theta \in \Theta} \sup_{\Q \in \B_{\eps}(\hat{\P})} ~ \E_{Z\sim\Q} \left[ \ell(\theta, Z) \right].
\end{align}
Problem~\eqref{eq:dro} can be viewed as a zero-sum game between the decision-maker, who chooses the decision~$\theta$, and a fictitious adversary, often envisioned as `nature', who chooses the distribution~$\Q$ of~$Z$. From now on we define nature's feasible set~$\B_{\eps}(\hat{\P}) =\{ \Q \in \cP(\cZ): d_c ( \Q , \hat{\P}) \leq \eps\}$ as a pseudo-ball of radius~$\eps\geq 0$ around~$\hat{\P}$ with respect to an optimal transport discrepancy $d_c:\cP(\cZ)\times \cP(\cZ)\to [0, +\infty]$ defined through 
\begin{equation*}
    d_c ( \P, \hat{\P}) = \inf_{\pi \in \Pi(\P, \hat{\P})} \E_{(Z,\hat Z)\sim\pi} [c(Z, \hat Z)].
\end{equation*}
Here, $c: \cZ \times \cZ \to [0, +\infty]$ is a prescribed transportation cost function satisfying the identity of indiscernibles ($c(z, \hat z) = 0$ if and only if $z = \hat z$), and $\Pi(\P, \hat{\P})$ represents the set of all joint probability distributions of~$Z$ and~$\hat Z$ with marginals~$\P$ and~$\hat{\P}$, respectively. It is conventional to refer to elements of $\Pi(\P, \hat{\P})$ as couplings or transportation plans. If~$c(z, \hat z) = \| z - \hat z \|^p$ for some norm~$\|\cdot\|$ on (the span of)~$\cZ$ and for some exponent~$p\in\N$, then~$W_p=\sqrt[p]{d_c}$ constitutes a metric on~$\cP_c(\cZ)$ \citep[\textsection~6]{villani2008optimal}, which is termed the $p$-th Wasserstein distance. Note that the ambiguity set~$\B_{\eps}(\hat{\P})$ constructed in this way may be interpreted as the family of all probability distributions~$\Q$ that can be obtained by reshaping the reference distribution~$\hat{\P}$ at a finite cost of at most~$\eps\geq 0$, where the cost of moving unit probability from $\hat z$ to~$z$ is given by~$c(z,\hat z)$.

The following regularity conditions will be assumed to hold throughout the paper.

\begin{assumption}[Continuity assumptions]~
    \label{assumption:continuity}
    \begin{enumerate}[label=(\roman*)]
        \item \label{assumption:cost:lsc} The transportation cost function $c(z,\hat z)$ is lower semicontinuous in $(z,\hat z)$.
        \item \label{assumption:loss:usc} For any $\theta \in \Theta$, the loss function $\ell(\theta, \hat z)$ is upper semicontinuous and $\hat{\P}$-integrable in~$\hat z$.
    \end{enumerate} 
\end{assumption}
Assumption~\ref{assumption:continuity}\,\ref{assumption:cost:lsc} ensures that the optimal transport problem in the definition of~$d_c(\P,\hat{\P})$ is solvable~\citep[Theorem~4.1]{villani2008optimal} and admits a strong dual linear program over a space of integrable functions~\citep[Theorem~5.10]{villani2008optimal}. Together, Assumptions~\ref{assumption:continuity}\,\ref{assumption:cost:lsc} and~\ref{assumption:continuity}\,\ref{assumption:loss:usc} ensure that the inner maximization problem in~\eqref{eq:dro} admits a strong dual minimization problem \cite[Theorem~1]{blanchet2019quantifying}; see also Proposition~\ref{proposition:strong:duality} below. Optimal transport-based DRO problems of the form~\eqref{eq:dro} are nowadays routinely studied in diverse areas such as statistical learning \citep{shafieezadeh2015distributionally, gao2024wasserstein, chen2018robust, blanchet2019multivariate, blanchet2019robust, chen2019selecting,  shafieezadeh2019regularization, ho2023adversarial,aolaritei2023wasserstein}, estimation and filtering \citep{shafieezadeh2018wasserstein, nguyen2023bridging, nguyen2020distributionally}, control \citep{aolaritei2023awasserstein,Yang2020b,coulson2021distributionally}, dynamical systems theory \citep{Boskos2021Data-drivenProcesses, aolaritei2022distributional}, hypothesis testing \citep{gao2018robust}, inverse optimization \citep{esfahani2018inverse}, and chance constrained programming \citep{chen2024data, xie2019distributionally, ho2020distributionally, ho2023strong, shen2020chance} etc.; see~\citep{kuhn2019wasserstein} for a recent survey. The existing literature focuses almost exclusively on Wasserstein ambiguity sets, which are obtained by setting the transportation cost function to~$c(z, \hat z) = \|z - \hat z\|^p$ for some norm~$\|\cdot\|$ on~$\R^d$. Allowing for more general transportation cost functions enables the decision-maker to inject prior information about the likelihood of certain subsets of~$\cZ$ into the definition of the ambiguity set. 
For example, setting $c(z, \hat z) = \infty\cdot\mathds{1}_{z\not\in\A(\hat z)}$ for some closed set $\A(\hat z)\subseteq \cZ$ with~$\hat z\in \A(\hat z)$ ensures that the probability mass located at~$\hat z$ cannot be moved outside of~$\A(\hat z)$. More generally, assigning $c(z, \hat z)$ a large value for every~$\hat z$ in the support of the reference distribution~$\hat{\P}$ makes it expensive for nature to transport probability mass to~$z$. Thus, $z$ has a low probability under every distribution in the ambiguity set. The following proposition describes a strong dual of nature's inner maximization problem in~\eqref{eq:dro} and reveals how~$c$ impacts the solution of the underlying DRO problem. This strong duality result was first established in~\citep{esfahani2018data, zhao2018data} for finite-dimensional problems with discrete reference distributions and then generalized in~\citep{blanchet2019quantifying, gao2023distributionally}. 

\begin{proposition}[{Strong duality }]
    \label{proposition:strong:duality}
    If Assumption~\ref{assumption:continuity} holds, then we have
    \begin{align}
    \label{eq:inner:supremum}
    \sup_{\Q \in \B_{\eps}(\hat{\P})}\; \E_{Z \sim \Q} \left[ \ell(\theta, Z) \right] 
    = \inf_{\lambda \geq 0}\; \lambda \eps + \E_{\hat Z \sim \hat{\P}} \left[ \ell_{c}(\theta, \lambda, \hat Z) \right]
    \end{align}
    for any $\theta \in \Theta$ and $\eps > 0$, where $\ell_{c}(\theta, \lambda, \hat z) = \sup_{z \in \cZ:\,c(z,\hat z)<\infty} \; \ell(\theta, z) - \lambda c(z, \hat z)$. 
\end{proposition}

Proposition~\ref{proposition:strong:duality} follows immediately from \cite[Theorem~1]{zhang2022simple}; see also \cite[Theorem~4.18]{kuhn2024distributionally}.
It is well known that checking whether a given distribution~$\Q\in \cP(\cZ)$ belongs to~$\B_{\eps}(\hat{\P})$ is $\#$P-hard even if~$\hat{\P}$ is discrete \citep[Theorem~2.2]{tacskesen2023semi}. Thus, unless $\#\text{P}=\text{FP}$, there is no efficient algorithm to check feasibility in the worst-case expectation problem on the left hand side of~\eqref{eq:inner:supremum}. Thanks to Proposition~\ref{proposition:strong:duality}, however, the optimal value of this problem can often be computed efficiently by solving the dual problem on the right hand side of~\eqref{eq:inner:supremum}, which constitutes a univariate convex stochastic program. Indeed, the dual objective function involves the expected value of~$\ell_{c}(\theta, \lambda, \hat Z)$ with respect to the reference distribution~$\hat{\P}$, which is convex in~$\lambda$. In the following, we will refer to~$\ell_{c}$ as the $c$-transform of~$\ell$. A key challenge towards solving the dual problem lies in evaluating~$\ell_{c}$ and in establishing useful structural properties of~$\ell_{c}$ as a function of~$\lambda$.  
If $-\ell$ and $c$ are (piecewise) convex in~$z$ and if~$\cZ$ is a convex set, for example, then robust optimization techniques can be used to express~$\ell_{c}(\theta, \lambda, \hat z)$ as the optimal value of a finite convex minimization problem. If, additionally, the reference distribution~$\hat \P$ has a finite support, then the entire dual problem on the right hand side of~\eqref{eq:inner:supremum} simplifies to a finite convex minimization problem amenable to off-the-shelf solvers~\citep{esfahani2018data, zhen2023unified}. Treating~$\theta\in\Theta$ as an additional decision variable finally yields a reformulation of the original DRO problem~\eqref{eq:dro} as a single (possibly nonconvex) finite minimization problem. This reformulation is convex if~$\ell$ is convex in~$\theta$ and if~$\Theta$ is a convex set.

If $-\ell$ or~$c$ fail to be (piecewise) convex in~$z$, $\cZ$ fails to be convex, or~$\hat{\P}$ fails to have a finite support, then exact convex reformulation results are scarce (an exception is described in~\citep[\textsection~6.2]{esfahani2018data}). In these cases, one may still attempt to attack the dual problem in~\eqref{eq:inner:supremum} directly with (stochastic) gradient descent-type algorithms \citep{blanchet2018optimal, sinha2018certifying}. The efficacy of these methods is predicated on the structural properties of the $c$-transform~$\ell_c$. If the transportation cost function is representable as $c(z, \hat z) = \Psi (z - \hat z)$ for some real-valued univariate function $\Psi(\cdot)$, and if $\cZ = \R^d$, for instance, then $-\ell_{c}(\theta, \lambda, \cdot)$ is readily recognized as the infimal convolution of~$-\ell(\theta,\cdot)$ and~$\lambda \Psi(\cdot)$. Hence, it can be interpreted as an epigraphical regularization of $-\ell(\theta,\cdot)$. Indeed, the epigraph of the infimal convolution $-\ell_{c}(\theta, \lambda, \cdot)$ (essentially) coincides with the Minkowski sum of the epigraphs of $-\ell(\theta,\cdot)$ and $\lambda \Psi(\cdot)$ \cite[Exercise~1.28\,(a)]{rockafellar2009variational}. Epigraphical regularizations are ubiquitous in approximation theory, and it is known that, as~$\lambda$ tends to infinity, $-\ell_{c}(\theta, \lambda, \cdot)$ provides an increasingly accurate approximation for~$-\ell(\theta,\cdot)$ that inherits desirable regularity properties from~$\Psi(\cdot)$ such as uniform continuity or smoothness \citep{attouch1993approximation,attouch1989epigraphical,bougeard1991towards,penot1998proximal}. Specifically, if~$\Psi(\cdot) = \|\cdot\|^2$,  then $-\ell_{c}(\theta, \lambda, \cdot)$ reduces to the Moreau-Yosida regularization (or Moreau envelope) \citep[Chapter~1.G]{rockafellar2009variational}, which is at the heart of most proximal algorithms \citep{parikh2014proximal}. In addition, if~$\Psi(\cdot)= \|\cdot\|$, then $-\ell_{c}(\theta, \lambda, \cdot)$ reduces to the Pasch-Hausdorff envelope (or Lipschitz regularization) proposed in~\citep[Example~9.11]{rockafellar2009variational} and~\citep{hiriart1980extension}, which constitutes the largest $\lambda$-Lipschitz continuous function majorized by~$-\ell(\theta,\cdot)$.

This paper aims to develop a deeper understanding of optimal transport-based DRO problems with general transportation cost functions, to showcase different regularizing effects of robustification and to elucidate the connections between DRO and game theory. Specifically, we show that robustification is intimately related to various forms of variation and Lipschitz regularization even if the transportation cost function is not (some power of) a metric. We also derive conditions for the existence and the computability of a Nash equilibrium between the decision-maker and nature in the DRO problem~\eqref{eq:dro}. In addition, in the context of image classification problems, we demonstrate that nature's Nash strategy can be viewed as a distribution that is supported on remarkably deceptive adversarial examples ({\em i.e.}, images). Finally, we identify practically relevant classes of optimal transport-based DRO problems that can be addressed with efficient gradient descent algorithms even if the loss function or the transportation cost function are nonconvex (but not both at the same time). 

The main contributions of this paper can be summarized as follows.
\begin{enumerate}[label=\roman*.]
    \item {\bf Existence of Nash equilibria.} We establish weak conditions under which the DRO problem~\eqref{eq:dro}, viewed as a zero-sum game between the decision-maker and nature, admits a Nash equilibrium. 
    \item {\bf Computation of Nash equilibria.} We prove that the dual DRO problem, which is obtained from~\eqref{eq:dro} by interchanging the order of minimization and maximization, can be reduced to a finite convex program under the same set of conditions that already ensured that the primal DRO problem~\eqref{eq:dro} admits a finite convex reduction. We also show that any solutions of the primal and dual DRO problems represent Nash strategies of the decision-maker and nature, respectively. 
    \item {\bf Adversarial examples.} Every Nash strategy of nature constitutes a best response to a Nash strategy of the decision-maker, but not vice versa. In the context of an image classification task, we show experimentally that nature's Nash strategy as well as nature's best response to the decision-maker's Nash strategy, as computed by Gurobi, encodes an adversarial dataset. While the adversarial images implied by nature’s best response can only deceive an algorithm, the adversarial images implied by nature’s Nash strategy can even deceive a human.
    \item {\bf Higher-order variation and Lipschitz regularization.} We show that, under natural regularity conditions, the worst-case expected loss in~\eqref{eq:dro} is bounded above by the sum of the expected loss under the reference distribution and several regularization terms that penalize certain $L^p$-norms or Lipschitz moduli of the higher-order derivatives of the loss function. This result generalizes and unifies several existing results, which have revealed intimate connections between robustification and gradient regularization~\cite{gao2024wasserstein}, Hessian regularization~\cite{bartl2020robust} and Lipschitz regularization~\citep{esfahani2018data}.
    \item {\bf Numerical solution of nonconvex DROs.} By leveraging techniques from nonconvex optimization such as Toland's duality principle, we show that distributionally robust linear prediction models, which emerge in portfolio selection, regression, classification, newsvendor, linear inverse or phase retrieval problems, can sometimes be solved efficiently by gradient descent-type algorithms even if the loss function or the transportation cost function are nonconvex.
\end{enumerate}

The rest of this paper is organized as follows. In Section~\ref{section:nash} we investigate sufficient conditions for the existence and computability of Nash equilibria between the decision-maker and nature. In Section~\ref{section:regularization} we shed new light on the intimate relations between regularization and robustification and discuss the numerical solution of certain nonconvex DRO problems. Numerical results are reported in Section~\ref{section:numerical}.

\paragraph{Notation} 
We denote the inner product of two vectors $x, y \in \R^d$ by $\inner{x}{y}$, and for any norm $\| \cdot \|$ on~$\R^d$, we use $\| \cdot \|_*$ to denote its dual norm defined through $\| y \|_* = \sup \{ \inner{x}{y}: \| x \| \leq 1 \}$. The domain of a function $f:\R^d \to [-\infty, \infty] $ is defined as $\dom(f) = \{x \in \R^d: f(x) < \infty \}$. The function $f$ is proper if $f(x) > -\infty$ and $\dom(f) \neq \emptyset$. The convex conjugate of $f$ is defined as $f^*(y) = \sup_{x \in \R^d} \inner{y}{x} - f(x)$. A function $f$ is (positively) homogeneous of degree $p \geq 1$ if $f(\lambda x) = \lambda^p f(x)$ for any $x \in \R^d$ and $\lambda > 0$. The indicator function $\delta_C:\R^d\to[0,\infty]$ of an arbitrary set~$\cX\subseteq\R^d$ is defined through $\delta_\cX(x)=0$ if $x\in \cX$ and $\delta_\cX(x)=\infty$ if $x\not\in \cX$. The perspective function of a proper, convex and lower semicontinuous function $f:\R^d \to (-\infty, \infty]$ is defined through $f(x, \lambda) = \lambda f(x / \lambda)$ if $\lambda > 0$ and $f(x, \lambda) = \delta^*_{\dom(f^*)}(x)$ if $\lambda = 0$, where $\delta_{\dom(f^*)}$ is the indicator function of the set $\dom(f^*)$. By slight abuse of notation, however, we use $\lambda f(x / \lambda)$ to denote the perspective function $f(x, \lambda)$ for all $\lambda \geq 0$. The Lipschitz modulus of a function $f:\R^d\to (-\infty,\infty]$ with respect to a norm~$\|\cdot\|$ on $\R^d$ is defined as 
\[
    \lip(f)=\sup_{x,x'\in\dom(f)} \left\{ |f(x)-f(x')|/\|x-x'\| : x\neq x' \right\}.
\]
For any $n\in\N$, we set $[n] = \{1,\ldots,n\}$, and we denote the relative interior of $\cZ\subseteq \R^d$ by~$\rint(\cZ)$.

\section{Nash Equilibria in DRO}
\label{section:nash}
By construction, the DRO problem~\eqref{eq:dro} constitutes a zero-sum game between an agent, who chooses the decision $\theta\in\Theta$, and some fictitious adversary or `nature', who chooses the distribution $\Q\in\B_\eps(\hat{\P})$. 
In this section we will show that, under some mild technical conditions, any optimal solution~$\theta^\star$ that solves the primal DRO problem~\eqref{eq:dro} and any optimal distribution~$\Q^\star$ that solves the dual DRO problem
\begin{align}
    \label{eq:dual:dro}
    \sup_{\Q \in \B_{\eps}(\hat{\P})} \inf_{\theta \in \Theta} ~ \E_{Z\sim\Q} \left[ \ell(\theta, Z) \right]
\end{align}
form a Nash equilibrium. Thus, $\theta^\star$ and~$\Q^\star$ satisfy the saddle point condition
\begin{align}
    \label{eq:Nash}
    \E_{Z\sim\Q}\left[\ell(\theta^\star, Z)\right] \leq \E_{Z\sim\Q^\star}\left[\ell(\theta^\star, Z)\right] \leq \E_{Z\sim\Q^\star}\left[\ell(\theta, Z)\right]\quad \forall \theta \in \Theta,~\Q \in \B_\eps(\hat{\P}).
\end{align}
The existence of a Nash equilibrium implies that~\eqref{eq:dro} and~\eqref{eq:dual:dro} are strong duals (that is, they share the same optimal values) and that they are both solvable. However, strong duality does not imply the existence of a Nash equilibrium (that is, the optimal values are not necessarily attained).

While it is well known that the primal DRO problem has a wide range of applications, the practical relevance of the dual DRO problem is less obvious. Indeed, the dual problem models a situation in which the decision maker observes the distribution that governs random quantities. As we will explain below, however, the dual DRO problem has deep connections to robust statistics, machine learning and related fields. In robust statistics, an optimal solution $\theta^\star$ of an estimation problem of the form~\eqref{eq:dro} is referred to as a minimax estimator or a robust estimator. When $\theta^\star$ and $\Q^\star$ satisfy the saddle point condition~\eqref{eq:Nash}, then the robust estimator $\theta^\star$ solves the stochastic program $\min_{\theta \in \Theta}~\E_{Z\sim\Q^\star} \left[ \ell(\theta, Z) \right]$, which minimizes the expected loss under the {\em crisp} distribution $\Q^\star$; see also \citep[Chapter~5]{lehmann2006theory}. For this reason, $\Q^\star$ is often referred to as a least favorable distribution. The existence of $\Q^\star$ makes the robust estimator~$\theta^\star$ particularly attractive. Indeed, if $\Q^\star$ exists, then $\theta^\star$ solves a classical stochastic program akin to the ideal learning problem one would want to solve if the true distribution was known. 

Nash equilibria are also relevant for adversarial machine learning, which aims to immunize neural networks against perturbations of the input data \citep{szegedy2014intriguing,ian15adversarial,madry2018towards}. An important aspect of adversarial training is the generation of adversarial examples, which are perturbations of training samples designed to mislead the neural network into making false predictions. Given a fixed neural network encoded by~$\hat \theta \in \Theta$, one can construct adversarial examples by solving the worst-case expectation problem
\begin{align}
    \label{eq:wc}
    \sup_{\Q \in \B_\eps(\hat \P)} \, \E_{Z \sim \Q}[\ell(\hat \theta, Z)].
\end{align}
Specifically, adversarial examples can be obtained by sampling from an extremal distribution $\Q^\star$ that solves problem~\eqref{eq:wc}. 
A desirable property of adversarial examples is their transferability across models~\citep{szegedy2014intriguing}. An adversarial example that fools one model often succeeds in fooling other models, too, thus enabling attacks on machine learning systems without direct access to the target model~\citep{papernot2016transferability,papernot2016distillation}. Hence, incorporating transferable adversarial examples into training is likely to enhance the robustness and security of the resulting model.
The dual DRO problem~\eqref{eq:dual:dro} offers a systematic approach to generate {\em transferable} adversarial examples. Specifically, the solutions of  \eqref{eq:dual:dro} constitute model-agnostic worst-case distributions that depend only on the decision space~$\Theta$. In contrast, the worst-case distributions that solve problem~\eqref{eq:wc} are tailored to a specific model~$\hat\theta\in\Theta$. The solutions of~\eqref{eq:dual:dro} are designed to challenge {\em all} decisions within $\Theta$, thus making them inherently transferable across all possible prediction models.

To date, dual DRO problems have only been investigated in specific applications. For example, it is known that the least favorable distributions in distributionally robust minimum mean square error estimation and Kalman filtering problems over type-$2$ Wasserstein ambiguity sets centered at Gaussian reference distributions are Gaussian, and that they can be computed efficiently via semidefinite programming~\citep{shafieezadeh2018wasserstein, nguyen2023bridging}. When the ambiguity set is defined in terms of the Kullback-Leibler divergence instead of the Wasserstein distance, the least favorable distribution remains Gaussian and can be found in quasi-closed form~\citep{levy2004robust,levy2012robust}. Similar results are available for generalized $\tau$-divergence ambiguity sets~\citep{zorzi2017robustness,zorzi2016robust}. In addition, Nash equilibria for distributionally robust pricing and auction design problems with rectangular ambiguity sets can sometimes also be derived in closed form~\citep{bergemann2008pricing,koccyiugit2020distributionally,koccyiugit2021robust}.

We emphasize that, in contrast to the worst-case expectation problem~\eqref{eq:wc}, whose objective function is {\em linear} in the distribution~$\Q$, the dual DRO problem~\eqref{eq:dual:dro} maximizes a {\em concave} objective function representable as the pointwise infimum of infinitely many linear functions of~$\Q$. This makes the dual DRO problem considerably more challenging to solve.

In the remainder of this section, we will first identify mild regularity conditions under which the primal and dual DRO problems~\eqref{eq:dro} and~\eqref{eq:dual:dro} are strong duals and/or admit a Nash equilibrium  (Section~\ref{section:nash:existence}). Next, we will show that if the reference distribution is discrete and the loss function is convex-piecewise concave, then the dual DRO problem~\eqref{eq:dual:dro} can be reformulated as a finite convex program (Section~\ref{section:nash:computation}). This reformulation will finally enable us to construct a least favorable distribution.

\subsection{Existence of Nash Equilibria}
\label{section:nash:existence}
The following assumption is instrumental to prove the existence of a Nash equilibrium. In the remainder of the paper we equip $\cP(\cZ)$ with the topology of weak convergence of probability distributions~\cite{billingsley2013convergence}.

\begin{assumption}[Transportation cost] ~
    \label{assumption:cost}
    \begin{enumerate}[label=(\roman*)]
        \item \label{assumption:cost:integrable} There exists a reference point $\hat z_0 \in \cZ$ such that $\E_{Z\sim\hat{\P}}[c(Z, \hat z_0)] \leq C$ for some constant~$C < \infty$.
        \item \label{assumption:cost:lb} There exists a metric $d(z,\hat z)$ on $\cZ$ with compact sublevel sets such that $c(z, \hat z) \geq d^p(z, \hat z)$ for some exponent~$p\in\N$.
    \end{enumerate}
\end{assumption}

Assumption~\ref{assumption:cost} will enable us to prove that the ambiguity set $\B_\eps(\hat{\P})$ is weakly compact. We emphasize that Assumption~\ref{assumption:cost}\,\ref{assumption:cost:lb} is unrestrictive and allows for transportation costs that fail to display common properties such as symmetry, convexity, homogeneity, or the triangle inequality.

\begin{example}[Local Mahalanobis transportation cost]~
    \label{example:Mahalanobis}
    The local Mahalanobis transportation cost is defined as $c(z,\hat z) = \inner{z - \hat z}{A(\hat z)(z -\hat z)}$, where $A(\hat z)$ represents a positive definite matrix for each $\hat z \in \cZ$ \citep{blanchet2018optimal}. This transportation cost trivially satisfies Assumption~\ref{assumption:continuity}\,\ref{assumption:cost:lsc} if $A(\hat z)$ is continuous in~$\hat z$. In addition, it satisfies Assumption~\ref{assumption:cost}\,\ref{assumption:cost:integrable} if the reference distribution~$\hat{\P}$ has finite second moments, and it satisfies Assumption~\ref{assumption:cost}\,\ref{assumption:cost:lb} for $d(z,\hat z)  = \alpha \| z - \hat z \|_2$ and $p = 2$ if $\alpha = \inf_{\hat z \in \cZ} \lambda_{\min}(A(\hat z))>0$. However, the local Mahalanobis transportation cost fails to be symmetric. 
\end{example}

\begin{example}[Discrete metric]
    If the transportation cost is identified with the (nonconvex) discrete metric $c(z,\hat z) = \mathds{1}_{\{z \neq \hat z\}}$, then the optimal transport distance reduces to the total variation distance \cite[\textsection~6]{villani2008optimal}. In this case, Assumptions~\ref{assumption:continuity}\,\ref{assumption:cost:lsc} and~\ref{assumption:cost}\,\ref{assumption:cost:integrable} are trivially satisfied, and Assumption~\ref{assumption:cost}\,\ref{assumption:cost:lb} holds for $d(z,\hat z)=c(z, \hat z)$ and for any $p\geq 1$ provided that the support set $\cZ$ is compact.
\end{example}

The next lemma identifies minimal conditions under which~$\B_{\eps}(\hat{\P})$ is weakly compact. Thus, it generalizes~\citep[Theorem~1]{yue2020linear} to ambiguity sets defined in terms of general optimal transport discrepancies.

\begin{lemma}[Weak compactness of optimal transport ambiguity sets]
    \label{lemma:compactness}
    If the transportation cost satisfies Assumptions~\ref{assumption:continuity}\,\ref{assumption:cost:lsc} and~\ref{assumption:cost}, then the ambiguity set $\B_\eps(\hat{\P})$ is weakly compact.
\end{lemma}

\begin{proof}
    Throughout the proof we use~$W_p$ to denote the $p$-th Wasserstein distance with respect to the metric~$d$ from Assumption~\ref{assumption:cost}\,\ref{assumption:cost:lb}. Specifically, for any probability distributions $\Q$ and $\hat{\P}$ on~$\cZ$ we~set
    \begin{align*}
        \textstyle W_p(\Q, \hat{\P}) = \left( \inf_{\pi \in \Pi(\Q, \hat{\P})} \E_{(Z,\hat Z)\sim\pi} [d(Z, \hat Z)^p] \right)^{1/p}.
    \end{align*}
    In addition, we define the Wasserstein space $\cW_p(\cZ) \subseteq \cP(\cZ)$ as the family of all probability distributions~$\Q$ on~$\cZ$ with $\E_{Z\sim\Q} [d(Z,\hat z_0)^p]<\infty$, where~$\hat z_0\in \cZ$ is the reference point from Assumption~\ref{assumption:cost}\,\ref{assumption:cost:integrable}. By using the triangle inequality for the metric~$d$, one can show that the Wasserstein space is in fact independent of the choice of~$\hat z_0$. 
    Next, we define the $p$-th Wasserstein ball of radius~$\eps\ge 0$ around $\hat{\P}$ as
    \begin{align*}
        \W_{\eps} (\hat{\P}) = \left\{ \Q \in \cP(\cZ): W_p ( \Q , \hat{\P} ) \leq \eps \right\}. 
    \end{align*}
    Assumptions~\ref{assumption:cost}\,\ref{assumption:cost:integrable} and~\ref{assumption:cost}\,\ref{assumption:cost:lb} imply that $\hat{\P} \in \cW_p(\cZ)$, which in turn implies via the triangle inequality for the $p$-th Wasserstein distance $W_p$ that $\W_\eps (\hat{\P})\subseteq \cW_p(\cZ)$; see also~\cite[Lemma~1]{yue2020linear}. Assumption~\ref{assumption:cost}\,\ref{assumption:cost:lb} further ensures that $d_c \geq W_p^p$, and thus the ambiguity set $\B_{\eps}(\hat{\P})$ constructed from the optimal transport discrepancy~$d_c$ is covered by the ambiguity set $\W_{\eps^{1/p}}(\hat{\P})$ constructed from the Wasserstein distance~$W_p$. 
    
    In order to prove that $\B_{\eps}(\hat{\P})$ is weakly compact, note first that the Wasserstein ball $\W_{\eps^{1/p}}(\hat{\P})$ is weakly sequentially compact thanks to \citep[Theorem~1]{yue2020linear}. As the notions of sequential compactness and compactness are equivalent in metric spaces \cite[Theorem~28.2]{munkres2000topology} and as the weak topology on~$\cP(\cZ)$ is metrized by the Prokhorov metric \cite[Theorem~6.8]{billingsley2013convergence}, $\W_{\eps^{1/p}}(\hat{\P})$ is also weakly compact. Next, recall that the transportation cost $c$ is nonnegative and lower semicontinuous by virtue of Assumption~\ref{assumption:continuity}\,\ref{assumption:cost:lsc}. This implies via \citep[Lemma~5.2]{clement2008wasserstein} that $d_c(\cdot, \hat{\P})$ is lower semicontinuous with respect to the weak topology on~$\cP(\cZ)$. Hence, $\W_{\eps^{1/p}}(\hat{\P})$ is weakly closed as a sublevel set of a weakly lower semicontinuous map. As weakly closed subsets of weakly compact sets are weakly compact, $\B_{\eps}(\hat{\P})$ is indeed weakly compact. 
\end{proof}

To prove the existence of Nash equilibria, we also need the following assumption on the loss function.

\begin{assumption}[Loss function] ~
    \label{assumption:loss}
    \begin{enumerate}[label=(\roman*)]
        \item \label{assumption:loss:lsc} For any $z \in \cZ$, the function $\ell(\theta,z)$ is lower semicontinuous in~$\theta$.
        \item \label{assumption:loss:Fatou} For any $\Q \in \B_{\varepsilon}(\hat{\P})$ and any $\bar{\theta} \in \Theta$, there exists a neighborhood $\cU$ of $\bar{\theta}$ such that the family of functions $\ell^-(\theta, z) = -\min \{0, \ell(\theta, z)\}$ parametrized by $\theta \in \cU$ is uniformly $\Q$-integrable in $z$.
        \item \label{assumption:loss:infcompact} There exists $\Q \in \B_\eps(\hat \P)$ such that $\E_{Z \sim \Q} [\ell(\theta, Z)]$ is inf-compact in~$\theta\in\Theta$. This means that the sublevel set $\{ \theta \in \Theta: \E_{Z \sim \Q} [\ell(\theta, Z)] \leq c \}$ is compact for every $c \in \R$.
        \item \label{assumption:loss:growth} For any $\theta \in \Theta$, there are $g > 0$, $\hat z_0 \in \cZ$ and $r \in (0,p)$ with $\ell(\theta,z)\leq g\left[1+d^r(z,\hat z_0)\right]$ for all $z \in \cZ$, where~$d$ and~$p$ are the metric and the exponent from Assumption~\ref{assumption:cost}\,\ref{assumption:cost:lb}, respectively.
    \end{enumerate}
\end{assumption}
As we will see below, Assumptions~\ref{assumption:loss}\,\ref{assumption:loss:lsc} and~\ref{assumption:loss}\,\ref{assumption:loss:Fatou} imply that $\E_{Z\sim\Q}\left[\ell(\theta, Z)\right]$ is lower semicontinuous in~$\theta$, while Assumptions~\ref{assumption:continuity}\,\ref{assumption:loss:usc} and~\ref{assumption:loss}\,\ref{assumption:loss:growth} imply that $\E_{Z\sim\Q}\left[\ell(\theta, Z)\right]$ is weakly upper semicontinuous in $\Q$. Note that Assumption~\ref{assumption:loss}\,\ref{assumption:loss:Fatou} trivially holds for non-negative loss functions. However, it also holds under the following two conditions, which are typically easy to check. First, $\B_{\varepsilon}(\hat{\P})$ is contained in a $p$-th Wasserstein ball around a nominal distribution~$\hat\P$ with finite $p$-th moments. This implies via \citep[Lemma~1]{yue2020linear} that {\em all} distributions $\Q\in \B_{\varepsilon}(\hat{\P})$ have finite $p$-th moments. Second, for any~$\bar\theta \in \Theta$, there exist $g > 0$, a reference point $\hat z_0 \in \cZ$ and a neighborhood~$\cU$ of~$\bar\theta$ such that $\ell(\theta,z)\geq -g\left[1+d^p(z,\hat z_0)\right]$ for all~$\theta\in\cU$ and~$z \in \cZ$. As for the inf-compactness condition in Assumption~\ref{assumption:loss}\,\ref{assumption:loss:infcompact}, suppose that $\E_{Z\sim\Q}\left[\ell(\theta, Z)\right]$ is lower semicontinuous in~$\theta$. In this case, Assumption~\ref{assumption:loss}\,\ref{assumption:loss:infcompact} is trivially satisfied if~$\Theta$ is compact. However, it also holds if $\E_{Z \sim \Q} [\ell(\theta, Z)]$ is coercive in~$\theta$, that is, if for every sequence $\{ \theta_n \}_{n \in \N}$ with $\lim_{n \to \infty} \| \theta_n \| = \infty$, we have $\lim_{n \to \infty} \E_{Z \sim \Q} [\ell(\theta, Z)] = \infty$. In practice, it is often convenient to verify the inf-compactness condition under the nominal distribution~$\Q=\hat \P$. 

\begin{lemma}[Continuity properties of the expected loss]
    \label{lemma:continuity}
    The following hold. 
    \begin{enumerate}[label=(\roman*)]
        \item \label{lemma:continuity:lsc} If Assumptions~\ref{assumption:loss}\,\ref{assumption:loss:lsc} and~\ref{assumption:loss}\,\ref{assumption:loss:Fatou} hold, then $\E_{Z\sim\Q}\left[\ell(\theta, Z)\right]$ is lower semicontinuous in~$\theta$ on $\Theta$.
        \item \label{lemma:continuity:usc} If Assumptions~\ref{assumption:continuity}\,\ref{assumption:loss:usc},~\ref{assumption:cost} and~\ref{assumption:loss}\,\ref{assumption:loss:growth} hold, then $\E_{Z\sim\Q}\left[\ell(\theta, Z)\right]$ is weakly upper semicontinuous in~$\Q$ on $\B_\eps(\hat{\P})$.
    \end{enumerate}
\end{lemma}

\begin{proof}
    Fix $\Q \in \B_\eps(\hat \P)$, $\theta\in\Theta$ and a sequence $\{\theta_n\}_{n\in\N}$ in~$\Theta$ converging to~$\theta$. Then, we have
    \begin{align*}
    \liminf_{\theta_n \rightarrow \theta} \E_{Z\sim\Q}\left[\ell(\theta_n, Z)\right] \geq \E_{Z\sim\Q}\left[\liminf_{\theta_n \rightarrow \theta} \ell(\theta_n, Z)\right] \geq \E_{Z\sim\Q}\left[\ell(\theta, Z)\right],
    \end{align*}
    where the two inequalities follow from Fatou's lemma, which applies thanks to Assumption~\ref{assumption:loss} \ref{assumption:loss:Fatou}, and from the lower semi-continuity of the loss function $\ell(\theta, z)$ in $\theta$ stipulated in Assumption~\ref{assumption:loss} \ref{assumption:loss:lsc}, respectively. This shows that $\E_{Z\sim\Q}\left[\ell(\theta, Z)\right]$ is lower semicontinuous in $\theta$, and thus assertion~\ref{lemma:continuity:lsc} follows.
    
    Next, define the Wasserstein space $\cW_p(\cZ)$ and the Wasserstein ball~$\W_{\eps^{1/p}}(\hat{\P})$ as in the proof of Lemma~\ref{lemma:compactness}. Recall also that $\hat{\P} \in \cW_p(\cZ)$ thanks to Assumptions~\ref{assumption:cost}\,\ref{assumption:cost:integrable} and~\ref{assumption:cost}\,\ref{assumption:cost:lb}. The proof of \citep[Theorem~3]{yue2020linear} then implies that $\E_{Z\sim\Q} [ \ell(\theta, Z) ]$ is weakly upper semicontinuous in~$\Q$ over~$\W_{\eps^{1/p}}(\hat{\P})$. This is true even if the sublevel sets of the metric~$d$ fail to be compact, meaning that Assumption~\ref{assumption:cost}\,\ref{assumption:cost:lb} could be relaxed. As~$\B_{\eps}(\hat{\P}) \subseteq \W_{\eps^{1/p}}(\hat{\P})$ by virtue of Assumption~\ref{assumption:cost}\,\ref{assumption:cost:lb}, we may thus conclude that $\E_{Z\sim\Q} [ \ell(\theta, Z) ]$ is weakly upper semicontinuous in~$\Q$ over $\B_{\eps}(\hat{\P})$. Hence, assertion~\ref{lemma:continuity:usc} follows. 
\end{proof}

The existence of a Nash equilibrium also relies on the following mild regularity conditions.

\begin{assumption}[Properness conditions]~ \label{assumption:regularity}
    \begin{enumerate}[label=(\roman*)]
        \item \label{assumption:regularity:infty}
        There exists $\theta \in \Theta$ with $\sup_{\Q \in \B_\eps(\hat \P)} \E_{Z\sim\Q}[\ell(\theta, Z)] < +\infty$.
        \item \label{assumption:regularity:minus:infty} There exists $\Q \in \B_\eps(\hat \P)$ with $\inf_{\theta \in \Theta } \E_{Z\sim\Q}[\ell(\theta, Z)] > -\infty$.
    \end{enumerate}
\end{assumption}

Assumption~\ref{assumption:regularity}\,\ref{assumption:regularity:infty} ensures that the optimal value of the {\em primal} DRO problem~\eqref{eq:dro} is strictly smaller than $+\infty$, whereas Assumption~\ref{assumption:regularity}\,\ref{assumption:regularity:minus:infty} ensures that the optimal value of the {\em dual} DRO problem~\eqref{eq:dual:dro} is strictly larger than $-\infty$. Note that any DRO problem that violates these conditions is pathological and of limited practical interest. We are now ready to state the main result of this section.

\begin{theorem}[Minimax theorem]
    \label{theorem:minimax}
    Suppose that the feasible set $\Theta$ is convex and that the loss function~$\ell(\theta,z)$ is convex in~$\theta$ for any fixed~$z \in \cZ$. 
    \begin{enumerate}[label=(\roman*)]
        \item \label{theorem:minimax:minsup} If Assumptions~\ref{assumption:loss}\,\ref{assumption:loss:lsc}-\ref{assumption:loss:infcompact} and~\ref{assumption:regularity}\,\ref{assumption:regularity:infty} hold, then we have
        \begin{align*}
            \min_{\theta \in \Theta} \sup_{\Q \in \B_{\eps}(\hat{\P})} ~ \E_{Z\sim\Q} \left[ \ell(\theta, Z) \right]  = \sup_{\Q \in \B_{\eps}(\hat{\P})} \inf_{\theta \in \Theta} ~ \E_{Z\sim\Q} \left[ \ell(\theta, Z) \right].
        \end{align*}
        \item \label{theorem:minimax:maxinf}If Assumptions~\ref{assumption:continuity}, \ref{assumption:cost}, \ref{assumption:loss}\,\ref{assumption:loss:growth} and \ref{assumption:regularity}\,\ref{assumption:regularity:minus:infty} hold, then we have
        \begin{align*}
            \inf_{\theta \in \Theta} \sup_{\Q \in \B_{\eps}(\hat{\P})} ~ \E_{Z\sim\Q} \left[ \ell(\theta, Z) \right]  = \max_{\Q \in \B_{\eps}(\hat{\P})} \inf_{\theta \in \Theta} ~ \E_{Z\sim\Q} \left[ \ell(\theta, Z) \right].
        \end{align*}
        \item \label{theorem:minimax:minmax}If Assumptions~\ref{assumption:continuity}, \ref{assumption:cost}, \ref{assumption:loss}, \ref{assumption:regularity} hold, then we have
        \begin{align*}
            \min_{\theta \in \Theta} \sup_{\Q \in \B_{\eps}(\hat{\P})} ~ \E_{Z\sim\Q} \left[ \ell(\theta, Z) \right]  = \max_{\Q \in \B_{\eps}(\hat{\P})} \inf_{\theta \in \Theta} ~ \E_{Z\sim\Q} \left[ \ell(\theta, Z) \right].
        \end{align*}
    \end{enumerate}
\end{theorem}
\begin{proof}
    Note that the objective function $\E_{Z\sim\Q}\left[\ell(\theta, Z)\right]$ inherits convexity in~$\theta$ from the loss function~$\ell(\theta, z)$. In addition, it is concave (in fact, linear) in~$\Q$. Note also that the optimal transport ambiguity set $\B_\eps(\hat \P)$ inherits convexity from the optimal transport discrepancy. Indeed, $d_c(\P, \hat \P)$ is defined as the optimal value of a convex minimization problem whose constraints are jointly linear in the parameters~$\P$ and~$\hat\P$ and the decision variable~$\pi$, that is, the transportation plan connecting~$\P$ and~$\hat\P$. Thus, $d_c(\P, \hat \P)$ is  jointly convex in $\P$ and~$\hat \P$ by virtue of \citep[Theorem~1]{rockafellar1974conjugate}. 

    As for assertion~\ref{theorem:minimax:minsup}, note that $\E_{Z\sim\Q}\left[\ell(\theta, Z)\right]$ is lower semicontinuous in~$\theta\in\Theta$ for all $\Q\in \B_\eps(\hat \P)$ thanks to Lemma~\ref{lemma:continuity}\,\ref{lemma:continuity:lsc}, which applies because of Assumptions~\ref{assumption:loss}\,\ref{assumption:loss:lsc} and~\ref{assumption:loss}\,\ref{assumption:loss:Fatou}. Also, $\E_{Z\sim\Q}\left[\ell(\theta, Z)\right]$ is inf-compact in~$\theta\in\Theta$ for some $\Q\in \B_\eps(\hat \P)$ by Assumption~\ref{assumption:loss}\,\ref{assumption:loss:infcompact}, and the infimum of the primal DRO problem~\eqref{eq:dro} is strictly smaller than~$+\infty$ by Assumption~\ref{assumption:regularity}\,\ref{assumption:regularity:infty}. The reverse lopsided minimax theorem \citep[Corollary~5.16]{kuhn2024distributionally} thus ensures that the optimal values of the primal and dual DRO problems match. In addition, the objective function of the primal DRO problem is proper and inf-compact in~$\theta$. Hence, the primal DRO problem is solvable thanks to Weierstrass’ extreme value theorem.

    As for assertion~\ref{theorem:minimax:maxinf}, note that $\E_{Z\sim\Q}\left[\ell(\theta, Z)\right]$ is weakly upper semicontinuous in~$\Q\in \B_\eps(\hat \P)$ for all $\theta \in\Theta$ thanks to Lemma~\ref{lemma:continuity}\,\ref{lemma:continuity:usc}, which applies because of Assumptions~\ref{assumption:continuity}\,\ref{assumption:loss:usc}, \ref{assumption:cost} and~\ref{assumption:loss}\,\ref{assumption:loss:growth}. In addition, $\B_\eps(\hat \P)$ is weakly compact thanks to Lemma~\ref{lemma:compactness}, which applies because of Assumptions~\ref{assumption:continuity}\,\ref{assumption:cost:lsc} and~\ref{assumption:cost}. Also, the supremum of the dual DRO problem~\eqref{eq:dual:dro} is strictly larger than~$-\infty$ by Assumption~\ref{assumption:regularity}\,\ref{assumption:regularity:minus:infty}. The lopsided minimax theorem \citep[Theorem~5.15]{kuhn2024distributionally} thus ensures that the optimal values of the primal and dual DRO problems match. In addition, the dual DRO problem is solvable by Weierstrass' theorem because its objective function is weakly upper semicontinuous and its feasible set is weakly compact.

    Finally, assertion~\ref{theorem:minimax:minmax} follows immediately from assertions~\ref{theorem:minimax:minsup} and \ref{theorem:minimax:maxinf}.
\end{proof}

Theorem~\ref{theorem:minimax} generalizes \citep[Theorem~2]{blanchet2019confidence}, which holds only for quadratic transportation cost functions, non-negative and smooth loss functions as well as discrete (empirical) reference distributions. In contrast, Theorem~\ref{theorem:minimax} allows~$c$ to be nonconvex, $\ell$ to adopt both positive and negative values and~$\hat \P$ to be non-discrete. We also present simpler sufficient conditions for the solvability of the primal and dual DRO problems. For example, the inf-compactness condition in Assumption~\ref{assumption:loss}\,\ref{assumption:loss:infcompact} is related to \citep[Assumption~A2\,(b)]{blanchet2019confidence}, which requires the stochastic optimization problem $\min_{\theta \in \Theta} \E_{Z \sim \P} [\ell(\theta, Z)]$ associated with the unknown data generating distribution~$\P$ to have a unique minimizer. While this assumption implies inf-compactness if the loss function is convex, it is arguably more difficult to check in practice---most notably because the true distribution $\P$ is typically unknown. 

\subsection{Computation of Nash Equilibria}
\label{section:nash:computation}

At first sight, the dual DRO problem~\eqref{eq:dual:dro} appears intractable as it constitutes a challenging maximin problem that maximizes the optimal value of a parametric minimization problem over an infinite-dimensional space of probability distributions. It is well known that the primal DRO problem~\eqref{eq:dro} can be reformulated as a finite convex program if the loss function~$\ell(\theta,z)$, the support set~$\cZ$, the feasible set~$\Theta$ and the transportation cost~$c(z,\hat z)$ display certain convexity properties and if the reference distribution~$\hat{\P}$ is discrete; see, {\em e.g.}, \citep[\textsection~4.1]{esfahani2018data} or~\citep[\textsection~6]{zhen2023unified}. In the remainder of this section we will demonstrate that essentially the same regularity conditions also enable us to reformulate the dual DRO problem~\eqref{eq:dual:dro} as a finite convex program. The solutions of the emerging convex programs can be used to construct a robust decision~$\theta^\star$ and a least favorable distribution~$\Q^\star$ that form a Nash equilibrium.

We will now show that dual DRO problems of the form~\eqref{eq:dro} and~\eqref{eq:dual:dro} can often be addressed with methods from convex optimization. To this end, we restrict attention to {\em discrete} reference distributions.

\begin{assumption}[Discrete reference distribution]
    \label{assumption:nash:reference}
    We have $\hat{\P} = \sum_{j \in [J]} p_j \delta_{\hat z_j}$ for some $J\in\N$, where each probability $p_j$ is strictly positive and $\delta_{\hat z_j}$ denotes the Dirac measure at the atom~$\hat z_j \in \cZ$.
\end{assumption}    

We also need the following assumption, which constrains the shape of the transportation cost function, the loss function, the support set and the feasible set. Thus, it limits modeling flexibility.

\begin{assumption}[Convexity conditions]~
    \label{assumption:nash:convexity}
    \begin{enumerate}[label=(\roman*)]
        \item \label{assumption:nash:cost} The transportation cost function~$c(z,\hat z)$ is lower semicontinuous in $(z,\hat z)$ and convex in~$z$.
        \item \label{assumption:nash:loss} The loss function is representable as a pointwise maximum of finitely many saddle functions, that is, we have $\ell(\theta, z) = \max_{i \in [I]} \ell_{i}(\theta, z)$ for some $I \in \N$, where $\ell_{i}(\theta, z)$ is proper, convex and lower semicontinuous in $\theta$, while $-\ell_{i}(\theta, z)$ is proper, convex and lower semicontinuous in~$z$. 
        \item
        \label{assumption:nash:set:Z} The support set is representable as $\cZ = \{ z \in \R^d : f_k(z) \leq 0~\forall k \in [K] \}$ for some $K \in \N$, where each function $f_k(z)$ is proper, convex and lower semicontinuous. 
        \item
        \label{assumption:nash:set:Theta}
        The feasible set is representable as $\Theta = \{ \theta \in \R^m: g_l(\theta) \leq 0~\forall l \in [L] \}$ for some $L \in \N$, where each function $g_l(\theta)$ is proper, convex and lower semicontinuous.
    \end{enumerate}
\end{assumption}

Recall that the transportation cost function~$c(z,\hat z)$ is non-negative and satisfies the identity of indiscernibles. Together with Assumption~\ref{assumption:nash:convexity}\,\ref{assumption:nash:cost}, this immediately implies that~$c(z,\hat z)$ is proper, convex and lower semicontinuous in~$z$ for every fixed~$\hat z$ and that it is proper and lower semicontinuous in~$\hat z$ for every fixed~$z$. Assumption~\ref{assumption:nash:convexity}\,\ref{assumption:nash:loss} requires both~$\ell_i(\theta,z)$ and~$-\ell_i(\theta,z)$ to be proper, which implies that the saddle function~$\ell_i(\theta,z)$ can adopt only finite values. However, the convex functions~$f_k(z)$ and~$g_l(\theta)$ introduced in Assumptions~\ref{assumption:nash:convexity}\,\ref{assumption:nash:set:Z} and~\ref{assumption:nash:convexity}\,\ref{assumption:nash:set:Theta}, respectively, may adopt the value~$\infty$. Assumption~\ref{assumption:nash:convexity} has been introduced in \cite{esfahani2018data, zhen2023unified} to derive finite convex reformulations of {\em primal} DRO problems with optimal transport amiguigty sets, and it holds in many relevant applications of DRO.

In the remainder, we adopt the following definition of a Slater point for a minimization problem.

\begin{definition}[Slater point]
    A Slater point of the set $\cX = \{ x \in \R^d : h_m(x) \leq 0 ~ \forall m \in [M] \}$ represented via $M$ generic inequality constraints is any vector $x^s \in \cX$ with $x^s \in \rint(\dom(h_m))$ for all $m \in [M]$ and $h_m (x^s) < 0$ for all $m \in[M]$ such that $h_m$ is nonlinear. In addition, $x^s$ is a Slater point of the minimization problem $\inf \{ h_0 (x): x \in \cX \}$ if it is a Slater point of $\cX$ and if $x^s \in \rint(\dom(h_0))$.
\end{definition}

Our results also rely on the following technical Slater conditions.
Even though they are indispensable for the proofs of our convex reformulation results below, they do not critically limit modeling flexibility.

\begin{assumption}[Slater conditions]~
    \label{assumption:slater}
    \begin{enumerate}[label=(\roman*)]
        \item \label{assumption:slater:Z} 
        For every $j \in [J]$, $\hat z_j\in\rint(\dom(c(\cdot, \hat z_j)))$ is a Slater point for the support set $\cZ$.
        \item \label{assumption:slater:Theta}
        The feasible set $\Theta$ admits a Slater point.
    \end{enumerate}
\end{assumption}

Assumption~\ref{assumption:slater} can always be enforced by slightly perturbing the problem data or by eliminating redundant constraints and decision variables; see, {\em e.g.}, \cite[Example~C.3]{zhen2023unified}.

Before addressing the dual DRO problem~\eqref{eq:dual:dro}, we show that the primal DRO problem~\eqref{eq:dro} admits a finite convex reformulation when the above regularity conditions are satisfied. This reformulation was first derived under the simplifying assumption that $c(z, \hat z) = \| z - \hat z \|$ is defined in terms of a norm on~$\cZ$~\cite[Theorem~4.2]{esfahani2018data} and later generalized to arbitrary convex transportation cost functions~\cite[\textsection~6]{zhen2023unified}. Proposition~\ref{proposition:primal} below re-derives this reformulation under the Assumptions~\ref{assumption:nash:reference}, \ref{assumption:nash:convexity} and~\ref{assumption:slater}\,\ref{assumption:slater:Z}, which can be checked {\em ex ante}. In contrast, the regularity conditions used in~\cite{zhen2023unified} can only be checked {\em ex post} by solving a convex program. To keep this paper self-contained, we provide a simple proof of Proposition~\ref{proposition:primal}, which will also allow us to streamline the proof of Proposition~\ref{proposition:dual} below.

\begin{proposition}[Primal reformulation]
    \label{proposition:primal}
    If Assumptions~\ref{assumption:nash:reference}, \ref{assumption:nash:convexity} and~\ref{assumption:slater}\,\ref{assumption:slater:Z} hold and if~$\eps > 0$, then the primal DRO problem~\eqref{eq:dro} has the same infimum as the finite convex program
    \begin{align}
        \label{eq:primal}
        \begin{array}{cll}
            \inf & \DS \lambda \eps + \sum_{j \in [J]} p_j s_j \\[1.5ex]
            \st & \theta \in \Theta, \, \lambda, \tau_{ijk} \in \R_+, \, s_j \in \R, \, \zeta^{\ell}_{ij}, \zeta^{c}_{ij}, \zeta^{f}_{ijk} \in \R^d & \forall i \in [I],\, j \in [J], \,k \in [K] \\[2ex]
            & \DS (- \ell_i)^{*2}(\theta, \zeta^{\ell}_{ij}) + \lambda c^{*1}( \zeta^{c}_{ij} / \lambda, \hat z_j) + \sum_{k \in [K]} \tau_{ijk} f_k^*( \zeta^{f}_{ijk} / \tau_{ijk}) \leq s_j & \forall i \in [I],\, j \in [J] \\[1.5ex]
            & \DS \zeta^{\ell}_{ij} + \zeta^{c}_{ij} + \sum_{k \in [K]} \zeta^{f}_{ijk} = 0 & \forall i \in [I],\, j \in [J],
        \end{array}
    \end{align}
    where $(-\ell_i)^{*2}(\theta,\zeta)$ denotes the conjugate of $-\ell_i(\theta, z)$ with respect to its second argument~$z$ for fixed~$\theta$, and where $c^{*1}(\zeta,\hat z)$ denotes the conjugate of $c(z,\hat z)$ with respect to its first argument~$z$ for fixed~$\hat z$.
\end{proposition}

\begin{proof}
    Note first that Assumptions~\ref{assumption:nash:reference} and~\ref{assumption:nash:convexity}\,\ref{assumption:nash:cost}--\ref{assumption:nash:loss} strengthen Assumption~\ref{assumption:continuity}. In particular, Assumption~\ref{assumption:nash:convexity}\,\ref{assumption:nash:cost} trivially implies Assumption~\ref{assumption:continuity}\,\ref{assumption:cost:lsc}. In addition, Assumption~\ref{assumption:nash:reference} requires~$\hat{\P}$ to be discrete, and Assumption~\ref{assumption:nash:convexity}\,\ref{assumption:nash:loss} implies that~$\ell(\theta, z)$ is finite-valued. As any finite-valued function is integrable with respect to any discrete distribution, Assumption~\ref{assumption:continuity}\,\ref{assumption:loss:usc} trivially holds. Thus, Proposition~\ref{proposition:strong:duality} applies and allows us to reformulate the inner supremum in~\eqref{eq:dro} as
    \begin{align*}
        \sup_{\Q \in \B_\eps(\hat{\P})} ~ \E_{Z\sim\Q} \left[ \ell(\theta, Z) \right]
        = \left\{
        \begin{array}{cll}
            \inf & \DS \lambda \eps + \sum_{j \in [J]} p_j s_j \\[1ex]
            \st & \lambda \in \R_+, ~ s_j \in \R & \forall j \in [J] \\[1ex]
            & \DS \sup_{z \in \cZ}~ \ell(\theta, z) - \lambda c(z, \hat z_j) \leq s_j & \forall j \in [J],
        \end{array}
        \right.
    \end{align*}
    where the epigraphical decision variable~$s_j$ coincides with~$\ell_{c}(\theta, \lambda, \hat z_j)=\sup_{z \in \cZ}~ \ell(\theta, z) - \lambda c(z, \hat z_j)$ at optimality. Note that since $\ell(\theta,z)$ is finite-valued, there is no need to make the constraint $c(z,\hat z)<\infty$ explicit. By Assumption~\ref{assumption:nash:convexity}\,\ref{assumption:nash:loss}, the loss function~$\ell$ can be expressed in terms of the saddle functions $\ell_i$, $i\in[I]$, and thus the resulting dual problem is equivalent to the following robust convex program.
    \begin{align}
        \label{eq:strong:duality:robust}
        \begin{array}{cll}
            \inf & \DS \lambda \eps + \sum_{j \in [J]} p_j s_j \\[1ex]
            \st & \lambda \in \R_+, ~ s_j \in \R & \forall j \in [J] \\[1ex]
            & \DS \sup_{z \in \cZ}~ \ell_i(\theta, z) - \lambda c(z, \hat z_j) \leq s_j & \forall i \in [I], \, j \in [J]
        \end{array}
    \end{align}
    By the convexity of the support set $\cZ$ imposed by Assumption~\ref{assumption:nash:convexity}\,\ref{assumption:nash:set:Z} and thanks to an explicit convex duality result~\cite[Theorem~2]{zhen2023unified}, the embedded maximization problems in~\eqref{eq:strong:duality:robust} can be recast as 
    \begin{align*}
        \sup_{z \in \cZ} ~ \ell_i(\theta, z) - \lambda c(z,\hat z_j) 
        = \left\{ 
        \begin{array}{cl}
            \min_{} & \DS (- \ell_i)^{*2}(\theta, \zeta^{\ell}_{ij}) + \lambda c^{*1}( \zeta^{c}_{ij} / \lambda, \hat z_j) + \sum_{k \in [K]} \tau_{ijk} f_k^*( \zeta^{f}_{ijk} / \tau_{ijk}) \\[2ex]
            \st & \tau_{ijk} \in \R_+, ~ \zeta^{\ell}_{ij}, \zeta^{c}_{ij}, \zeta^{f}_{ijk} \in \R^d \quad \forall k \in [K] \\[1.5ex]
            & \DS \zeta^{\ell}_{ij} + \zeta^{c}_{ij} + \sum_{k \in [K]} \zeta^{f}_{ijk} = 0
        \end{array}
        \right.
    \end{align*}
    for all~$i \in [I]$ and~$j \in [J]$. Indeed,  Assumption~\ref{assumption:slater}\,\ref{assumption:slater:Z} and the finiteness of~$\ell_i$ implied by Assumption~\ref{assumption:nash:convexity}\,\ref{assumption:nash:loss} ensure that the primal maximization problem admits a Slater point. This implies both strong duality as well as solvability of the dual minimization problem. Substituting all resulting dual minimization problems into~\eqref{eq:strong:duality:robust} and eliminating the corresponding minimization operators yields~\eqref{eq:primal}.
    
    Note that the conjugate $(-\ell_i)^{*2}$ of the negative saddle function~$-\ell_i$ with respect to its second argument is convex by construction. Similarly, when~$\hat z_j$ is kept fixed, the perspectives of the conjugates~$c^{*1}(\cdot, \hat z_j)$ and~$f^\star_k$ are readily seen to be convex. Thus, problem~\eqref{eq:primal} is indeed a convex program.
\end{proof}

Next, we prove that the dual DRO problem~\eqref{eq:dual:dro} also admits a finite convex reformulation if all assumptions of Proposition~\ref{assumption:nash:reference} are satisfied and if at least one out of three regularity conditions holds. 

\begin{assumption}
    \label{assumption:nash:regularity}
    At least one of the following three conditions is satisfied: (i)
    $\E_{Z \sim \Q} [\ell(\theta, Z)]$ is inf-compact in~$\theta\in\Theta$ for some $\Q\in \B_\eps(\hat \P)$, (ii) $\cZ$ is compact or (iii)~$c(z,\hat z)$ grows superlinearly with $z$.
\end{assumption}

Assumption~\ref{assumption:nash:regularity} is unrestrictive in practice, and we will later see that it can be further relaxed.

\begin{proposition}[Dual reformulation]
    \label{proposition:dual}
    If Assumptions~\ref{assumption:loss}\,\ref{assumption:loss:Fatou}, \ref{assumption:regularity}\,\ref{assumption:regularity:infty}, \ref{assumption:nash:reference}, \ref{assumption:nash:convexity}, \ref{assumption:slater} and~\ref{assumption:nash:regularity} hold and if~$\eps > 0$, then the dual DRO problem~\eqref{eq:dual:dro} has the same supremum as the finite convex program
    \begin{align}
    \label{eq:dual}
    \begin{array}{cl@{\,}l}
        \max & \DS - \sum_{i \in [I]} \sum_{j \in [J]} q_{ij} \ell_i^{*1}(\alpha_{ij} /q_{ij}, \hat z_j + \xi_{ij} / q_{ij}) - \sum_{l \in [L]} \nu_l& g^*_l(\beta_l / \nu_l) \\[1ex]
        \st & \DS q_{ij}, \nu_l \in \R_+, ~ \xi_{ij} \in \R^d, ~ \alpha_{ij}, \beta_l \in \R^m & \forall i \in [I], \, j \in [J], \, l \in [L] \\[1ex]
        & q_{ij} f_k(\hat z_j + {\xi_{ij}}/{q_{ij}}) \leq 0 & \forall i \in [I], \, j \in [J], \, k \in [K] \\[1ex]
        & \DS \sum_{i \in [I]} q_{ij} = p_j & \forall j \in [J] \\[1ex]
        & \DS \sum_{i \in [I]} \sum_{j \in [J]} \alpha_{ij} + \sum_{l \in [L]} \beta_l = 0 \\[1ex]
        &\DS \sum_{i \in [I]} \sum_{j \in [J]} q_{ij} \, c \big( \hat z_j+ {\xi_{ij}}/{q_{ij}}, \hat z_j) \leq \eps,
    \end{array}
    \end{align}
    where $\ell_i^{*1}(\alpha,z)$ denotes the conjugate of $\ell_i(\theta, z)$ with respect to its first argument~$\theta$ for fixed~$z$.
\end{proposition}
\begin{proof}
    We first establish that the infimum of the primal DRO problem~\eqref{eq:dro} equals the maximum of the finite convex program~\eqref{eq:dual} and that~\eqref{eq:dual} is indeed solvable (Step~1). Using the insights from Step~1, we then show that any optimal solution of the convex program~\eqref{eq:dual} can be used to construct a near-optimal solution of the dual DRO problem~\eqref{eq:dual:dro} and that the two problems have the same supremum (Step~2). 
    
    \textbf{Step~1.} By Proposition~\ref{proposition:primal} for $\Theta=\{\theta\}$, which applies thanks to Assumptions~\ref{assumption:nash:reference}, \ref{assumption:nash:convexity} and~\ref{assumption:slater}\,\ref{assumption:slater:Z} and because~$\eps > 0$, the supremum of the inner maximization problem in~\eqref{eq:dro} equals
    \begin{align}
        \label{eq:inner:inf}
    \begin{array}{cll}
        \inf & \DS \lambda \eps + \sum_{j \in [J]} p_j s_j \\[1.5ex]
        \st & \lambda, \tau_{ijk} \in \R_+, ~ s_j \in \R, ~ \zeta^{\ell}_{ij}, \zeta^{c}_{ij}, \zeta^{f}_{ijk} \in \R^d & \forall i \in [I], \, j \in [J],\,  k \in [K] \\[2ex]
        & \DS (- \ell_i)^{*2}(\theta, \zeta^{\ell}_{ij}) + \lambda c^{*1}( \zeta^{c}_{ij} / \lambda, \hat z_j) + \sum_{k \in [K]} \tau_{ijk} f_k^*( \zeta^{f}_{ijk} / \tau_{ijk}) \leq s_j & \forall i \in [I], \, j \in [J] \\
        & \DS \zeta^{\ell}_{ij} + \zeta^{c}_{ij} + \sum_{k \in [K]} \zeta^{f}_{ijk} = 0 & \forall i \in [I], \,j \in [J].
    \end{array}
    \end{align}
    In~\citep[\textsection~6]{zhen2023unified} it is shown that the maximization problem dual to the above minimization problem can be simplified by eliminating auxiliary epigraphical decision variables and applying a linear variable substitution; see problem~(19) in~\citep{zhen2023unified}. In our notation, this simplified dual convex program is given~by
    \begin{align}
        \label{eq:inner:sup}
        \begin{array}{cll}
            \max & \DS \sum_{i \in [I]} \sum_{j \in [J]} q_{ij} \ell_i(\theta, \hat z_j + \xi_{ij} / q_{ij}) \\[0.5ex]
            \st & \DS q_{ij} \in \R_+, ~ \xi_{ij} \in \R^d & \forall i \in [I], \, j \in [J] \\[0.5ex]
            & q_{ij} f_k(\hat z_j + {\xi_{ij}}/{q_{ij}}) \leq 0 & \forall i \in [I], \, j \in [J], \, k \in [K] \\[0.5ex]
            & \DS \sum_{i \in [I]} q_{ij} = p_j & \forall j \in [J] \\[0.5ex]
            & \DS \sum_{i \in [I]} \sum_{j \in [J]} q_{ij} \, c \big( \hat z_j+ {\xi_{ij}}/{q_{ij}}, \hat z_j) \leq \eps.
        \end{array}
    \end{align}
    Strong duality holds thanks to Assumption~\ref{assumption:slater}\,\ref{assumption:slater:Z}, which implies that the dual maximization problem admits a Slater point $(\{q_{ij}^s, \xi_{ij}^s\}_{i,j})$ with~$q_{ij}^s = p_j / J$ and $\xi_{ij}^s = 0$ for all $i \in [I]$ and $j \in [J]$.
    In addition, \citep[Proposition~20]{zhen2023unified} implies that the dual problem has a compact feasible set and is thus solvable. Hence, its maximum is indeed attained. 
    Substituting this dual problem back into~\eqref{eq:dro} and interchanging the infimum and the maximum, which is allowed by Sion's minimax theorem~\citep{sion1958general}, we then obtain
    \begin{align}
        \label{eq:dual:robust}
        \begin{array}{cll}
            \max & \DS \inf_{\theta \in \Theta} ~ \sum_{i \in [I]} \sum_{j \in [J]} q_{ij} \ell_i(\theta, \hat z_j + \xi_{ij} / q_{ij}) \\[0.5ex]
            \st & \DS q_{ij} \in \R_+, ~ \xi_{ij} \in \R^d & \forall i \in [I], \, j \in [J] \\[0.5ex]
            & q_{ij} f_k(\hat z_j + {\xi_{ij}}/{q_{ij}}) \leq 0 & \forall i \in [I], \, j \in [J], \, k \in [K] \\[0.5ex]
            & \DS \sum_{i \in [I]} q_{ij} = p_j & \forall j \in [J] \\[0.5ex]
            & \DS \sum_{i \in [I]} \sum_{j \in [J]} q_{ij} \, c \big( \hat z_j+ {\xi_{ij}}/{q_{ij}}, \hat z_j) \leq \eps.
        \end{array}
    \end{align}
    The inner minimization problem in~\eqref{eq:dual:robust} has a Slater point because of
    Assumptions~\ref{assumption:nash:convexity}\,\ref{assumption:nash:loss} and~\ref{assumption:slater}\,\ref{assumption:slater:Theta}. By~\citep[Theorem~2]{zhen2023unified}, this minimization problem therefore admits a strong dual of the form
    \begin{align}
        \label{eq:outer:max}
        \begin{array}{cl}
            \max & \DS - \sum_{i \in [I]} \sum_{j \in [J]} q_{ij} \ell_i^{*1}(\alpha_{ij}/q_{ij}, \hat z_j + \xi_{ij} / q_{ij}) - \sum_{l \in [L]} \nu_l g^*_l(\beta_l / \nu_l) \\[1ex]
            \st & \alpha_{ij} \in \R^m, ~ \beta_l \in \R^m, ~ \nu_l \in \R_+ \quad \forall i \in [I], \, j \in [J], \, l \in [L] \\[1ex]
            & \DS \sum_{i \in [I]} \sum_{j \in [J]} \alpha_{ij} + \sum_{l \in [L]} \beta_l = 0,
        \end{array}
        \end{align}
    which is solvable. Substituting this dual maximization problem back into~\eqref{eq:dual:robust} finally yields~\eqref{eq:dual}. Note that the maximum in~\eqref{eq:dual} is indeed attained because problem~\eqref{eq:dual:robust} as well as the dual of the parametric minimization problem in its objective function are both solvable.

    \textbf{Step~2.} Fix now any maximizer $(\{q_{ij}^\star, \xi_{ij}^\star, \alpha_{ij}^\star\}_{i,j}, \{\nu_l^\star, \beta_l^\star\}_{l})$ of problem~\eqref{eq:dual}, which exists thanks to the results of Step~1. In Step~2 we will show that this maximizer can be used to construct a maximizer of the dual DRO problem~\eqref{eq:dual:dro} (if~\eqref{eq:dual:dro} is solvable) or a sequence of distributions that are feasible and asymptotically optimal in~\eqref{eq:dual:dro} (if~\eqref{eq:dual:dro} is {\em not} solvable). To this end, define three disjoint index sets
    \begin{align*}
        \cI_j^+ = \left\{ i \in [I] : q_{ij}^\star > 0 \right\}, \quad
        \cI_j^0 = \left\{ i \in [I] : q_{ij}^\star = 0,\, \xi_{ij}^\star = 0 \right\} \quad \text{and} \quad
        \cI_j^\infty = \left\{ i \in [I] : q_{ij}^\star = 0,\, \xi_{ij}^\star \neq 0 \right\},
    \end{align*}
    which form a partition of~$[I]$ for every fixed~$j$. Assume first that~$\cI_j^\infty =\emptyset$ for every~$j\in [J]$. In this case, we can use similar arguments as in
    \citep[\textsection~2.2]{kuhn2019wasserstein} and \citep[\textsection~6]{zhen2023unified} to show that the discrete distribution 
    \begin{equation*}
        \Q^\star = \sum_{j \in [J]} \sum_{i \in \cI_j^+} q_{ij}^\star \delta_{\hat z_j+{\xi_{ij}^\star}/{q_{ij}^\star}}
    \end{equation*}
    solves the dual DRO problem~\eqref{eq:dual:dro}. Note first that $\Q^\star$ is indeed a probability distribution because
    \[
        \Q^\star(z\in\R^d) =\sum_{j\in[J]}\sum_{i \in \cI_j^+} q_{ij}^\star=\sum_{j\in[J]}\sum_{i \in [I]} q_{ij}^\star = \sum_{j\in[J]} p_j =1,
    \]
    where the second equality holds because~$q^\star_{ij}=0$ for all~$i
    \in \cI_j^0\cup \cI_j^\infty$, while the third equality follows from the constraints of problem~\eqref{eq:dual}. More precisely, we have~$\Q^\star(z\in\cZ) =1$ because the constraints of~\eqref{eq:dual} imply that~$f_k(\hat z_j + \xi_{ij}^\star / q_{ij}^\star) \leq 0$ whenever~$i\in\cI^+_j$. In addition, one readily verifies that
    \[
        d_c(\Q^\star,\hat{\P}) \leq \sum_{j \in [J]} \sum_{i \in \cI^+_j} q_{ij}^\star c(\hat z_j + \xi_{ij}^\star / q_{ij}^\star, \hat z_j)= \sum_{j \in [J]} \sum_{i \in [I]} q_{ij}^\star c(\hat z_j + \xi_{ij}^\star / q_{ij}^\star, \hat z_j) \leq \eps,
    \]
    where the first inequality holds because the transportation plan that moves mass~$q_{ij}^\star$ from~$\hat z_j$ to~$\hat z_j+\xi_{ij}^\star / q_{ij}^\star$ for all~$j\in[J]$ and~$i\in\cI^+_j$ is an admissible coupling of~$\hat{\P}$ and~$\Q^\star$, the equality holds because~$\cI_j^\infty=\emptyset$ and because $q_{ij}^\star c(\hat z_j + \xi_{ij}^\star / q_{ij}^\star, \hat z_j)= 0 c(\hat z_j+0/0,\hat z_j)=0$ for all~$i \in \cI_j^0$ thanks to our conventions for perspective functions, and the second inequality follows from the constraints of problem~\eqref{eq:dual}. In summary, we have thus shown that~$\Q^\star\in\B_\eps(\hat{\P})$. It remains to be shown that~$\Q^\star$ is optimal in~\eqref{eq:dual:dro}. To this end, note that
    \begin{align}
        \label{eq:sandwich}
        \begin{aligned}
            \sup_{\Q \in \B_\eps(\hat{\P})} \inf_{\theta \in \Theta} \E_{Z \sim \Q} [\ell(\theta, Z)]
            \,&\geq \inf_{\theta \in \Theta} ~ \E_{Z\sim\Q^\star} [\ell(\theta, Z)] \geq \inf_{\theta \in \Theta} ~ \sum_{j \in [J]} \sum_{i \in \cI_j^+} q_{ij}^\star \ell_i(\theta, \hat z_j + \xi_{ij}^\star / q_{ij}^\star) \\
            &= \inf_{\theta \in \Theta} ~ \sum_{i \in [I]} \sum_{j \in [J]} q_{ij}^\star \ell_i(\theta, \hat z_j + \xi_{ij}^\star / q_{ij}^\star)  \\
            &\geq - \sum_{i \in [I]} \sum_{j \in [J]} q_{ij}^\star \ell_i^{*1}(\alpha_{ij}^\star/q_{ij}^\star, \hat z_j + \xi_{ij}^\star / q_{ij}^\star) - \sum_{l \in [L]} \nu_l^\star g^*_l(\beta_l^\star / \nu_l^\star) \\
            &= \inf_{\theta \in \Theta} \sup_{\Q \in \B_\eps(\hat{\P})} \E_{Z \sim \Q} [\ell(\theta, Z)]\geq \sup_{\Q \in \B_\eps(\hat{\P})} \inf_{\theta \in \Theta} \E_{Z \sim \Q} [\ell(\theta, Z)],
        \end{aligned}
    \end{align}
    where the first inequality holds because~$\Q^\star\in\B_\eps(\hat{\P})$, the second inequality exploits Assumption~\ref{assumption:nash:convexity}\,\ref{assumption:nash:loss}, which implies that $\ell$ majorizes~$\ell_i$ for every~$i\in[I]$, and the first equality follows from the properties of perspective functions and from our assumption that~$\cI^\infty_j$ is empty. The third inequality then holds because $(\{\alpha_{ij}^\star\}_{i,j}, \{\nu_l^\star, \beta_l^\star\}_{l})$ is feasible in~\eqref{eq:outer:max}, which is dual to the resulting minimization problem over~$\theta$, and because any dual feasible solution provides a lower bound on the infimum of the primal problem. The second equality holds because the primal DRO problem~\eqref{eq:dro} and the finite convex program~\eqref{eq:dual} share the same optimal value, which was established in Step~1, while the fourth and last inequality follows from weak duality. Hence, all inequalities in~\eqref{eq:sandwich} are in fact equalities, which proves that the optimal values of the primal and dual DRO problems~\eqref{eq:dro} and~\eqref{eq:dual:dro} coincide and that $\Q^\star$ solves~\eqref{eq:dual:dro}.

    Assume now that $\cI_j^\infty \neq \emptyset$ for some~$j \in [J]$,
    and select any~$i\in\cI_j^\infty$. Since $q_{ij}^\star = 0$, \cite[Theorem~8.6]{rockafellar1970convex} implies that $\xi^\star_{ij}$ is a recession direction of~$\cZ$ (see also \citep[Lemma~C.8\,(ii)]{zhen2023unified}), and since~$\xi_{ij}^\star \neq 0$, the support set~$\cZ$ cannot be compact. Similarly, one readily verifies that
    $c(z,\hat z)$ cannot grow superlinearly with~$z$ for otherwise the left hand side of the last constraint in~\eqref{eq:dual} would evaluate to infinity \cite[Remark~C.5]{zhen2023unified}, which contradicts the feasibility of~$(\{q_{ij}^\star, \xi_{ij}^\star, \alpha_{ij}^\star\}_{i,j}, \{\nu_l^\star, \beta_l^\star\}_{l})$. Assumption~\ref{assumption:nash:regularity} thus~implies that $\E_{Z \sim \Q} [\ell(\theta, Z)]$ is inf-compact in~$\theta\in\Theta$ for some $\Q \in \B_\eps(\hat \P)$.
    Consider now the discrete distributions 
    \begin{align*}
        \Q(n) = \sum_{j \in [J]} \sum_{i \in \cI_j^+\cup \cI^\infty_j} q_{ij}(n) \,  \delta_{z_{ij}(n)}
    \end{align*}
    indexed by $n \in \N$, where
    \begin{align*}
        q_{ij}(n) = \begin{cases} q_{ij}^\star ( 1 - | \cI_j^\infty | / n ) & \text{if } i \in \cI_j^+ \\ p_j / n & \text{if } i \in \cI_j^\infty \end{cases} 
        \quad \text{and} \quad
        z_{ij}(n) = \begin{cases} \hat z_j + \xi_{ij}^\star / q_{ij}^\star & \text{if } i \in \cI_j^+ \\ \hat z_j + n \hspace{0.1ex} \xi_{ij}^\star / p_j & \text{if } i \in \cI_j^\infty. \end{cases}
    \end{align*}
    Following \citep[\textsection~6]{zhen2023unified}, we now prove that~$\Q(n)$ is feasible and asymptotically optimal in~\eqref{eq:dual:dro} as~$n$ increases. To this end, we first show that $\Q(n)\in\B_{\eps}(\hat{\P})$ for all $n > \underline n=\max_{j\in[J]}|\cI_j^\infty|$. It is easy to verify that $\Q(n)$ is indeed a probability distribution for every $n >\underline n$ because $\sum_{j\in[J]} \sum_{i\in \cI_j^+\cup \cI^\infty_j}q_{ij}(n)=1$ and~$q_{ij}(n) \geq 0$ for all $i \in \cI_j^+ \cup \cI_j^\infty$ and $j \in [J]$ thanks to the constraints of problem~\eqref{eq:dual}. In addition, $\Q(n)$ is supported on $\cZ$, that is, $z_{ij}(n) \in \cZ$ for each~$j \in [J]$ and $i \in \cI_j^+ \cup \cI_j^\infty$. In particular, the constraints of problem~\eqref{eq:dual} imply that~$\xi_{ij}^\star$ is a recession direction of~$\cZ$ for every~$i \in \cI_j^\infty$ (see, {\em e.g.}, \cite[Lemma~C.8]{zhen2023unified}). Hence, for any fixed~$j\in[J]$ and~$i\in\cI^\infty_j$, the distribution~$\Q(n)$ sends a decreasing probability mass~$q_{ij(n)}=p_j/n$ along the ray emanating from~$\hat z_j$ in the direction of~$\xi_{ij}^\star$ without ever leaving the support set~$\cZ$ as $n$ grows. Finally, for all~$n > \underline n$, one readily verifies that
    \begin{align*}
        d_c(\Q(n), \hat{\P}) 
        &\leq \sum_{j \in [J]} \sum_{i \in \cI_j^+ \cup \cI_j^\infty} q_{ij}(n) \, c \big(z_{ij}(n), \hat z_j \big)  \\
        &= \sum_{j \in [J]} \sum_{i \in \cI_j^+} q_{ij}^\star \left( 1 - \frac{|\cI_j^\infty|}{n} \right) c \left( \hat z_j + \frac{\xi_{ij}^\star}{q_{ij}^\star}, \hat z_j \right) + \sum_{j \in [J]} \sum_{i \in \cI_j^\infty} \frac{p_j}{n} c \left( \hat z_j + n \frac{\xi_{ij}^\star}{p_j}, \hat z_j \right) \\
        &\leq \sum_{j \in [J]} \sum_{i \in \cI_j^+} q_{ij}^\star c \left( \hat z_j + \frac{\xi_{ij}^\star}{q_{ij}^\star}, \hat z_j \right) + \sum_{j \in [J]} \sum_{i \in \cI_j^\infty} \lim_{n \to \infty} \frac{p_j}{n} c \left( \hat z_j + n\frac{\xi_{ij}^\star}{p_j}, \hat z_j \right) \\
        &= \sum_{j \in [J]} \sum_{i \in \cI_j} q_{ij}^\star c \left( \hat z_j + \frac{\xi_{ij}^\star}{q_{ij}^\star}, \hat z_j \right) \leq \eps,
    \end{align*}
    where the first inequality holds because the transportation plan that moves mass~$q_{ij}(n)$ from~$\hat z_j$ to~$z_{ij}(n)$ for all~$j\in[J]$ and~$i\in\cI^+_j\cup \cI^\infty_j$ is an admissible coupling of~$\hat{\P}$ and~$\Q(n)$, and the second inequality holds because the transportation cost function is non-negative and convex (see Assumption~\ref{assumption:nash:convexity}\,\ref{assumption:nash:cost}), which implies that both terms in the second line of the above expression are non-decreasing in~$n$. The last equality follows from the definition of the perspective function~$q_{ij}^\star c(\hat z_j + \xi_{ij}^\star/q_{ij}^\star, \hat z_j)$ for~$q_{ij}^\star=0$, and the last inequality uses the constraints of~\eqref{eq:dual}. Hence, $\Q(n)\in \B_{\eps}(\hat{\P})$ for all~$n > \underline n$. It remains to be shown that~$\Q(n)$ is asymptotically optimal in~\eqref{eq:dual}. Using a similar reasoning as in~\eqref{eq:sandwich}, we obtain
    \begin{align*}
        \sup_{\Q \in \B_\eps(\hat{\P})}& \inf_{\theta \in \Theta} \; \E_{Z \sim \Q} [\ell(\theta, Z)] 
        = \inf_{\theta \in \Theta} \sup_{\Q \in \B_\eps(\hat{\P})} \, \E_{Z\sim\Q} [\ell(\theta, Z)] 
        \geq \inf_{\theta \in \Theta} \lim_{n \to \infty} 
        \E_{Z\sim\Q(n)} [\ell(\theta, Z)] \\
        &\geq \inf_{\theta \in \Theta} \lim_{n \to \infty} \sum_{j \in [J]} \sum_{i \in \cI_j^+ \cup \cI_j^\infty} q_{ij}(n) \ell_i \left(\theta, z_{ij}(n) \right) \\&= \inf_{\theta \in \Theta} \sum_{j \in [J]} \sum_{i \in \cI_j^+ \cup \cI_j^\infty} q_{ij}^\star \ell_i (\theta, \hat z_j + \xi_{ij}^\star / q_{ij}^\star ) \\
        &= \inf_{\theta \in \Theta} \sum_{j \in [J]} \sum_{i \in [I]} q_{ij}^\star \ell_i (\theta, \hat z_j + \xi_{ij}^\star / q_{ij}^\star ) \\
        & \geq \left\{ \begin{array}{cl}
            \max & \DS - \sum_{j \in [J]} \sum_{i \in [I]} q_{ij}^\star \ell_i^{*1}(\alpha_{ij}/q_{ij}^\star, \hat z_j + \xi^\star_{ij} / q^\star_{ij}) - \sum_{l \in [L]} \nu_l g^*_l(\beta_l / \nu_l) \\[1ex]
            \st & \alpha_{ij} \in \R^m, ~ \beta_l \in \R^m, ~ \nu_l \in \R_+ \quad \forall i \in \cI_j^+ \cup \cI_j^\infty, \, j \in [J], \, l \in [L] \\[1ex]
            & \DS \sum_{j \in [J]} \sum_{i \in \cI_j^+ \cup \cI_j^\infty} \alpha_{ij} + \sum_{l \in [L]} \beta_l = 0
        \end{array}
        \right. \\
        & \geq - \sum_{i \in [I]} \sum_{j \in [J]} q^\star_{ij} \ell_i^{*1}(\alpha^\star_{ij}/q_{ij}^\star, \hat z_j + \xi^\star_{ij} / q^\star_{ij}) - \sum_{l \in [L]} \nu^\star_l g^*_l(\beta^\star_l / \nu^\star_l) \\
        & = \inf_{\theta \in \Theta} \sup_{\Q \in \B_\eps(\hat{\P})} \E_{Z \sim \Q} [\ell(\theta, Z)] 
        \geq \sup_{\Q \in \B_\eps(\hat{\P})} \inf_{\theta \in \Theta} ~ \E_{Z \sim \Q} [\ell(\theta, Z)].
    \end{align*}
    Here, the first equality follows from Theorem~\ref{theorem:minimax}\,\ref{theorem:minimax:minsup}, which applies for the following reasons. First, $\E_{Z\sim\Q}[\ell(\theta, Z)]$ is lower semicontinuous in~$\theta$ for every $\Q \in \B_\eps(\hat \P)$ by Fatou's lemma, which may be used thanks to Assumptions~\ref{assumption:loss}\,\ref{assumption:loss:Fatou} and~\ref{assumption:nash:convexity}\,\ref{assumption:nash:loss}. In addition, $\E_{Z \sim \Q} [\ell(\theta, Z)]$ is inf-compact in~$\theta\in\Theta$ for some $\Q \in \B_\eps(\hat \P)$ thanks to Assumption~\ref{assumption:nash:regularity}. Finally, we have $\sup_{\Q\in \B_\eps(\hat{\P})}\E_{Z\sim\Q}[\ell(\theta, Z)]<\infty$ thanks to Assumption~\ref{assumption:regularity}\,\ref{assumption:regularity:infty}. The first and second inequalities in the above expression hold because~$\Q(n)\in \B_{\eps}(\hat{\P})$ for all~$n > \underline n$ and because~$\ell$ majorizes~$\ell_i$ for all~$i\in[I]$, respectively. 
    The second equality holds because~$\ell_i(\theta,z)$ adopts only finite values and is concave and thus also continuous in~$z$ for every~$\theta\in\Theta$ and~$i\in[I]$. The third equality exploits the properties of perspective functions. The third inequality follows from weak duality; see also problem~\eqref{eq:outer:max}. The fourth inequality holds because $(\{\alpha_{ij}^\star\}_{i,j}, \{\nu_l^\star, \beta_l^\star\}_{l})$ is feasible in the emerging dual maximization problem. The last equality exploits our insight from Step~1 whereby the optimal values of the primal DRO problem~\eqref{eq:dro} and the finite convex program~\eqref{eq:dual} match, and the last inequality follows from weak duality. Again, we may conclude that all inequalities in the above expression are in fact equalities, which proves that the optimal values of the primal and dual DRO problems~\eqref{eq:dro} and~\eqref{eq:dual:dro} coincide and that $\Q(n)$ is asymptotically optimal in~\eqref{eq:dual:dro}.
\end{proof}

The proof of Proposition~\ref{proposition:dual} reveals that if Assumptions~\ref{assumption:loss}\,\ref{assumption:loss:Fatou}, \ref{assumption:regularity}\,\ref{assumption:regularity:infty}, \ref{assumption:nash:reference}, \ref{assumption:nash:convexity}, \ref{assumption:slater} and~\ref{assumption:nash:regularity} hold and if~$\eps > 0$, then the infimum of the primal DRO problem~\eqref{eq:dro} equals the supremum of the dual DRO problem~\eqref{eq:dual:dro}. Hence, strong duality holds. 
The proof of Proposition~\ref{proposition:dual} also reveals that the convex reformulation~\eqref{eq:dual} of the dual DRO problem~\eqref{eq:dual:dro} is always solvable and that any maximizer of~\eqref{eq:dual} can be used to construct a least favorable distribution that is optimal in~\eqref{eq:dual:dro}, if $\cI^\infty_j=\emptyset$ for all~$j\in[J]$, or a sequence of distributions that are {\em asymptotically} optimal in~\eqref{eq:dual:dro}, if $\cI^\infty_j\neq \emptyset$ for some~$j\in[J]$. Finally, the proof of Proposition~\ref{proposition:dual} shows that Assumption~\ref{assumption:nash:regularity} can be replaced with the following condition.

\begin{assumption}
    \label{assumption:nash:regularity2}
    At least one of the following two conditions is satisfied: (i) $\E_{Z \sim \Q} [\ell(\theta, Z)]$ is inf-compact in~$\theta\in\Theta$ for some $\Q\in \B_\eps(\hat \P)$ or (ii) the index set $\cI_j^\infty = \{ i \in [I] : q_{ij}^\star = 0,\, \xi_{ij}^\star \neq 0 \}$ is empty for every $j\in[J]$, where $(\{q_{ij}^\star, \xi_{ij}^\star, \alpha_{ij}^\star\}_{i,j}, \{\nu_l^\star, \beta_l^\star\}_{l})$ is a maximizer of problem~\eqref{eq:dual}.
\end{assumption}

The next theorem formalizes the above insights, thus identifying conditions under which one can compute Nash equilibra for the DRO problem~\eqref{eq:dro} by solving the finite convex programs~\eqref{eq:primal} and~\eqref{eq:dual}.

\begin{theorem}[Computing Nash equilibria]
    \label{theorem:nash:computation}
    If Assumptions~\ref{assumption:loss}\,\ref{assumption:loss:Fatou}, \ref{assumption:regularity}\,\ref{assumption:regularity:infty}, \ref{assumption:nash:reference}, \ref{assumption:nash:convexity}, \ref{assumption:slater} and~\ref{assumption:nash:regularity} hold and if~$\eps > 0$, then the optimal values of the primal and dual DRO problems~\eqref{eq:dro} and~\eqref{eq:dual:dro} match, and the following~hold.
    \begin{enumerate}[label=(\roman*)]
        \item \label{theorem:nash:computation:theta}
        If $(\theta^\star, \lambda^\star, \{ s_j^\star \}_j, \{ \zeta^{\ell \star}_{ij}, \zeta^{c \star}_{ij} \}_{i,j}, \{ \tau_{ijk}^\star, \zeta^{f \star}_{ijk} \}_{i,j,k} )$ solves~\eqref{eq:primal}, then~$\theta^\star$ solves the primal DRO problem~\eqref{eq:dro}.
        \item \label{theorem:nash:computation:Q}
        If $(\{\beta_l^\star\}_l, \{\nu_l^\star\}_l, \{q_{ij}^\star\}_{i,j}, \{\xi_{ij}^\star\}_{i,j}, \{\alpha_{ij}^\star\}_{i,j})$ solves~\eqref{eq:dual} with $\cI_j^\infty=\emptyset$ for every~$j\in[J]$, then the discrete distribution $\Q^\star = \sum_{j \in [J]} \sum_{i \in \cI_j^+} q_{ij}^\star \delta_{\hat z_j+{\xi_{ij}^\star}/{q_{ij}^\star}}$ solves the dual DRO problem~\eqref{eq:dual:dro}.
    \end{enumerate}
\end{theorem}

\begin{proof}
The claim follows from the proofs of Propositions~\ref{proposition:primal} and~\ref{proposition:dual} and the above discussion.
\end{proof}

The conditions of Theorem~\ref{theorem:nash:computation} do not ensure that~\eqref{eq:primal} is solvable and, even though they guarantee the solvability of~\eqref{eq:dual}, they do not ensure that~\eqref{eq:dual} has a solution with~$\cI_j^\infty=\emptyset$ for every~$j\in[J]$; see the proofs of Propositions~\ref{proposition:primal} and~\ref{proposition:dual}. However, Theorem~\ref{theorem:nash:computation} implies that if~\eqref{eq:primal} is solvable and~\eqref{eq:dual} has a solution with~$\cI^\infty_j=\emptyset$ for all~$J\in [J]$, then the DRO problem~\eqref{eq:dro} admits a Nash equilibrium $(\theta^\star,\Q^\star)$, which can be computed from the solutions of the finite convex programs~\eqref{eq:primal} and~\eqref{eq:dual}. The following lemma identifies easily checkable sufficient conditions for problem~\eqref{eq:primal} to be solvable.

\begin{lemma}[Existence of~$\theta^\star$]
\label{lem:solvability-primal}
    Suppose that Assumptions~\ref{assumption:loss}\,\ref{assumption:loss:Fatou}, \ref{assumption:regularity}\,\ref{assumption:regularity:infty}, \ref{assumption:nash:reference}, \ref{assumption:nash:convexity}, \ref{assumption:slater} and~\ref{assumption:nash:regularity} hold and that~$\eps > 0$. Then, the finite convex program~\eqref{eq:primal} is solvable if there exists~$\Q\in \B_\eps(\hat \P)$ such that $\E_{Z \sim \Q} [\ell(\theta, Z)]$ is inf-compact in~$\theta\in\Theta$ or if~\eqref{eq:dual} admits a Slater point. 
\end{lemma}

\begin{proof}
Assume first that~$\E_{Z \sim \Q} [\ell(\theta, Z)]$ is inf-compact in~$\theta\in\Theta$ for some $\Q\in \B_\eps(\hat \P)$. This implies that the worst-case expected loss $\sup_{\Q\in \B_\eps(\hat \P)}\E_{Z \sim \Q} [\ell(\theta, Z)]$ is also inf-compact in~$\theta\in\Theta$. Without loss of generality, we may thus assume that~$\Theta$ is compact. Otherwise, we can simply replace $\Theta$ with a non-empty compact sublevel set of the worst-case expected loss. Note that if the worst-case expected loss is identically equal to~$\infty$ such that {\em all} of its sublevel sets are empty, then we can replace~$\Theta$ with {\em any} non-empty compact subset of~$\Theta$. From the proof of Proposition~\ref{proposition:dual} we know that~\eqref{eq:inner:inf} and~\eqref{eq:inner:sup} are dual convex programs and that~\eqref{eq:inner:sup} admits a Slater point. This implies via~\cite[Theorem~2\,(i)]{zhen2023unified} that~\eqref{eq:inner:inf} is solvable irrespective of~$\theta$. As~\eqref{eq:primal} minimizes the minimum of~\eqref{eq:inner:inf} across all~$\theta\in\Theta$ and as~$\Theta$ is compact, we may conclude that problem~\eqref{eq:primal} is solvable. Assume next that~\eqref{eq:dual} admits a Slater point. As~\eqref{eq:dual} is dual to~\eqref{eq:primal}, it follows again from~\cite[Theorem~2\,(i)]{zhen2023unified} that~\eqref{eq:primal} must be solvable.
\end{proof}

The next lemma identifies conditions ensuring that~\eqref{eq:dual} has a solution with~$\cI^\infty_j=\emptyset$ for all~$J\in [J]$.

\begin{lemma}[Existence of $\Q^\star$]
\label{lem:solvability-dual}
    Suppose that Assumptions~\ref{assumption:loss}\,\ref{assumption:loss:Fatou}, \ref{assumption:regularity}\,\ref{assumption:regularity:infty}, \ref{assumption:nash:reference}, \ref{assumption:nash:convexity}, \ref{assumption:slater} and~\ref{assumption:nash:regularity} hold and that~$\eps > 0$. Then, the finite convex program \eqref{eq:dual} has a solution with~$\cI^\infty_j=\emptyset$ for all~$J\in [J]$ if the support set~$\cZ$ is compact or if the transportation cost function~$c(z,\hat z)$ grows superlinearly with~$z$.
\end{lemma}

\begin{proof}
This is an immediate consequence of the proof of Proposition~\ref{proposition:dual}. 
\end{proof}

Under the conditions of Lemma~\ref{lem:solvability-primal}, any solution of~\eqref{eq:primal} gives rise to a robust decision~$\theta^\star$, and under the conditions of Lemma~\ref{lem:solvability-dual}, any solution of~\eqref{eq:dual} gives rise to a least favorable distribution~$\Q^\star$. By construction, $\theta^\star$ and~$\Q^\star$ form a Nash equilibrium for the DRO problem~\eqref{eq:dro}. The constructive results of this section thus strengthen and complement the existential results of Theorem~\ref{theorem:minimax}. We emphasize, in particular, that the conditions of Lemma~\ref{lem:solvability-dual} are reminiscent of the growth condition specified in Assumption~\ref{assumption:cost}\,\ref{assumption:cost:lb}, which was instrumental in Theorem~\ref{theorem:minimax}\,\ref{theorem:minimax:maxinf} to prove that the dual DRO problem~\eqref{eq:dual:dro} is solvable. Similarly, the conditions of Lemma~\ref{lem:solvability-primal} are reminiscent of the inf-compactness condition that was needed in Theorem~\ref{theorem:minimax}\,\ref{theorem:minimax:minsup} to prove that the primal DRO problem~\eqref{eq:dro} is solvable.

The results of this section readily extend to decision problems with loss functions that display an arbitrary dependence on some of the uncertain problem parameters provided that nature cannot perturb their marginal distribution. Such decision problems naturally arise in machine learning applications (see Section~\ref{section:numerical}). Propositions~\ref{proposition:primal} and~\ref{proposition:dual} as well as Theorem~\ref{theorem:nash:computation} remain valid in this generalized setting with obvious minor modifications but at the expense of a higher notational overhead; see Appendix~\ref{appendix:nash:ml}.

Focusing on discrete reference distributions is the key enabler for the finite convex reformulation results of this section. Indeed, these results critically rely on the insight that if~$\hat\P$ is discrete, then we may restrict $\B_\eps(\hat \P)$ to contain only discrete distributions without reducing the optimal value of the DRO problem~\eqref{eq:dro} and its dual~\eqref{eq:dual:dro}. Searching over discrete distributions is significantly easier than searching over arbitrary distributions in~$\B_\eps(\hat \P)$. An analogous reasoning is used in \citep{shafieezadeh2018wasserstein,nguyen2023bridging}, where~$\hat\P$ is an elliptical distribution and~$\ell(\theta, z)$ is quadratic in~$z$. In that case, the optimal values of the primal and dual DRO problems do not change if one restricts~$\B_\eps(\hat \P)$ to contain only elliptical distributions of the same type as~$\hat\P$. All these elliptical distributions are uniquely determined by their---finite-dimensional---mean vectors and covariance matrices, thus giving rise to tractable convex reformulations.

\section{Regularization by Robustification}
\label{section:regularization}

It is well known that many regularization schemes in statistics and machine learning admit a robustness interpretation. Apart from carrying an aesthetic appeal, such interpretations often lead to generalization bounds for the optimizers of regularized learning models. This prompts us to seek a comprehensive theory of regularization and optimal transport-based robustification that unifies various specialized results scattered across the existing literature and that rationalizes, for the first time, a broad range of higher-order variation regularization schemes.
In Section~\ref{section:highorder:variation} we first study the {\em primal} regularizing effects of robustification by relating the worst-case expected loss to the expected loss under the reference distribution adjusted by Lipschitz and higher-order variation regularization terms. In Section~\ref{section:regularization:linear} we then study the {\em dual} regularizing effects of robustification by relating the worst-case expected loss to the expected value of a regularized ({\em e.g.}, smoothed) loss function under the reference distribution.

\subsection{Primal Regularizating Effects of Robustification}
\label{section:highorder:variation}

In this section we will show that, under natural regularity conditions, the worst-case expected loss across all distributions in a generic optimal transport-based ambiguity set is bounded above by the sum of the expected loss under the reference distribution and several regularization terms that penalize $L^p$-norms and Lipschitz moduli of the higher-order derivatives of the loss function. Our results generalize and unify several existing results, which have revealed intimate connections between robustification and gradient regularization~\cite{gao2024wasserstein}, Hessian regularization~\cite{bartl2020robust} and Lipschitz regularization~\citep{esfahani2018data}; see also~\cite{blanchet2019robust,gao2024wasserstein,shafieezadeh2019regularization,shafieezadeh2015distributionally}.

The subsequent discussion focuses on the inner maximization problem in~\eqref{eq:dro}, where~$\theta$ is fixed. We thus temporarily suppress the dependence of the loss function on~$\theta$ and use the shorthand notation~$\ell(z)$ throughout this section. To exclude trivialities, we assume here that the support set~$\cZ$ is contained in the relative interior of the domain of $\ell(z)$. 
In addition, we will also adopt the following notational conventions. For any~$k \in \Z_+$, we use~$D^k \ell(z)$ to denote the totally symmetric tensor of all $k$-th order partial derivatives of~$\ell(z)$. Accordingly, $D^k \ell(z)[\xi_1, \ldots, \xi_k]$ stands for the directional derivative of $\ell(z)$ along the directions~$\xi_i\in\R^d$ for~$i \in [k]$. If~$\xi_i=\xi$ for all~$i\in[k]$, then we use the shorthand~$D^k \ell(z)[\xi]^k$. Any norm~$\|\cdot\|$ on~$\R^d$ induces a norm on the space of totally symmetric $k$-th order tensors through
\begin{align*}
\|D^k \ell(z)\| 
&= \sup_{\xi_1, \ldots, \xi_k} \left\{ \left| D^k \ell(z)[\xi_1, \ldots, \xi_k] \right|: \|\xi_i\| \leq 1 ~ \forall i \in [k] \right\} = \sup_{\xi} \left\{ \left| D^k \ell(z)[\xi]^k \right|: \| \xi \| \leq 1 \right\} ,
\end{align*}
where the second equality follows from the symmetry of~$D^k \ell(z)$ \citep[Satz~1]{Banach1938}. By slight abuse of notation, we use the same symbol~$\|\cdot\|$ for the tensor norm induced by the vector norm~$\|\cdot\|$.

The results of this section will depend on the following smoothness conditions.

\begin{assumption}[Smoothness properties of the loss function] ~
    \label{assumption:derivatives}
    \begin{enumerate}[label=(\roman*)]
        \item \label{assumption:differentiability:p:essential}
        The domain of the loss function $\ell(z)$ is convex, $\ell(z)$ is $p$ times continuously differentiable on $\rint(\dom(\ell))$, and there exist~$G>0$ and~$p\in\N$ with $\ell(z) \leq G(1 + \| z \|^{p})$, $\| D^{k}\ell(z) \| \leq G(1 + \| z \|^{p-k})$ for all $k \in [p-1]$, and $\|D^{p}\ell(z)\|\leq G$ for all~$z \in \rint(\dom(\ell))$.
        \item \label{assumption:derivatives:lipschitz:uniform} The tensor $D^k\ell(z)$ of $k$-th order partial derivatives is Lipschitz continuous in~$z$ for all~$k\in[p-1]$ throughout $\rint(\dom(\ell))$.
    \end{enumerate}
\end{assumption}

We are now ready to state the main result of this section.

\begin{theorem}[Regularization by robustification]
    \label{theorem:regularization}
    Suppose that Assumption~\ref{assumption:cost}\,\ref{assumption:cost:integrable} holds. If there exists a norm~$\|\cdot\|$ on~$\R^d$ such that~$c(z,\hat z)$ satisfies Assumption~\ref{assumption:cost}\,\ref{assumption:cost:lb} for $d(z,\hat z) = \|z-\hat z\|$ and such that~$\ell(z)$ satisfies Assumption~\ref{assumption:derivatives}\,\ref{assumption:differentiability:p:essential} with respect to~$\|\cdot\|$, then the worst-case expected loss satisfies
    \begin{align}
    \label{eq:regularization:1}
        \sup_{\Q \in \B_{\eps}(\hat{\P})} \E_{Z \sim \Q} \left[ \ell(Z) \right] 
        \leq \E_{\hat Z \sim \hat{\P}} [\ell(\hat Z)] + \sum_{k=1}^{p-1} \frac{\eps^\frac{1}{p_k}}{k!}\, \E_{\hat Z \sim \hat{\P}} \left[ \| D^{k} \ell(\hat Z)\|^{q_k} \right]^{\frac{1}{q_k}} + \frac{\eps}{p!} \, \lip( D^{p-1}\ell)<\infty, 
    \end{align}
    where ~$p_k = p/k$ and~$q_k = p / (p-k)$ for all $k\in[p-1]$. 
\end{theorem}

\begin{proof}
    Fix any~$\Q\in \B_{\eps}(\hat{\P})$, and use $\pi^\star$ to denote an optimal coupling of $\Q$ and $\hat{\P}$, which exists thanks to~\citep[Theorem~4.1]{villani2008optimal}. By the defining properties of couplings, we have
    \begin{align*}
    \E_{Z \sim \Q} \left[ \ell(Z) \right] - \E_{\hat Z \sim \hat{\P}} [ \ell(\hat Z)] 
    = \int_{\cZ \times \cZ} \left[ \ell(z) - \ell(\hat z)\right] \mathrm{d}\pi^\star (z,\hat z).
    \end{align*}
    Representing the integrand as a Taylor series with Lagrange remainder yields 
    \begin{align*}
    \ell(z) - \ell(\hat z) 
    &= \sum_{k=1}^{p-1} \frac{1}{k!} D^{k} \ell(\hat z)\left[ z - \hat z \right]^k + \frac{1}{p!} D^{p} \ell(\hat z + t\cdot (z-\hat z)) \left[ z - \hat z \right]^p \\ 
    &\leq \sum_{k=1}^{p-1} \frac{1}{k!} \| D^{k} \ell(\hat z)\| \| z - \hat z \|^k + \frac{1}{p!} \|D^{p} \ell(\hat z + t\cdot(z - \hat z)) \| \| z - \hat z \|^p
    \end{align*}
    for some $t \in (0,1)$ that depends (measurably) on~$z$ and~$\hat z$ \citep[Theorem~2.2.5]{krantz2002primer}. The inequality in the second line follows directly from the multi-linearity of tensors and the definition of the tensor norm. Note that~$\ell$ is indeed differentiable at $\hat z+t(z-\hat z)$ because $z,\hat z\in \cZ\subseteq \rint(\dom(\ell))$ and because $\dom(\ell)$ is convex and $\ell$ is $p$ times continuously differentiable on $\rint(\dom(\ell))$ thanks to Assumption~\ref{assumption:derivatives}\,\ref{assumption:differentiability:p:essential}. By Hölder's inequality, the integral of the $k$-th term in the above sum against $\pi^\star$ admits the estimate
    \begin{align*}
    \int_{\cZ \times \cZ} \| D^{k} \ell(\hat z)\| \| z-\hat z \|^k \mathrm{d}\pi^\star(z,\hat z) 
    &\leq \left( \int_{\cZ \times \cZ} \| z-\hat z \|^{p_k} \mathrm{d}\pi^\star(z,\hat z)\right)^\frac{1}{p_k} \left(\int_{\cZ} \| D^{k} \ell(\hat z)\|^{q_k} \mathrm{d}\hat{\P}(\hat z) \right)^{\frac{1}{q_k}}  \\ 
    &\leq \eps^\frac{1}{p_k} \left(\int_{\cZ} \| D^{k} \ell(\hat z)\|^{q_k} \mathrm{d}\hat{\P}(\hat z) \right)^{\frac{1}{q_k}},
    \end{align*}
    where $p_k = p/k$ and $q_k = p / (p-k)$. The second inequality follows from Assumption~\ref{assumption:cost}\,\ref{assumption:cost:lb}, which ensures that $c(z, \hat z) \geq \| z - \hat z\|^p$ for all $z,\hat z\in\cZ$, and from the definition of $\pi^\star$ as an optimal coupling of $\hat{\P}$ and~$\Q\in \B_{\eps}(\hat{\P})$. A similar reasoning based on Hölder's inequality implies that
    \begin{align*}
    \int_{\cZ \times \cZ} \|D^{p}\ell(\hat z + t\cdot(z-\hat z))\| \| z-\hat z \|^p \mathrm{d}\pi^\star(z,\hat z) 
    &\leq \sup_{\tilde z \in \conv(\cZ)} \| D^{p} \ell(\tilde z)\| \int_{\cZ \times \cZ} \| z-\hat z \|^p \mathrm{d}\pi^\star(z,\hat z) \\
    &\leq \eps \sup_{\tilde z \in \conv(\cZ)} \| D^{p} \ell(\tilde z)\| 
    \leq \eps \lip( D^{p-1}\ell),
    \end{align*}
    where the dependence of $t$ on $z$ and $\hat z$ is notationally suppressed to avoid clutter. In summary, the above reasoning readily implies that
    \begin{align*}
        \E_{Z \sim \Q} \left[ \ell(Z) \right] - \E_{\hat Z \sim \hat{\P}} [ \ell(\hat Z)] \leq \sum_{k=1}^{p-1} \frac{\eps^\frac{1}{p_k}}{k!}  \E_{\hat Z \sim \hat{\P}} \left[ \| D^{k} \ell(\hat Z)\|^{q_k} \right]^{\frac{1}{q_k}} + \frac{\eps}{p!} \lip( D^{p-1}\ell). 
    \end{align*}
    Note that the right hand side of the resulting inequality is independent of~$\Q$. Thus, the inequality remains valid if we maximize its left hand side across all~$\Q \in \B_{\eps}(\hat{\P})$, which finally yields the upper bound in~\eqref{eq:regularization:1}. To demonstrate that this upper bound is finite, we observe that
    \begin{align*}
        \E_{\hat Z \sim \hat{\P}} \left[\| D^{k} \ell(\hat Z)\|^{q_k} \right] & \leq \E_{\hat Z \sim \hat{\P}} \left[ G(1+\|\hat Z\|^{p-k})^{q_k}\right] \leq \E_{\hat Z \sim \hat{\P}} \left[ G(2+\|\hat Z\|)^p\right] \\ & \leq G'+ \E_{\hat Z \sim \hat{\P}} \left[ \|\hat Z-\hat z_0\|^p \right] \leq G'+ \E_{\hat Z \sim \hat{\P}} \left[ c(\hat Z,\hat z_0) \right] \leq \infty\quad\forall k\in[p-1],
    \end{align*}
    where the first two inequalities follow from Assumption~\ref{assumption:derivatives}\,\ref{assumption:differentiability:p:essential} and the definition of~$q_k$, respectively. By the triangle inequality and the binomial theorem, there is $G'>0$ with~$G(2+\|z\|)^p\leq  G'(1+\|z-\hat z_0\|^p)$ for all~$z\in\cZ$, and this justifies the third inequality. The fourth inequality exploits Assumption~\ref{assumption:cost}\,\ref{assumption:cost:lb}, which holds for~$d(z,\hat z)=\|z-\hat z\|$, and the fifth inequality follows from Assumption~\ref{assumption:cost}\,\ref{assumption:cost:integrable}. Similarly, Assumption~\ref{assumption:derivatives}\,\ref{assumption:differentiability:p:essential} implies that $\lip( D^{p-1}\ell)\leq G<\infty$. Thus, the upper bound in~\eqref{eq:regularization:1} is finite.
\end{proof}

Theorem~\ref{theorem:regularization} asserts that the worst-case expected loss with respect to a generic optimal transport-based ambiguity set is bounded above by the sum of the expected loss under the reference distribution, $p-1$ different variation regularization terms, and a Lipschitz regularization term. The $k$-th variation regularization term penalizes the $L^{q_k}$-norm of the tensor of $k$-th order partial derivatives, and the Lipschitz regularization term penalizes the Lipschitz modulus of the tensor of $(p-1)$-st order partial derivatives of the loss. Variation regularizers such as regularizers penalizing the gradients, Hessians or tensors of higher-order partial derivatives are successfully used in various applications in machine learning. For example, they emerge in the adversarial training of neural networks~\citep{hein2017formal, jakubovitz2018improving, lyu2015unified, ororbia2017unifying, varga2017gradient, finlay2021scaleable} or in the stabilizing training of generative adversarial networks~\citep{roth2017stabilizing, nagarajan2017gradient, gulrajani2017improved}. In the context of image recovery, regularizers that penalize higher-order partial derivatives are sometimes preferred to total variation regularizers because they prevent undesirable staircase artifacts~\citep{bredies2010total}. For example, the ability of Hessian regularizers (or their finite difference approximations) to mitigate staircase effects has been documented in~\citep{lefkimmiatis2013hessian, lefkimmiatis2013poisson}. More generally, regularizers based on higher-order partial derivatives (or their finite difference approximations) can be used to preserve or enhance various image features such as edges or ridges \citep{hu2014generalized}. Theorem~\ref{theorem:regularization} gives these heuristic regularization schemes a theoretical justification.

The next corollary shows that the estimate~\eqref{eq:regularization:1} can be simplified by upper bounding all variation regularization terms by corresponding Lipschitz regularizers.

\begin{corollary}[Lipschitz regularization]
    \label{corollary:lipschitz-regularization}
    If all assumptions of Theorem~\ref{theorem:regularization} as well as Assumption~\ref{assumption:derivatives}\,\ref{assumption:derivatives:lipschitz:uniform} hold, then the worst-case expected loss satisfies
    \begin{align}
    \label{eq:regularization:2}
        \sup_{\Q \in \B_{\eps}(\hat{\P})} \E_{ Z \sim \Q} [ \ell( Z)] 
        \leq \E_{\hat Z \sim \hat{\P}} [\ell(\hat Z)] + \sum_{k=1}^{p} \frac{\eps^\frac{1}{p_k}}{k!} \lip ( D^{k-1}\ell)<\infty.
    \end{align}
\end{corollary}
\begin{proof}
    Assumption~\ref{assumption:derivatives}\,\ref{assumption:derivatives:lipschitz:uniform} implies that $\|D^k\ell(\hat Z)\|\leq \sup_{\hat z \in \conv(\cZ)} \| D^{k} \ell(\hat z)\| \leq \lip( D^{k-1} \ell)<\infty$ $\hat{\P}$-almost surely for all~$k\in[p-1]$. Substituting these conservative upper bounds into~\eqref{eq:regularization:1} yields~\eqref{eq:regularization:2}.
\end{proof}

We now specialize the results of Theorem~\ref{theorem:regularization} and Corollary~\ref{corollary:lipschitz-regularization} to Wasserstein ambiguity sets. 

\begin{corollary}[Regularization by robustification over Wasserstein balls]
    \label{corollary:Wasserstein-p:regularization}
    If all assumptions of Theorem~\ref{theorem:regularization} hold and $\W_{\eps} (\hat{\P}) = \{ \Q \in \cP(\cZ): W_p ( \Q , \hat{\P} ) \leq \eps \}$ is the $p$-th Wasserstein ball, then
    \begin{subequations}
    \begin{align}
    \label{eq:regularization_wasserstein:1}
        \sup_{\Q \in \W_{\eps}(\hat{\P})} \E_{Z \sim \Q} \left[ \ell(Z) \right] 
        \leq \E_{\hat Z \sim \hat{\P}} [\ell(\hat Z)] + \sum_{k=1}^{p-1} \frac{\eps^k}{k!} \E_{\hat Z \sim \hat{\P}}\left[ \| D^{k} \ell(\hat Z)\|^{q_k} \right]^{\frac{1}{q_k}} + \frac{\eps^p}{p!} \lip(D^{p-1} \ell)<\infty,
    \end{align}
    where $q_k = p / (p-k)$ for all~$k\in[p-1]$. If additionally Assumption~\ref{assumption:derivatives}\,\ref{assumption:derivatives:lipschitz:uniform} holds, then
    \begin{align}
    \label{eq:regularization_wasserstein:2}
        \sup_{\Q \in \W_{\eps}(\hat{\P})} \E_{Z \sim \Q} \left[ \ell(Z) \right] 
        \leq \E_{\hat Z \sim \hat{\P}} [\ell(\hat Z)] + \sum_{k=1}^{p} \frac{\eps^k}{k!} \lip( D^{k-1} \ell)<\infty.
    \end{align}
    \end{subequations}
\end{corollary}

The proof of Corollary~\ref{corollary:Wasserstein-p:regularization} follows from those of Theorem~\ref{theorem:regularization} and Corollary~\ref{corollary:lipschitz-regularization} and is thus omitted. We emphasize that Corollary~\ref{corollary:Wasserstein-p:regularization} holds even if~$\cZ$ fails to be convex. If $p > 1$ and $\cZ$ is convex, then one can show that the upper bound in~\eqref{eq:regularization_wasserstein:1} exactly matches the worst-case expected loss to first order in~$\eps$. To see this, observe first that the norm $\| D^{1} \ell(\hat z)\|$ of the tensor of first-order partial derivatives of~$\ell$ coincides with the dual norm $\|\nabla \ell(\hat z)\|_*$ of the gradient of~$\ell$. This is an direct consequence of the definition of the tensor norm. 
From \citep[Lemma~1]{gao2024wasserstein} (see also \citep[Theorem~8.7]{kuhn2024distributionally}) we further know that
\begin{align*}
    \sup_{\Q \in \W_{\eps}(\hat{\P})} \E_{Z \sim \Q} \left[ \ell(Z) \right] 
    = \E_{\hat Z \sim \hat{\P}} [\ell(\hat Z)] + \eps \, \E_{\hat Z \sim \hat{\P}}\left[ \| \nabla \ell(\hat Z)\|_*^q \right]^{\frac{1}{q}} + \cO(\eps^2),
\end{align*}
where $q=p/(p-1)$. Thus, the upper bound in~\eqref{eq:regularization_wasserstein:1} is indeed exact to first order in~$\eps$. The following example shows, however, that this bound is generally {\em not} exact to higher orders in~$\eps$. 

\begin{example}[Inexactness of~\eqref{eq:regularization_wasserstein:1}]
    \label{ex:inexactness}
    Assume that $\cZ=\R$ and $\ell(z)=-z^2/2$, and consider the Wasserstein ambiguity set $\W_{\eps} (\hat{\P}) = \{ \Q \in \cP(\cZ): W_p ( \Q , \hat{\P} ) \leq \eps \}$ with~$p=3$ and $\hat\P=\delta_0$. As the nominal distribution concentrates all probability mass at the maximum of the loss function, it is clear that $\sup_{\Q \in \W_{\eps}(\hat{\P})} \E_{Z \sim \Q}[ \ell(Z)]=0$. To construct the upper bound in~\eqref{eq:regularization_wasserstein:1}, note that $D^1\ell(z)=-z$ and $D^2\ell(z)=-1$ such that $\|D_1\ell(z)\|=|z|$, $\|D^2\ell(z)\|=1$ and $\lip(D^2\ell)=0$. As $\hat\P=\delta_0$, one thus readily verifies that the upper bound in~\eqref{eq:regularization_wasserstein:1} is given by $\eps^2/2$. This upper bound is correct to first order in~$\varepsilon$ but overestimates the second-order term in the Taylor expansion of the worst-case expected loss.
\end{example}

To conclude, we show that the results of this section generalize several bounds in the literature.

\begin{remark}[Hessian regularization]
    Corollary~\ref{corollary:Wasserstein-p:regularization} generalizes~\citep[Remark~10]{bartl2020robust}, which establishes an upper bound on the worst-case expected loss across all distributions in a Wasserstein ball with~$p>2$ and~$d(z,\hat z)=\|z-\hat z\|_2$. This bound penalizes the gradient and the Hessian of the loss function and holds only {\em asymptotically} for small~$\eps$. One readily verifies that this bound can be obtained from~\eqref{eq:regularization_wasserstein:1} by ignoring all terms of the order~$\cO(\eps^k)$ for $k>2$ and by noting that $\|D^{2} \ell(\hat z)\|_2 = \lambda_{\text{max}}\left(D^{2} \ell(\hat z)\right)$.
\end{remark}

\begin{remark}[Lipschitz regularization]
    For $p=1$ the upper bound in~\eqref{eq:regularization_wasserstein:2} collapses to the sum of the expected loss under the reference distribution and the Lipschitz modulus of~$\ell(z)$ weighted by the Wasserstein radius~$\eps$. This bound is exact for every~$\eps\geq 0$ provided that~$\cZ =\R^d$ and the loss function is convex~\citep[\textsection~6.2]{esfahani2018data}. This result gives classical regularization schemes in statistics and machine learning a robustness interpretation
    \cite{blanchet2019robust,gao2024wasserstein,shafieezadeh2019regularization,shafieezadeh2015distributionally}. More generally, this bound is exact when the loss function~$\ell(z)$ grows asymptotically at rate~$\lip(\ell)$ along some recession direction of~$\cZ$ \citep[Corollary~2]{gao2024wasserstein}. Note, however, that the upper bound in~\eqref{eq:regularization_wasserstein:2} fails to be tight for~$p>1$ even if~$\ell(z)$ is convex.
\end{remark}

To conclude, we re-introduce the decision variable~$\theta$ and highlight the usefulness of the bounds developed in this section for solving the DRO problem~\eqref{eq:dro}. If the loss function~$\ell(\theta,z)$ is {\em nonconvex} in~$\theta$ and/or~$z$, then it is generically hard to evaluate and to minimize the worst-case expectation {\em exactly} even if the nominal distribution is discrete. If the Lipschitz modulus of~$D^{p-1}\ell$ can be evaluated in closed form, however, then the upper bound of Theorem~\ref{theorem:minimax} can be computed efficiently because it consists of several sample averages. In addition, this upper bound can be minimized {\em approximately} by using stochastic gradient descent-type algorithms. By Theorem~\ref{theorem:minimax}, the resulting approximate minimizers enjoy the same robustness guarantees as the exact minimizers of the nonconvex DRO problem.

\subsection{Dual Regularizing Effects of Robustification}
\label{section:regularization:linear}
A distributionally robust learning model over linear hypotheses is a DRO problem of the form
\begin{align}
\label{eq:dro:linear}
\inf_{\theta \in \Theta}\; \sup_{\Q \in \B_{\eps}(\hat{\P})} \E_{Z \sim \Q} \left[ L(\inner{\theta}{Z}) \right],
\end{align}
where $L: \R \to (-\infty, +\infty]$ is a univariate loss function. Thus, \eqref{eq:dro:linear} is a special case of~\eqref{eq:dro} with $\ell(\theta, z)=L(\inner{\theta}{z})$. Throughout this section we will focus on problem~\eqref{eq:dro:linear} instead of the more general problem~\eqref{eq:dro}. Several important problems in operations research and machine learning can be framed as instances of~\eqref{eq:dro:linear}. Examples include portfolio optimization and newsvendor problems, plain vanilla or kernelized regression and classification problems, (linear) inverse problems, and phase retrieval problems. For problem~\eqref{eq:dro:linear} to be well-defined, we assume throughout this section that Assumption~\ref{assumption:continuity} holds. Proposition~\ref{proposition:strong:duality} thus implies that the DRO problem~\eqref{eq:dro:linear} is equivalent to the stochastic program
\begin{align}
    \label{eq:stochastic-program}
    \inf_{\theta\in\Theta,\,\lambda \geq 0}\; \lambda \eps + \E_{\hat Z \sim \hat{\P}} \left[ \ell_{c}(\theta, \lambda, \hat Z) \right],
\end{align}
whose objective function involves the expectation of the $c$-transform $\ell_c(\theta,\lambda, \hat Z)$ under the reference distribution~$\hat{\P}$. Throughout this section we also assume that $\cZ=\R^d$, and thus the $c$-transform satisfies
\begin{align}
    \label{eq:robust:loss}
    \ell_c(\theta,\lambda, \hat z) = \sup_{z \in \dom(c(\cdot,\hat z\,))} L(\inner{\theta}{z}) - \lambda c(z, \hat z).
\end{align}
While the conservative bound~\eqref{eq:regularization:1} derived in Section~\ref{section:highorder:variation} exposes the {\em primal} regularizing effects, the reformulation~\eqref{eq:stochastic-program} exposes the {\em dual} regularizing effects of robustification. Indeed, from a primal perspective, robustification essentially amounts to adding various regularization terms to the expected loss, and from a dual perspective, robustification essentially amounts to replacing the loss with its $c$-transform, which is best viewed as a regularized version of the loss function, and tuning the corresponding regularization parameter~$\lambda$. The main goal of this section is to shed more light on this dual perspective on robustification. Specifically, we will show that the $c$-transform $\ell_c$ can often be expressed in terms of a suitable envelope of~$L$ such as the Pasch-Hausdorff envelope or the Moreau envelope. In addition, we will show how the formulas~\eqref{eq:stochastic-program} and~\eqref{eq:robust:loss} can be exploited algorithmically. Indeed, in Section~\ref{section:nash:computation} we showed that problem~\eqref{eq:dro:linear} is equivalent to a finite convex program if all convexity conditions of Assumption~\ref{assumption:nash:convexity} are satisfied. In this section we will show that, as its objective function depends only on a one-dimensional projection of~$z$, problem~\eqref{eq:dro:linear} can sometimes be solved efficiently even if~$L$ fails to be (piecewise) concave or even if~$c$ fails to be convex.

We now leverage techniques from nonconvex optimization to recast the (possibly nonconvex) problem~\eqref{eq:robust:loss} for evaluating the $c$-transform as a {\em univariate} optimization problem.

\begin{theorem}[Nonconvex duality]
    \label{theorem:nonconvex:duality}
    The $c$-transform~\eqref{eq:robust:loss} admits the following dual reformulations.
    \begin{enumerate}[label=(\roman*)]
        \item \label{theorem:toland} If $L$ is proper, convex and lower semicontinuous, while $c$ is proper and lower semicontinuous, then\footnote{Recall from Proposition~\ref{proposition:primal} that $c^{*1}(\cdot,\hat z)$ denotes the conjugate of $c(\cdot,\hat z)$.} $\ell_c(\theta,\lambda, \hat z)=\sup_{\gamma \in \dom(L^*)} \; \lambda c^{*1}(\gamma \theta / \lambda, \hat z) - L^*(\gamma)$, which is convex in $(\theta, \lambda)$ for any~$\hat z$.
        \item \label{theorem:norm} If $c(z, \hat z) = \| z - \hat z \|^p$ for a norm~$\|\cdot\|$ and $p\geq 1$, then $\ell_c(\theta,\lambda, \hat z)=\sup_{\gamma \in \R} \; L\big( \inner{\theta}{\hat z} + \gamma \| \theta \|_* \big) - \lambda \hspace{0.1ex} |\gamma|^p.$
    \end{enumerate}
\end{theorem}

The reformulations of the nonconvex problem~\eqref{eq:robust:loss} derived in Theorem~\ref{theorem:nonconvex:duality} are generically also nonconvex. As these reformulations involve conjugate functions and dual norms, respectively, we refer to them as {\em dual} problems even though they constitute maximization problems like the primal problem~\eqref{eq:robust:loss}. These dual problems may be preferable to the primal problem~\eqref{eq:robust:loss} as they involve only a scalar decision. We remark that Theorem~\ref{theorem:nonconvex:duality}\,\ref{theorem:norm} generalizes~\citep[Theorem~1]{blanchet2018optimal}, where $\|\cdot\|$ is a Mahalanobis norm and~$p=2$. Indeed, Theorem~\ref{theorem:nonconvex:duality}\,\ref{theorem:norm} holds for arbitrary norms~$\|\cdot\|$ and exponents~$p$.

\begin{proof}[Proof of Theorem~\ref{theorem:nonconvex:duality}]
        If~$\lambda > 0$, then assertion~\ref{theorem:toland} follows from a classical duality result for nonconvex optimization problems due to Toland~\citep[\textsection~3.1]{toland1978duality}. In order to keep this paper self-contained, we nevertheless provide a short proof. Note first that~$L^{**} = L$ thanks to~\citep[Theorem~12.2]{rockafellar1970convex}, which applies because~$L$ is proper, convex and lower semicontinuous. Thus, the $c$-transform can be reformulated as
    \begin{align*}
        \ell_c(\theta,\lambda, \hat z)&= \sup_{z \in \dom(c(\cdot,\hat z))} \; L^{**}(\inner{\theta}{z})- \lambda c(z, \hat z) \\
        &= \sup_{z \in \dom(c(\cdot,\hat z))} \; \sup_{\gamma \in \dom(L^*)} \; \gamma \inner{\theta}{z} - L^*(\gamma) - \lambda c(z, \hat z) \\
        &= \sup_{\gamma \in \dom(L^*)} \; \sup_{z \in \dom(c(\cdot,\hat z))} \inner{\gamma \theta}{z} - \lambda c(z, \hat z) - L^*(\gamma)  \\
        &= \sup_{\gamma \in \dom(L^*)} \; \lambda c^{*1}(\gamma \theta / \lambda, \hat z) - L^*(\gamma).
    \end{align*}
    where the second equality holds because the biconjugate~$L^{**}$ is defined as the conjugate of~$L^*$ and the last equality follows from the elementary insight that the partial conjugate of~$\lambda c(\cdot, \hat z)$ coincides with~$\lambda c^{*1}(\cdot/\lambda, \hat z)$. This establishes the claim for~$\lambda > 0$. Note that the restriction to $\dom(L^*)$ in the last maximization problem is necessary because the objective is undefined if $\lambda c^{*1}(\gamma \theta / \lambda, \hat z)=L^*(\gamma)=\infty$.
 
    Assume now that~$\lambda = 0$. For clarity of exposition we define $\Phi(z) = c(z, \hat z)$ for $\hat z$ fixed. 
    As~$L^{**} = L$ and as $0c(z, \hat z)=0\Phi(z)=0$ throughout $\dom(\Phi)$, the $c$-transform can now be reformulated as
    \begin{align*}
        \ell_c(\theta,0, \hat z) =\sup_{z \in \dom(\Phi)} L^{**}(\inner{\theta}{z}) = \sup_{z \in \dom(\Phi)} \;\sup_{\gamma \in \dom(L^*)}\inner{\gamma \theta}{z}  - L^*(\gamma) = \sup_{\gamma \in \dom(L^*)} \delta^*_{\dom(\Phi)}(\gamma\theta) - L^*(\gamma).
    \end{align*}
    The desired reformulation follows if we can show that $\delta^*_{\dom(\Phi)}(\gamma\theta) = 0 \Phi^*(\gamma \theta / 0) = 0 c^{*1}(\gamma \theta / 0, \hat z)$. To this end, note first that~$0 \in \dom(\Phi^*)$ because~$\Phi$ is non-negative. Thus, we have
    \begin{align*}
        \delta^*_{\dom(\Phi)}(\gamma\theta)& = \sup_{\Delta \in \R^d} \gamma \inner{\theta}{\Delta} - \inf_{\alpha > 0} \Phi(\Delta)/\alpha \;=\; \sup_{\alpha > 0} \sup_{\Delta \in \R^d} \gamma \inner{\theta}{\Delta} - \Phi(\Delta)/\alpha\\
        & = \lim_{\alpha \to +\infty} \alpha^{-1} \sup_{\Delta \in \R^d} \alpha \gamma \inner{\theta}{\Delta} - \Phi(\Delta) \;=\; \lim_{\alpha \to +\infty} \Phi^*(\alpha \gamma \theta)/\alpha \\
        & =\; 0 \Phi^*(\gamma \theta/0) \;=\; 0c^{*1}(\gamma \theta/0, \hat z),
    \end{align*}
    where the first equality exploits the definition of the support function and the non-negativity of~$\Phi$, which implies that~$\inf_{\alpha > 0} \Phi(\Delta)/\alpha=\delta_{\dom(\Phi)}$, whereas the second, third and fourth equalities follow from elementary reformulations and from the definition of the conjugate. The fifth equality holds thanks to~\cite[Corollary~8.5.2 \& Theorem~13.3]{rockafellar1970convex}. As~$0\in\dom(\Phi^*)$, these results ensure that the recession function~$\lim_{\alpha \to +\infty} \Phi^*(\alpha \gamma \theta)/\alpha$ of~$\Phi^*$ coincides with the support function~$\delta^*_{\dom(\Phi^{**})}(\gamma\theta)$ of the domain of~$\Phi^{**}$, which in turn is our definition of~$0\Phi^*(\gamma \theta/0)$. The last equality follows from the definition of~$\Phi$.
    
    It remains to be shown that $\ell_c(\theta,\lambda, \hat z)$ is convex in $(\theta, \lambda)$. This is  evident from our dual reformulation of $\ell_c(\theta,\lambda, \hat z)$, however, because convexity is preserved by maximization. Hence, assertion~\ref{theorem:toland} follows.
    
    As for assertion~\ref{theorem:norm}, define the auxiliary function $\Psi(z) = \| z \|^p$ for $z\in\R^d$, and recall that $L$ is an arbitrary (possibly nonconvex) proper function. In this case, the $c$-transform can be reformulated as
    \begin{align*}
        \ell_c(\theta,\lambda, \hat z)&= \sup_{z \in \R^d}\; L(\inner{\theta}{z}) - \lambda \Psi(z - \hat z)\\
        &= \sup_{\gamma \in \R}\; \sup_{\Delta \in \R^d} \big\{ L(\inner{\theta}{\hat z}  + \gamma) - \lambda \Psi(\Delta): \inner{\theta}{\Delta} = \gamma \big\} \\
        &= \sup_{\gamma \in \R}\; \inf_{\beta \in \R}\; \sup_{\Delta \in \R^d}\;  L(\inner{\theta}{\hat z}  + \gamma) - \lambda \Psi(\Delta) + \beta \inner{\theta}{\Delta} - \beta \gamma.
    \end{align*}
    In order to justify the third equality, consider the two maximization problems over~$\gamma$ in the second and the third lines of the above expression. If~$\theta=0$, then~$\gamma=0$ is the only feasible solution in both problems, and one readily verifies that the corresponding objective function values coincide. If~$\theta\neq 0$, on the other hand, then the objective functions of the two problems coincide for every fixed~$\gamma\in\R^d$ due to strong Lagrangian duality \citep[Proposition~5.3.1]{bertsekas2009convex}, which applies because the primal maximization problem over~$\Delta$ in the second line is convex and feasible and because every feasible solution constitutes a Slater point. Recalling the definitions of conjugates and perspectives then yields
    \begin{align}
        \ell_c(\theta,\lambda, \hat z)
        &= \sup_{\gamma \in \R}\; \inf_{\beta \in \R}\; L(\inner{\theta}{\hat z} + \gamma) + \Big(\displaystyle \sup_{\Delta \in \R^d}\; \beta \inner{\theta}{\Delta} - \lambda \Psi(\Delta) \Big) - \beta \gamma \notag\\
        &= \sup_{\gamma \in \R}\; \inf_{\beta \in \R}\; L(\inner{\theta}{\hat z} + \gamma) + \lambda \Psi^*(\beta \theta / \lambda) - \beta \gamma . \label{eq:aux}
    \end{align}
    To simplify~\eqref{eq:aux}, we assume first that $p=1$ (Case~I) and later that~$p>1$ (Case~II).
    
    \textbf{Case I} ($p=1$){\bf :} As $\Psi(\cdot) = \| \cdot \|$ is a norm, one can show that $\Psi^*(\cdot)=\delta_C(\cdot)$ is the indicator function of the dual norm ball $C = \{ \theta \in \R^d: \| \theta \|_* \leq 1 \}$. If~$\lambda>0$ and~$\theta\neq 0$, we can thus re-express~\eqref{eq:aux} as
    \begin{align*}
    \sup_{\gamma \in \R}\; \inf_{\beta \in \R}\; L(\inner{\theta}{\hat z} + \gamma) + \lambda \delta_C(\beta \theta / \lambda) - \beta \gamma
    &= \sup_{\gamma \in \R}\; \inf_{\beta \in \R}\; \left\{ L(\inner{\theta}{\hat z}  + \gamma) - \beta \gamma: \| \beta \theta \|_* \leq \lambda \right\} \\
    & = \sup_{\gamma \in \R} \; L(\inner{\theta}{\hat z}  + \gamma) - \lambda |\gamma|/\| \theta \|_* \\
    &= \sup_{\gamma \in \R}\; L(\inner{\theta}{\hat z} + \gamma \| \theta \|_* ) - \lambda |\gamma|,
    \end{align*}
    where the second equality holds because the constraint~$\|\beta\theta\|_*\leq \lambda$ is equivalent to~$|\beta|\leq \lambda/\|\theta\|_*$, while the third equality follows form the variable substitution~$\gamma \leftarrow \gamma/\|\theta\|_*$. If $\lambda >0$ and $\theta = 0$, on the other hand, then $\lambda\Psi^*(\beta\theta/\lambda)=\lambda\Psi^*(0)=0$, and thus we can re-express~\eqref{eq:aux} as
    \begin{align*}
    \sup_{\gamma \in \R}\; \inf_{\beta \in \R}\; \beta \gamma - L( \gamma) 
    = -L(0) = \sup_{\gamma \in \R}\; L(\inner{\theta}{z}  + \gamma \| \theta \|_*) - \lambda |\gamma| .
    \end{align*}
    Here, the first equality holds because~$\gamma = 0$ is the only feasible solution of the maximization problem on the left hand side, and the second equality holds because~$\theta=0$, which implies that the maxization problem on the right hand side is also solved by~$\gamma=0$. If $\lambda=0$ and $\theta\neq 0$, then $\lambda\Psi^*(\beta\theta/\lambda)$ simplifies to $0\Psi^*(\beta\theta/0)=\delta^*_{\R^d}(\beta\theta) = \delta_{\{0\}}(\beta\theta)$
    thanks to our conventions for perspective functions, and~\eqref{eq:aux} becomes
    \begin{align*}
        \sup_{\gamma \in \R}\; L(\inner{\theta}{\hat z} + \gamma) 
        = \sup_{\gamma \in \R}\; L(\inner{\theta}{\hat z} + \gamma \| \theta \|_*)
           = \sup_{\gamma \in \R}\; L(\inner{\theta}{\hat z} + \gamma \| \theta \|_*) - \lambda |\gamma|,
    \end{align*}
    where the first equality exploits the substitution~$\gamma \leftarrow \gamma/\|\theta\|_*$, and the second equality holds because $\lambda=0$. If $\lambda=0$ and $\theta= 0$, finally, then a similar reasoning reveals that~\eqref{eq:aux} reduces to~$L(\gamma)= \sup_{\gamma \in \R}\; L(\inner{\theta}{\hat z} + \gamma \| \theta \|_*) - \lambda |\gamma|$. This completes the proof of assertion~\ref{theorem:norm} for~$p=1$.
    
    \textbf{Case II} ($p>1$){\bf :} By~\cite[Lemma~C.9\,(ii)]{zhen2023unified}, the conjugate of~$\Psi(\cdot)=\|\cdot\|^p$ is given by~$\Psi^*(\cdot) = h(q) \| \cdot \|_*^q$, where~$h(p) = (p - 1)^{p-1} / p^p$ and~$q>0$ is the H\"{o}lder conjugate of~$p$. 
    If~$\lambda>0$ and~$\theta\neq 0$, \eqref{eq:aux} reduces to
    \begin{align*}
    \sup_{\gamma \in \R}\; \inf_{\beta \in \R}\; L(\inner{\theta}{\hat z} + \gamma) + h(q) \lambda^{1-q} \| \theta \|_*^q |\beta|^q - \beta \gamma 
        &= \sup_{\gamma \in \R}\; L(\inner{\theta}{\hat z}  + \gamma) - \lambda \left| \gamma/\| \theta \|_* \right|^p  \\
    &= \sup_{\gamma \in \R}\; L(\inner{\theta}{\hat z}  + \gamma \| \theta \|_*) - \lambda |\gamma|^p,
    \end{align*}
    where the two equalities follow from a tedious analytical calculation and form the substitution~$\gamma \leftarrow \gamma/\|\theta\|_*$, respectively. If $\lambda =0$ or $\theta = 0$, then we can proceed exactly as in Case~I to prove that~\eqref{eq:aux} is equivalent to~$\sup_{\gamma \in \R}\; L(\inner{\theta}{\hat z} + \gamma \| \theta \|_*) - \lambda |\gamma|^p$. This completes the proof of assertion~\ref{theorem:norm}.
\end{proof}

Theorem~\ref{theorem:nonconvex:duality} shows that, in the context of linear prediction models, evaluating the $c$-transform~$\ell_c$ is quite generally equivalent to solving a univariate maximization problem. If this maximization problem is easy to solve, then the stochastic program~\eqref{eq:stochastic-program} equivalent to the DRO problem~\eqref{eq:dro:linear} can be addressed with stochastic gradient descent-type algorithms even if the reference distribution fails to be discrete and even if the convexity assumptions of Section~\ref{section:nash:computation} fail to hold. Specifically, if the univariate maximization problem that evaluates the $c$-transform~$\ell_c$ can be solved analytically, then one may use the envelope theorem~\citep[Theorem~2.16]{de2000mathematical} to generate stochastic subgradients for the objective function of~\eqref{eq:stochastic-program}. In this case, the stochastic program~\eqref{eq:stochastic-program} is amenable to standard (projected) stochastic subgradient methods; see, {\em e.g.}, \cite{polyak1987intro}. If the $c$-transform evaluation problem can only be solved approximately via a numerical procedure such as a grid-search method, a bisection algorithm or a sequential convex optimization approach, then the envelope theorem only provides {\em inexact} stochastic subgradients. In this case, and provided that $L$ is convex, one can use the stochastic subgradient descent algorithms with {\em inexact} gradient oracles studied in \citep[\textsection~3]{blanchet2018optimal} and \citep[\textsection~4]{tacskesen2023semi} to solve problem~\eqref{eq:stochastic-program}. As an example, in Appendix~\ref{section:convexity} we identify conditions under which the univariate $c$-transform evaluation problem purported in Theorem~\ref{theorem:nonconvex:duality}\,\ref{theorem:norm} is a convex maximization problem even if~$L$ is convex instead of concave.

The $c$-transform $\ell_c(\theta,\lambda, \hat z)$ defined in~\eqref{eq:robust:loss} can be viewed as an approximation of~$\ell(\theta, \hat z)=L(\inner{\theta}{\hat z})$, where the parameter $\lambda$ controls the approximation error.
If $c(z, \hat z) = \| z - \hat z \|^p$ for some $p\geq 1$, then the $c$-transform is closely related to the $p$-th envelope $L_{p}:\R\times \R_+ \to [-\infty, +\infty]$ of~$L$, which is defined via
\begin{align*}
    L_p(s, \lambda) = \sup_{s' \in \R}\; L(s')-\lambda | s-s' |^p.
\end{align*}
The next proposition uses Theorem~\ref{theorem:nonconvex:duality}\,\ref{theorem:norm} to formalize this relation and specializes the strong duality result of Proposition~\ref{proposition:strong:duality} to DRO problems of the form~\eqref{eq:dro:linear} with norm-based transportation costs.

\begin{proposition}[Strong duality for DRO problems of the form~\eqref{eq:dro:linear}] \label{proposition:envelope}
    Assume that $\cZ=\R^d$, that $c(z, \hat z) = \| z - \hat z \|^p$, where~$p\geq 1$ and $\|\cdot\|$ is an arbitrary norm, and that the univariate loss function~$L$ is proper and upper semicontinuous and satisfies $L(s) \leq C (1 + |s|^p)$ for some~$C\ge 0$ and all~$s\in\R$. If additionally $\ell(\theta,z)=L(\inner{\theta}{z})$ satisfies Assumption~\ref{assumption:continuity}, then we have
    \begin{align}
    \label{eq:envelope:reg}
        \sup_{\Q \in \B_{\eps}(\hat{\P})} \E_{Z \sim \Q} \left[ L(\inner{\theta}{Z}) \right]
        = \inf_{\lambda \geq 0} \lambda \eps \| \theta \|_*^p + \E_{\hat{Z} \sim \hat{\P}} \left[L_p ( \inner{\theta}{\hat{Z}} , \lambda ) \right].
    \end{align}
\end{proposition}
\begin{proof}
    If~$\theta=0$, then the left hand side of~\eqref{eq:envelope:reg} equals $L(0)$, and the right hand side of~\eqref{eq:envelope:reg} evaluates to
    \[
        \inf_{\lambda \geq 0} L_p ( 0 , \lambda ) = \lim_{\lambda\to +\infty} L_p(0,\lambda) = L(0) ,
    \]
    where the first equality holds because $L_p (s , \lambda )$ is non-increasing in $\lambda$ for any fixed~$s$, whereas the second equality follows from Lemma~\ref{lemma:conv:envelope}. 
    This establishes the claim for $\theta=0$. In the remainder of the proof we may thus assume that~$\theta\neq 0$. By Proposition~\ref{proposition:strong:duality}, which applies thanks to Assumption~\ref{assumption:continuity}, the worst-case expected loss on the left hand side of~\eqref{eq:envelope:reg} matches the optimal value of the stochastic program in~\eqref{eq:inner:supremum}, which involves the $c$-transform \eqref{eq:robust:loss}. By Theorem~\ref{theorem:nonconvex:duality}\,\ref{theorem:norm}, we further have
    \begin{align*}
        \ell_c(\theta, \lambda, \hat z) 
        &= \sup_{\gamma \in \R}\;  L(\inner{\theta}{\hat z} + \gamma \| \theta \|_*) - \lambda | \gamma |^p = \sup_{\gamma \in \R}\; L(\gamma) - \frac{\lambda}{\| \theta \|_*^p} | \gamma - \inner{\theta}{\hat z} |^p = L_p \left( \inner{\theta}{\hat z} , {\lambda}/{\| \theta \|_*^p} \right),
    \end{align*}
    where the second and the third equalities follow the variable substitution $\gamma \gets \inner{\theta}{\hat z}+\gamma \| \theta \|_*$ and the definition of~$L_p$, respectively. Substituting this expression for the $c$-transform into~\eqref{eq:inner:supremum} finally yields
    \[
        \sup_{\Q \in \B_{\eps}(\hat{\P})} \E_{Z \sim {\Q}} \left[ L(\inner{\theta}{Z}) \right] = \inf_{\lambda \geq 0} \lambda \eps + \E_{\hat Z \sim \hat{\P}} \left[ L_p \left( \inner{\theta}{\hat z} , {\lambda}/{\| \theta \|_*^p} \right) \right] = \inf_{\lambda \geq 0} \lambda \eps \| \theta \|_*^p + \E_{\hat{Z} \sim \hat{\P}} \left[L_p ( \inner{\theta}{\hat{Z}} , \lambda ) \right],
    \]
    where the second equality follows from the variable substitution $\lambda \gets \lambda/\| \theta \|_*^p$. This establishes the claim for $\theta\neq 0$. In summary, we have thus shown that~\eqref{eq:envelope:reg} holds indeed for every $\theta\in\R^d$.
\end{proof}

\subsubsection{Pasch-Hausdorff Envelope}
Throughout this section we assume that all conditions of Proposition~\ref{proposition:envelope} hold and that $p=1$. In this case, $L_1(\cdot,\lambda)$ is usually called the
Pasch-Hausdorff envelope of $L$.
One can show that $L_1(\cdot, \lambda)$ coincides with the smallest $\lambda$-Lipschitz continuous majorant of $L$ or evaluates to $+\infty$ if no such majorant exists \citep[Proposition~12.17]{bauschke2011convex}. Under the conditions of this section, the univariate minimizaton problem on the right hand side of~\eqref{eq:envelope:reg} can sometimes be solved analytically. This allows us to recover---in a unified and simplified manner---several reformulation results from the recent  literature on Wasserstein DRO. 

\begin{example}[Aymptotically steep Lipschitz continuous loss]
\label{ex:steep}
    If the asymptotic linear growth rate of the loss function~$L$, which is defined as $\limsup_{|s|\to \infty} L(s)/|s|$, coincides with the Lipschitz modulus of~$L$, then the Pasch-Hausdorff envelope of~$L$ satisfies $L_1 (s, \lambda)=L(s)$ if $\lambda \geq \lip(L)$ and $L_1(s, \lambda)=+\infty$ otherwise. To see this, assume first that  $\lambda \geq \lip(L)$. In that case, we have 
    \begin{align*}
        L(\inner{\theta}{\hat z})
        &\leq \sup_{\gamma \in \R}\; L( \gamma ) - \lambda | \gamma - \inner{\theta}{\hat z} | \leq \sup_{\gamma \in \R}\; L\big( \inner{\theta}{\hat z} \big) - \big( \lambda - \lip(L) \big) |\gamma - \inner{\theta}{\hat z}|
        = L\big( \inner{\theta}{\hat z} \big) ,
    \end{align*}    
    where the second inequality holds because the loss function~$L$ is Lipschitz continuous with Lipschitz modulus $\lip(L)$. If $\lambda < \lip(L)$, on the other hand, then we have
    \begin{align*}
        L_1(\inner{\theta}{\hat z}, \lambda) &\geq \limsup_{|\gamma| \to \infty}\; L( \gamma ) - \lambda | \gamma - \inner{\theta}{\hat z} |  \\
        &= L(\inner{\theta}{\hat z}) + \limsup_{|\gamma| \to \infty} |\gamma - \inner{\theta}{\hat z}| \left(\frac{L(\gamma) - L(\inner{\theta}{\hat z})}{|\gamma - \inner{\theta}{\hat z}| } - \lambda \right) \\
        &= L(\inner{\theta}{\hat z}) + \limsup_{|\gamma| \to \infty} |\gamma - \inner{\theta}{\hat z}| \left(\lip(L) - \lambda \right) = +\infty,
    \end{align*}
    where the second inequality holds because the asymptotic linear growth rate of~$L$ coincides with~$\lip(L)$. Hence, the objective function in the minimization problem on the right hand side of~\eqref{eq:envelope:reg} evaluates to~$+\infty$ whenever~$\lambda < \lip(L)$. As $\eps\|\theta\|_*\geq 0$, we thus have~$\lambda=\lip(L)$ at optimality, and~\eqref{eq:envelope:reg} simplifies~to
    \begin{align}
        \label{eq:exact-Lipschitz-regularization}
        \sup_{\Q \in \B_{\eps}(\hat{\P})} \E_{Z \sim \Q} \left[ L(\inner{\theta}{Z}) \right]
        = \E_{\hat Z \sim \hat{\P}}[ L (\inner{\theta}{\hat Z})] + \eps \lip(L) \| \theta \|_*.
    \end{align}
    Hence, the worst-case expected loss over a 1-Wasserstein ball of radius~$\eps$ coincides exactly with the expected loss under the reference distribution adjusted by the regularization term $\eps \lip(L) \| \theta \|_*$. This result strengthens the inequality derived in Corollary~\ref{corollary:Wasserstein-p:regularization} for~$p=1$ to an equality. It was first discovered in the context of distributionally robust linear regression~\citep{shafieezadeh2019regularization}, where the random vector $Z = (X, Y) \in \R^d$ consists of a multi-dimensional input $X \in \R^{d-1}$ and a scalar output $Y \in \R$, the decision~$\theta=(\theta_{x},\theta_y)\in\R^d$ consists of the weight vector $\theta_{x}\in\R^{d-1}$ of a linear predictor and the constant~$\theta_y=-1$, while the prediction loss $L(\inner{\theta}{Z})=L(\inner{\theta_{x}}{X}-Y)$ is determined by a Lipschitz continuous {\em convex} function~$L$. Note that convexity is a simple sufficient condition for the asymptotic linear growth rate of~$L$ to coincide with~$\lip(L)$. In the context of distributionally robust linear classification, where the output~$Y$ is restricted to~$+1$ or~$-1$ and the prediction loss is given by~$L(Y\inner{\theta}{X})$, it has further been shown that
    \begin{align*}
        \sup_{\Q \in \B_{\eps}(\hat{\P})} \E_{(X,Y) \sim \Q} \left[ L(Y \inner{\theta}{X}) \right] 
        = \E_{(\hat X, \hat Y) \sim \hat{\P}} \left[ L(\hat Y \inner{\theta}{\hat X}) \right] + \eps \lip(L) \| \theta \|_*
    \end{align*}
    whenever the transportation cost function satisfies $c((x, y), (\hat x, \hat y)) = \| x - \hat x \|$ if $y=\hat y$; $=+\infty$ if~$y\neq\hat y$, where~$\|\cdot\|$ is an arbitrary norm on the input space \citep{shafieezadeh2019regularization, shafieezadeh2015distributionally}. In this case, the output~$Y$ has the same marginal under every distribution $\Q \in \B_{\eps}(\hat{\P})$ as under the reference distribution $\hat{\P}$. The above identity can be derived by repeating the arguments that led to~\eqref{eq:exact-Lipschitz-regularization} with obvious minor modifications. Details are omitted for brevity. The general identity~\eqref{eq:exact-Lipschitz-regularization} for {\em nonconvex} loss functions whose asymptotic linear growth rate coincides with their Lipschitz modulus was first established in~\citep[Corollary~2]{gao2024wasserstein}.
\end{example}

\begin{example}[Zero-one loss]
    \label{example:classification}
    One readily verifies that the Pasch-Hausdorff envelope of the zero-one loss $L(s)= \mathds 1_{\{s \leq 0 \}}$ satisfies $L_1 (s, \lambda) = \max \{0, 1 - \max\{0, \lambda s \}\}$. By Proposition~\ref{proposition:envelope}, we thus have
    \begin{align*}
        \sup_{\Q \in \B_{\eps}(\hat{\P})} \Q \left( \inner{\theta}{Z}\leq 0 \right) 
        & = \sup_{\Q \in \B_{\eps}(\hat{\P})} \E_{Z \sim \Q} \left[ L(\inner{\theta}{Z}) \right] \\
        &= \min_{\lambda \geq 0}\; \lambda \eps \| \theta \|_* + \E_{\hat Z \sim \hat{\P}} \left[ \max \{0, 1 - \max\{0, \inner{\lambda \theta}{\hat Z} \}\} \right] .
    \end{align*}
    If $\Theta$ is a cone, then the scaling factor~$\lambda$ can be eliminated by using the variable substitution $\theta' \gets \lambda \theta$. In this case, the DRO problem reduces a stochastic program under the reference distribution, that is,
    \begin{align*}
        \inf_{\theta\in\Theta}\sup_{\Q \in \B_{\eps}(\hat{\P})} \Q \left( \inner{\theta}{Z}\leq 0 \right) 
        = \inf_{\theta'\in\Theta} \eps \| \theta' \|_* + \E_{\hat Z \sim \hat \P} \left[ \max \{0, 1 - \max\{0, \inner{\theta'}{\hat Z} \}\} \right].
    \end{align*}
    The above identity was first derived in the context of distributionally robust linear classification \citep{ho2023adversarial} using recent results on ambiguous chance constraints and their relation to conditional value-at-risk constraints \citep{xie2019distributionally}. Our derivation based on the Pasch-Hausdorff envelope is new and shorter.
\end{example}

\subsubsection{Moreau Envelope}
Throughout this section we assume that all conditions of Proposition~\ref{proposition:envelope} hold and that $p=2$. In this case, $L_2(s,\lambda)$ is termed the Moreau envelope of~$L$. Under the conditions of this section, the univariate minimizaton problem on the right hand side of~\eqref{eq:envelope:reg} can be solved analytically in interesting special cases, which leads again to simplified derivations of existing reformulation results. 

\begin{example}[Quadratic loss]
    \label{example:square:loss}
    The Moreau envelope of the quadratic loss function $L(s) =  s^2$ satisfies $L_2 (s, \lambda) =\lambda s^2/(\lambda - 1)$ if $\lambda > 1$, $L_2 (s, \lambda) =+\infty\cdot \mathds{1}_{\{s\neq 0\}}$ if $\lambda=1$, and $L_2 (s, \lambda) =+\infty$ if $0 \leq \lambda \leq 1$. If we assume that $\hat{\P}(\inner{\theta}{\hat Z}\neq 0)>0$ to rule out trivialities, then Proposition~\ref{proposition:envelope} implies that
    \begin{align*}
        \sup_{\Q \in \B_{\eps}(\hat{\P})} \E_{Z \sim \Q} \left[ L(\inner{\theta}{Z}) \right] = \min_{\lambda >1 }\; \lambda \eps \| \theta \|_*^2 + \frac{\lambda}{\lambda - 1} \E_{\hat Z \sim \hat{\P}} \left[ (\inner{\theta}{\hat Z})^2 \right] = \left( \sqrt{\E_{\hat Z \sim \hat{\P}} \left[ (\inner{\theta}{\hat Z})^2 \right]} + \sqrt{\eps} \| \theta \|_* \right)^2.
    \end{align*}
    The second equality in the above expression holds because the minimization problem over~$\lambda$ is solved analytically by $\lambda^\star = 1 + \sqrt{\eps} \| \theta \|_* (\E_{\hat Z \sim \hat{\P}} [ (\inner{\theta}{\hat Z})^2 ])^{-1/2}$. This identity was first discovered in the context of distributionally robust least squares regression; see \citep[Proposition~2]{blanchet2019robust}.
\end{example}

\begin{example}[Zero-one loss]
    \label{example:binary:loss:2}
    The Moreau envelope of the zero-one loss $L(s) =  \mathds{1}_{\{s \leq 0\}}$ is given by
    $L_2 (s, \lambda) 
    = \max\{0, 1 - \sqrt{\lambda} s \max\{0, \sqrt{\lambda} s \}\}$. If $\Theta$ is a cone, then, in analogy to Example~\ref{example:classification}, we can use Proposition~\ref{proposition:envelope} and the variable substitution $\theta' \gets \sqrt{\lambda} \theta$ to conclude that
    \begin{align*}
       \inf_{\theta\in\Theta} \sup_{\Q \in \B_{\eps}(\hat{\P})} \Q \left( \inner{\theta}{Z}\leq 0 \right) 
       &= \inf_{\theta'\in\Theta} \eps \| \theta' \|_*^2 + \E_{\hat Z \sim \hat{\P}} \left[ \max \left\{0, 1 - \inner{\theta'}{\hat Z} \max\{0, \inner{\theta'}{\hat Z} \} \right\} \right].
    \end{align*}
\end{example}

\begin{remark}
    \label{rem:theta-is-a-matrix}
     If the loss function in~\eqref{eq:dro:linear} is representable as $L(\theta^\top z)$ for some matrix~$\theta\in\R^{d\times k}$ and some proper, convex and lower semicontinuous function $L: \R^k \to (-\infty, +\infty]$, while $c$ is proper and lower semicontinuous, then one can proceed as in the proof of Theorem~3.11\,(i) to show that 
    $$
        \ell_c(\theta,\lambda, \hat z)=\sup_{\gamma \in \dom(L^*)} \; \lambda c^{*1}(\theta \gamma / \lambda, \hat z) - L^*(\gamma).
    $$
    In this case, $\gamma$ is no longer a scalar but ranges over~$\R^k$. Even though nonconvex, the resulting maximization problem may still be much easier to solve than~\eqref{eq:robust:loss} when $k\ll d$. That is, we have reformulated a high-dimensional nonconvex problem as a low-dimensional nonconvex problem. Loss functions of the type considered here arise in multi-output regression and classification.
\end{remark}

\section{Numerical Experiments}
\label{section:numerical}
All linear and second-order cone programs in our experiments are implemented in Python and solved with Gurobi~10.0.0 on a \mbox{$ 2.4 $ GHz} quad-core machine with \mbox{$ 8 $ GB RAM}. To ensure reproducibility, all source codes are made available at \url{https://github.com/sorooshafiee/regularization_with_OT}.

\subsection{Nash Equilibria}
\label{sec:nash-experiments}
We first illustrate the computation of Nash equilibria between a statistician and nature in the context of a distributionally robust support vector machine problem. We thus assume that $Z=(X,Y) \in \R^d$ consists of a feature vector $X \in \cX\subseteq \R^{d-1}$ and a label $Y \in \cY = \{-1, +1\}$. In addition, $\theta \in \Theta = \R^{d-1}$ is the weight vector of a linear classifier, and~$\ell(\theta, Z) = \max\{ 0, 1 - Y \inner{\theta}{X}\}$ is the hinge loss function, which represents the pointwise maximum of $I=2$ saddle functions. Finally, the reference distribution~$\hat \P$ is set to the empirical (uniform) distribution on~$J$ training samples $\hat z_j=(\hat x_j,\hat y_j)$, $j\in[J]$, and the transportation cost function is defined as $c((x_1,y_1),(x_2,y_2)) = \| x_1 - x_2 \|$ if $y_1=y_2$, where $\|\cdot\|$ is an arbitrary norm on $\R^{d-1}$, and  $c((x_1,y_1),(x_2,y_2)) =+\infty$ if $y_1\neq y_2$. In this setting, the Nash equilibrium between the statistician and nature can be computed by using the techniques that were developed in Section~\ref{section:nash:computation} and further generalized in Appendix~\ref{appendix:nash:ml}. Indeed, one readily verifies that Assumptions~\ref{assumption:nash:reference}, \ref{assumption:nash:convexity:ml} and \ref{assumption:slater:ml} hold. In addition, as we will see below, one can also show that Assumption~\ref{assumption:nash:regularity2} is satisfied. This implies that Nash strategies for the statistician and nature can be computed by solving the finite convex programs~\eqref{eq:primal-revisited} and~\eqref{eq:dual-revisited}, which generalize problems~\eqref{eq:primal} and~\eqref{eq:dual}, respectively; see Theorem~\ref{theorem:nash:computation:ml}. Specifically, under the given assumptions about the loss and transportation cost functions as well as the reference distribution, one can show that if $\cX=\R^{d-1}$, then~\eqref{eq:primal-revisited} simplifies to
\begin{align}
        \label{eq:reg:svm}
        \begin{array}{cll}
        \min & \DS \eps \lambda + \frac{1}{J} \sum_{j \in [J]} s_j\\
        \st & \theta\in\R^{d-1},~ \lambda, s_j\in\R & \forall j\in[J]\\
        & \lambda \geq \| \theta \|_*,~ s_j \geq 0, ~ s_j \geq 1 - \hat{y}_j \inner{\theta}{\hat{x}_j} & \forall j\in[J],
        \end{array}
\end{align}
while problem~\eqref{eq:dual-revisited} reduces to
\begin{align}
    \label{eq:dual:svm}
    \begin{array}{cll}
        \max & \DS \sum_{j \in [J]} q_{2j} \\[0.5ex]
        \st & \DS q_{ij} \in \R_+, ~ \xi_{ij} \in \R^{d-1} & \forall i \in [I], \, j \in [J] \\[0.5ex]
        & \DS \sum_{i \in [I]} q_{ij} = 1/J & \forall j \in [J] \\[0.5ex]
        & \DS \sum_{i \in [I]} \sum_{j \in [J]} \left\| \xi_{ij} \right\| \leq \eps \\[0.5ex]
        & \DS \sum_{j \in [J]} \hat{y}_j (q_{2j} \hat{x}_j + \xi_{2j}) = 0.
    \end{array}
\end{align}
The following proposition shows that distributionally robust support vector machine problem under consideration admits a continuum of least favorable distributions, which represent different Nash strategies of nature. This implies that there is in fact a continuum of many different Nash equilibria.

\begin{proposition}[Non-uniqueness of Nash equilibria]
    \label{prop:non-uniqueness}
    If~$(\{\tilde q^\star_j\}_{j})$ solves the maximization problem
    \begin{align}
    \label{eq:dual:svm:light}
    \max_{\tilde q_j,\,j\in[J]} \left\{ \sum_{j \in [J]} \tilde q_j: \left\| \sum_{j \in [J]} \tilde q_{j} \hat{y}_j \hat{x}_j \right\| \leq \eps, \, 0 \leq \tilde q_j \leq \frac{1}{J} ~ \forall j \in [J] \right\},
    \end{align}
    and if $\cJ_0 = \{j \in [J]: \tilde q^\star_{j} = 0 \}$, $\cJ_+ = \{j \in [J]: \tilde q^\star_{j} > 0 \}$ and
    $\tilde \xi = - \sum_{j \in [J]} \tilde q^\star_{j} \hat{y}_j \hat{x}_j$, then the distribution
    \begin{align*}
        \Q^\star(\alpha)=\frac{1}{J} \sum_{j \in \cJ_0} \delta_{(\hspace{0.1em}\hat{x}_j, \hat{y}_j)} 
        + \sum_{j \in \cJ_+} \left( \frac{1}{J} - \tilde q^\star_j \right) \delta_{(\hspace{0.1em}\hat{x}_j, \hat{y}_j)}
        + \sum_{j \in \cJ_+} \tilde q^\star_j \delta_{(\hspace{0.1em} \hat{x}_j + \alpha_j \hat{y}_j \tilde \xi / \tilde q_j^\star, \hat{y}_j)}
    \end{align*}
    constitutes a Nash strategy of nature for any $\alpha \in A = \{ \alpha \in \R_+^J: \sum_{j \in [J]} \alpha_j = 1, ~ \alpha_j = 0 ~\forall j \in \cJ_0 \}$. 
\end{proposition}

Consider now an instance of the distributionally robust support vector machine problem with~$d=3$, $\eps = 0.1$ and~$J=20$, where $\hat y_j=1$ and $\hat x_j$ is sampled independently from $\cN((1,-1), \frac{1}{2} I_2)$ for $j=1,\ldots,10$, while $\hat y_j=-1$ and $\hat x_j$ is sampled independently from $\cN((-1,1), \frac{1}{2} I_2)$ for $j=11,\ldots,20$. In addition, we use the $p$-norm to quantify the transportation cost in the feature space for $p\in\{1,2,\infty\}$. We then solve problem~\eqref{eq:reg:svm} to find an optimal weight vector~$\theta^\star$, and we solve problem~\eqref{eq:dual:svm:light} to construct the index sets~$\cJ_0$ and $\cJ_+$, the transportation budget~$\tilde\xi$ as well as different least favorable distributions~$\Q^\star(\alpha)$ with~$\alpha \in\cA$.
Figure~\ref{figure:nash} represents all training samples with~$\hat y_j=1$ and~$\hat y_j=-1$ as ``+'' and red ``--'' markers, respectively. The optimal classifier is visualized by the separating hyperplane~$\{x\in\R^2:\inner{\theta^\star}{x}=0\}$ (solid line) and the maximum margin hyperplanes~$\{x\in\R^2:\inner{\theta^\star}{x}=\pm 1\}$ (dashed lines). By Proposition~\ref{prop:non-uniqueness}, there are infinitely many least favorable distributions, which are obtained from the empirical distribution by moving probability mass from the samples in~$\cJ_+$ in a direction of increasing hinge loss at a total transportation budget~$\tilde \xi$. The left charts of Figure~\ref{figure:nash} show the empirical distribution, where all training samples in~$\cJ_+ $ are designated by a solid circle. The middle charts show a least favorable distribution obtained by assigning the transportation budget~$\tilde \xi$ to a single training sample in~$\cJ_+$, and the right charts show the least favorable distribution obtained by evenly distributing the transportation budget~$\tilde \xi$ across all training samples in~$\cJ_+$.

\begin{figure}[t]
    \centering
    \includegraphics[width=0.8\columnwidth]{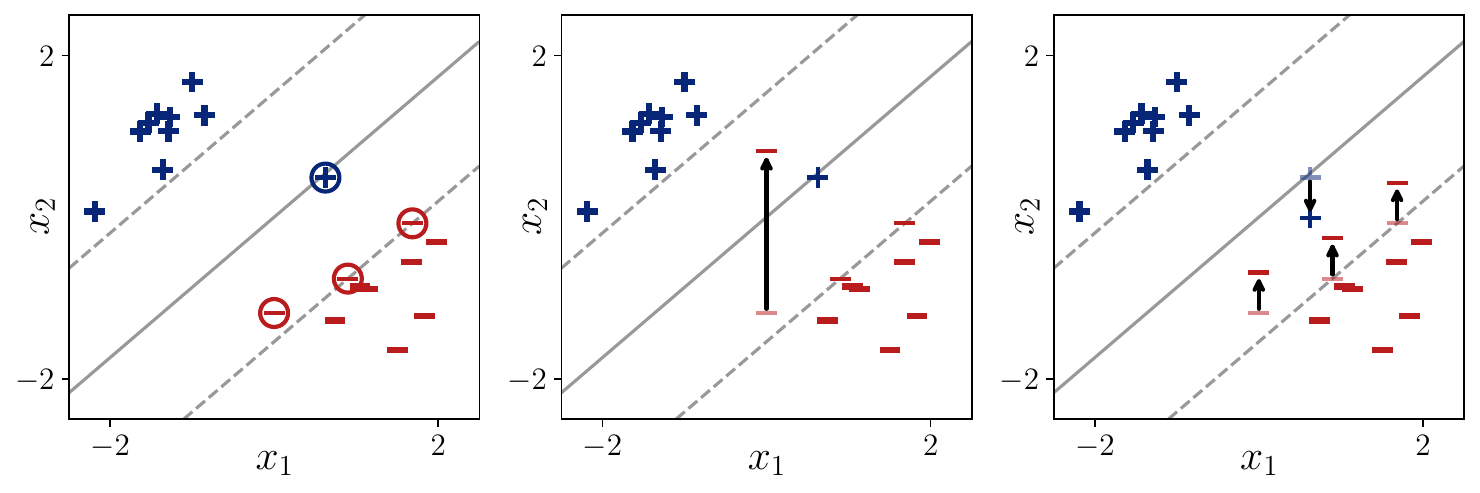} \\
    \includegraphics[width=0.8\columnwidth]{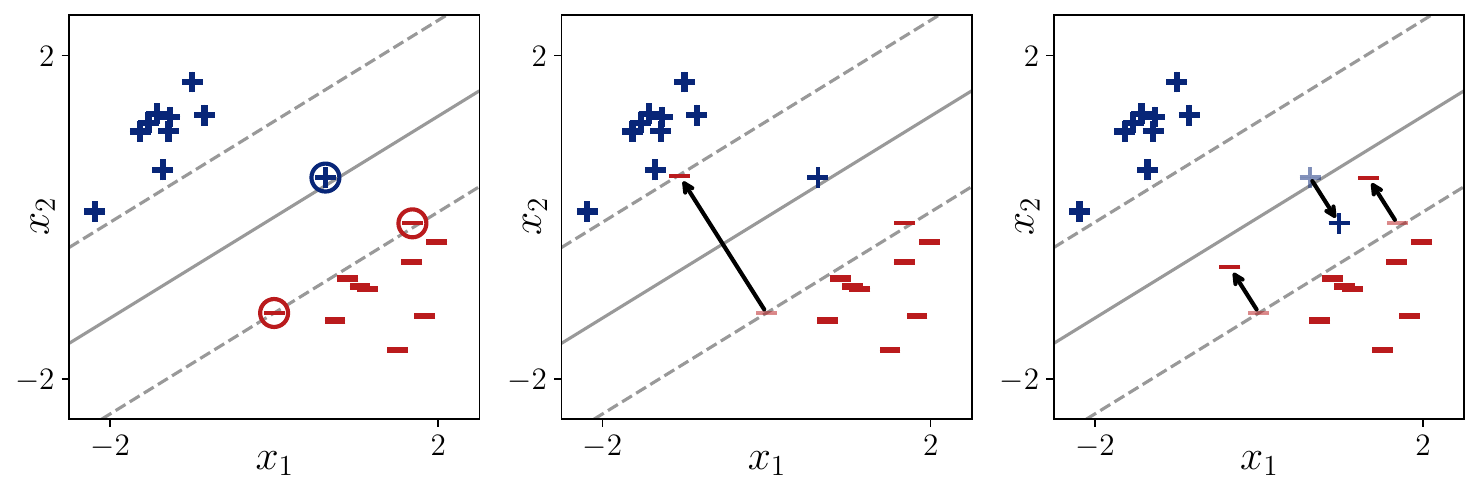} \\
    \includegraphics[width=0.8\columnwidth]{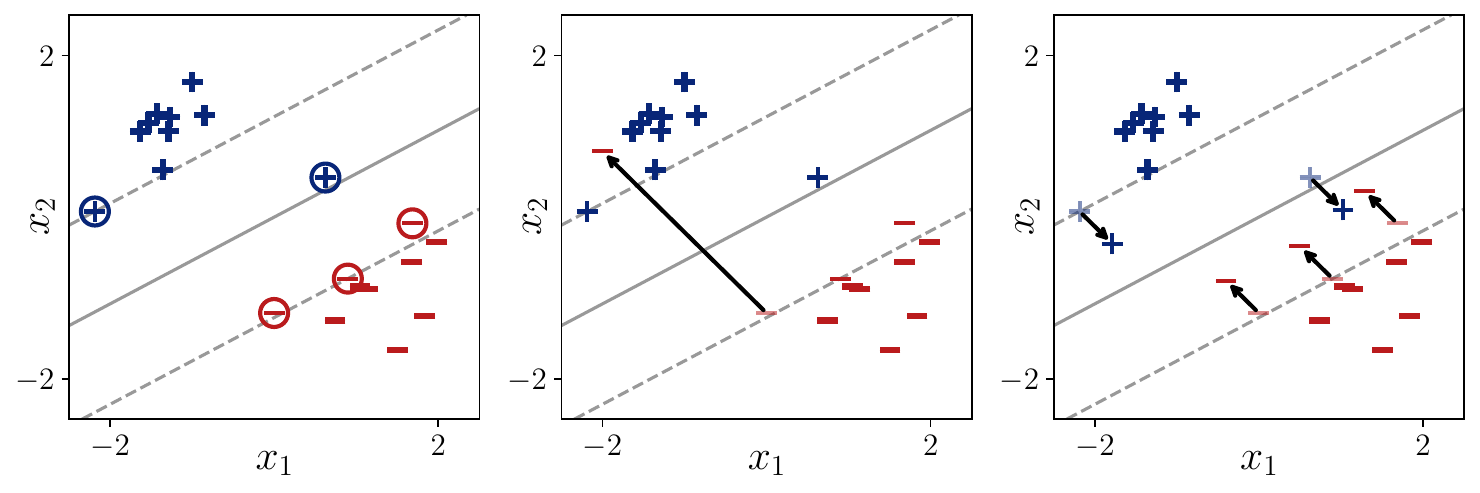}
    \caption{Different feature distributions in a distributionally robust support vector machine problem: Empirical distribution (left), a least favorable distribution obtained by perturbing a single sample in~$\cJ_+$ (center), and the least favorable distribution obtained by perturbing all samples in~$\cJ_+$ (right), using the 1-norm transportation cost (top), the 2-norm transportation cost (middle), and the $\infty$-norm transportation cost (bottom).}
    \label{figure:nash}
\end{figure}

In the second experiment we use a distributionally robust support vector machine to distinguish greyscale images of handwritten numbers~3 and~8 from the MNIST $3$-vs-$8$ dataset \citep{lecun1998gradient}. Any such image consists of 786 pixels with an intensity ranging from~0 (black) to~1 (white) and thus represents a feature vector~$X\in\cX=[0,1]^{d-1}$ with $d=787$. The corresponding label~$Y$ is set to~$+1$ if the image shows the number~3 and to $-1$ if the image shows the number~8. For ease of visualization, we use the first~$J=10$ records of the dataset as the training samples. In addition, we use the $\infty$-norm to quantify the transportation cost in the feature space. The goal of this experiment is to compare the least favorable distributions that solve the dual DRO problem ({\em i.e.}, the Nash strategies of nature) against the worst-case distributions that maximize the expected hinge loss when the classifier's weight vector is fixed to a minimizer~$\theta^\star$ of the primal DRO problem ({\em i.e.}, the best response strategies of nature when the statistician plays~$\theta^\star$). Specifically, we construct a least favorable distribution as in Theorem~\ref{theorem:nash:computation:ml}\,\ref{theorem:nash:computation:Q:ml}, which is possible because~$\cI_j^\infty = \emptyset$ for every $j\in[J]$ thanks to the compactness of~$\cX$. In addition, we compute~$\theta^\star$ as in Theorem~\ref{theorem:nash:computation:ml}\,\ref{theorem:nash:computation:theta:ml} and construct a worst-case distribution for~$\theta^\star$ as in~\citep[\textsection~3.2]{shafieezadeh2019regularization}. As every Nash strategy is a best response to the adversary's Nash strategy, it is clear that every least favorable distribution is a worst-case distribution for~$\theta^\star$. If the finite convex program used to construct the worst-case distribution has multiple optimal solutions, then the reverse implication is generally {\em false}. The optimization algorithm used to solve this convex program thus outputs an arbitrary worst-case distribution that generically fails to be a least favorable distribution.

Figure~\ref{figure:mnist} visualizes specific worst-case and least favorable distributions found by Gurobi for different radii~$\eps$ of the ambiguity set. The ten images corresponding to~$\eps=0$ represent the features of the ten unperturbed training samples. For any~$\eps>0$, both the worst-case distribution as well as the least favorable distribution are obtained by moving probability mass from some of the training samples to corresponding {\em adversarial} samples with the same labels but perturbed features. Whenever this happens, Figure~\ref{figure:mnist} shows the adversarial samples instead of repeating the corresponding training samples, and underneath each adversarial sample we indicate the probability mass---as a percentage of~$1/J$---inherited from the underlying training sample. We emphasize that the adversarial samples differ from all training samples and were thus `invented by nature' with the goal to confuse the statistician. The adversarial samples of the worst-case distribution differ from the corresponding training samples only in a few pixels that look like noise to the human eye. Some adversarial samples of the least favorable distribution, however, are truly deceptive. For example, the feature of the sixth training sample arguably shows the number~8, but for any $\eps>0$ this~8 is ostensibly converted to a~3 that was not present in the training dataset. We conclude that nature’s best response to the optimal classifier with weight vector~$\theta^\star$ is at best suitable to deceive an algorithm, whereas nature’s Nash strategy can even deceive a human. A possible explanation for this observation is that nature's Nash strategy aims to fool {\em every possible} classifier and not only {\em one single optimal} classifier.

\begin{figure}[t]
    \center
    \subfigure[Worst-case distributions]{\includegraphics[width=0.45\columnwidth]{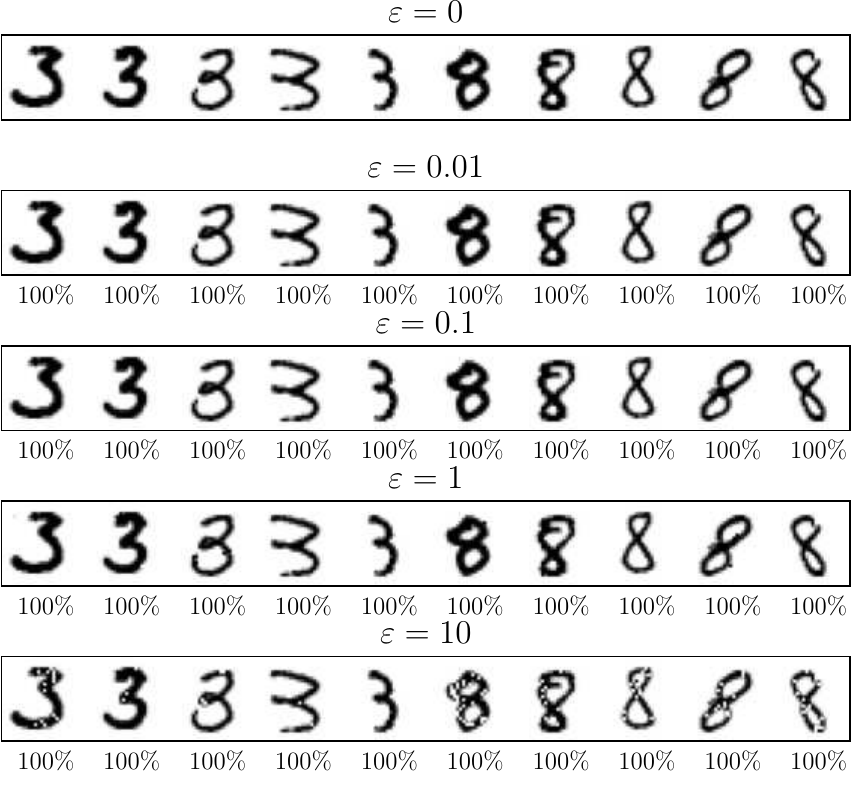}} ~
    \subfigure[Least favorable distributions]{\includegraphics[width=0.45\columnwidth]{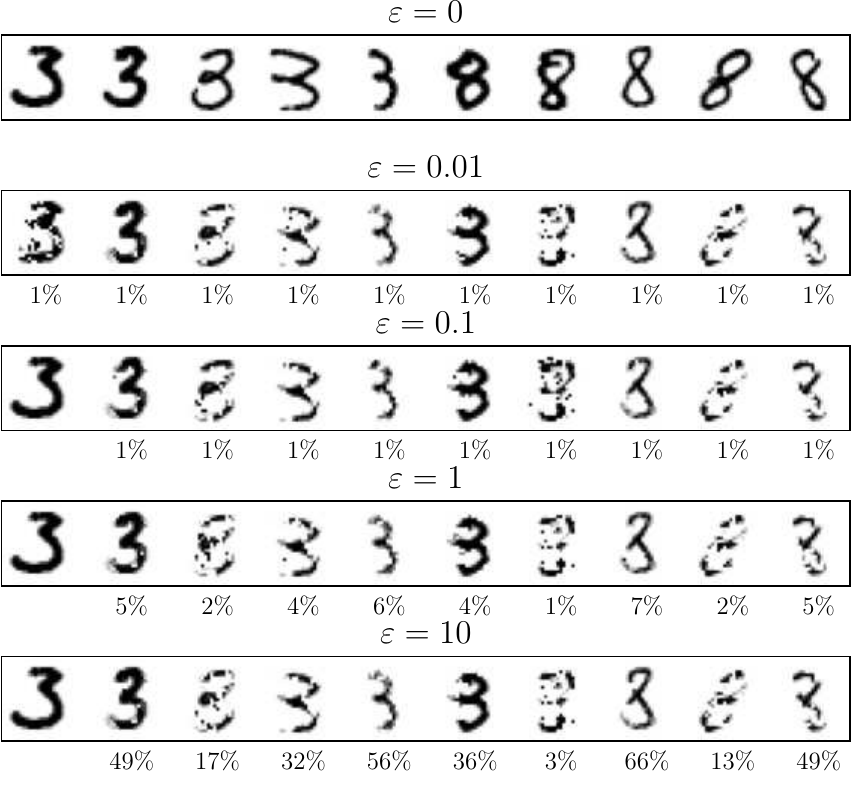}} 
    \caption{Comparison of worst-case and least favorable distributions for different values of~$\eps$. The number underneath each adversarial example indicates its probability mass as a percentage of~$1/J$.}
    \label{figure:mnist}
\end{figure}

\subsection{Distributionally Robust Log-Optimal Portfolio Selection}

Assume now that the components of the random vector~$Z\in\cZ=\R^d$ represent the total returns of~$d$ assets over the next month, say. If the asset returns over consecutive months are serially independent and governed by the same distribution~$\P\in\cP(\cZ)$ satisfying some plausible mild regularity conditions (such as $\P[Z\in\R^d_{++}]=1$), and if~$\Theta$ represents the probability simplex in~$\R^d$, then one can show that the constantly rebalanced portfolio~$\theta\in \Theta$ that maximizes the expected log-utility $\E_{Z \sim \P}[ \log(\inner{\theta}{Z})]$ generates more wealth than any other causal portfolio strategy with probability~1 in the long run \cite[Theorem~15.3.1]{cover2006elements}. Unfortunately, however, the asset return distribution~$\P$ is unknown in practice. It is therefore natural to study a distributionally robust problem formulation. In contrast to \cite{rujeerapaiboon2016robust}, where~$\P$ is assumed to be unknown except for its first- and second-order moments, we model distributional ambiguity here via an optimal transport-based ambiguity set centered at the empirical distribution~$\hat{\P} = (1 / J) \sum_{j \in [J]} \delta_{\, \hat z_j}$ on $J$ training samples $\hat z_j \in \R^d_{++}$, $j \in [J]$. Thus, we aim to solve
\begin{align}
\label{eq:dro:portfolio}
    \min_{\theta \in \Theta} \sup_{\Q \in \B_{\eps}(\hat{\P})} ~ \E_{Z \sim \Q} \left[ -\log(\inner{\theta}{Z}) \right],
\end{align}
which is an instance of~\eqref{eq:dro:linear} with $L(s)=-\log(s)$ if $s>0$ and $L(s)=+\infty$ if $s\leq 0$. If the transportation cost function is set to~$c(z,\hat z)=\|z-\hat z\|^p$ for some norm~$\|\cdot\|$ on~$\R^d$ and exponent~$p\geq 1$, then~$\B_{\eps}(\hat{\P})$ reduces to the $p$-th Wasserstein ball of radius~$\eps^p$ around~$\hat{\P}$. One readily verifies that any such Wasserstein ball contains distributions that assign a strictly positive mass to~0. Thus, the worst-case expected log-utility of any portfolio~$\theta\in\Theta$ is unbounded from above, which implies that problem~\eqref{eq:dro:portfolio} is infeasible.
To ensure that problem~\eqref{eq:dro:portfolio} is well-defined, the cost of moving any fixed probability mass towards~0 must tend to infinity. This can be ensured, for example, by setting $c(z, \hat z) = \sum_{i \in [d]} | \log(z_i / \hat z_i)) |$ with $\dom(c(
\cdot,\hat z))=\R^d_{++}$ for every $\hat z\in \R^d_{++}$. Even though it is nonconvex in both of its arguments, this transportation cost function defines a metric on $\cZ$ that gives rise to a valid optimal transport discrepancy. Elementary calculations show that $L^*(s) = -1 - \log(-s)$ and that $c^{*1}(s, \hat z) = \sum_{i\in[d]} h(s_i \hat z_i)$, where the auxiliary function $h:\R\rightarrow (-\infty, +\infty]$ is defined through $h(s)=-1-\log(-s)$ if $s<-1$, $h(s)=s$ if $-1\leq s\leq 0$ and $h(s)=+\infty$ if $s>0$. Note that the proper convex function $h$ is differentiable on its domain. By Theorem~\ref{theorem:nonconvex:duality}\,\ref{theorem:toland} and as~$\dom(L^*) = (-\infty,0)$, the $c$-transform~\eqref{eq:robust:loss} can thus be reformulated as
\begin{align}
    \label{eq:robust:portfolio}
    \ell_c(\theta, \lambda, \hat z)
    = \sup_{\gamma <0} \; 1 + \log(-\gamma)+
    \sum_{i \in [d]} \lambda \, h(\gamma \theta_i \hat z_i / \lambda).
\end{align}
The next proposition shows that the maximization problem in~\eqref{eq:robust:portfolio} can be solved efficiently by sorting.

\begin{proposition}
\label{proposition:portfolio}
    Given $\theta\in\Theta$, $\lambda \in \R_+$ and $\hat z \in \R^d_{++}$, set $u_i = \theta_i \hat z_i$ for every $i \in [d]$, let $\sigma$ be a permutation of $[d]$ with $u_{\sigma(1)} \leq u_{\sigma(2)} \leq \cdots \leq u_{\sigma(d)}$, and define
    \begin{align*}
        k = \max \left\{ i \in [d]\;:\; (1 - (d-i) \lambda) u_{\sigma(i)} \leq \lambda \sum_{j \in [i]} u_{\sigma(j)} \right\} \quad \text{and} \quad \gamma^\star = \frac{(d-k) \lambda - 1}{\sum_{i \in [k]} u_{\sigma(i)}}.
    \end{align*}
    If $\lambda \geq 1 / \| \theta \|_0$, then problem~\eqref{eq:robust:portfolio} is  solved by $\gamma^\star$. Otherwise, we have $\ell_c(\theta, \lambda, \hat z)=+\infty$.
\end{proposition}

By Proposition~\ref{proposition:strong:duality}, the distributionally robust portfolio selection problem~\eqref{eq:dro:portfolio} is equivalent to the convex stochastic optimization problem $\inf_{\theta\in\Theta,\,\lambda \geq 0} \E_{\hat Z \sim \hat{\P}} [ f(\theta, \lambda, \hat Z) ]$ with $f(\theta, \lambda,\hat z)=\lambda\eps +\ell_c(\theta, \lambda, \hat z)$, and Proposition~\ref{proposition:portfolio} implies that $f(\theta, \lambda, \hat z)=+\infty$ unless $\lambda\geq1/\|\theta\|_0$, which we may thus impose as an explicit constraint. Proposition~\ref{proposition:portfolio} further enables us to solve the equivalent stochastic program by ordinary or stochastic gradient descent. Indeed, Proposition~\ref{proposition:portfolio} implies via the envelope theorem~\citep[Theorem~2.16]{de2000mathematical} that if $\lambda \geq 1/\|\theta\|_0$, then the gradients of $f$ with respect to $\theta$ and $\lambda$ are given by 
\begin{align*}
    \nabla_{\theta} f (\theta, \lambda, \hat Z) &= \left( \gamma^\star \hat Z_1 \, h'(\gamma^\star \theta_1 \hat Z_1/\lambda),\ldots, \gamma^\star \hat Z_d \, h'(\gamma^\star \theta_d \hat Z_d/\lambda)\right)\\
    \nabla_{\lambda} f(\theta, \lambda, \hat Z) & = \eps + \sum_{i \in [d]} h(\gamma^\star \theta_i \hat Z_i/\lambda) - \frac{1}{\lambda} \sum_{i \in [d]} \gamma^\star \theta_i \hat Z_i \, h'(\gamma^\star \theta_i \hat Z_i/\lambda),
\end{align*}
where $\gamma^\star$ is defined in as Proposition~\ref{proposition:portfolio}, and $h'(s)= -1 / s$ if $s < -1$; $=1$ if $-1\leq s< 0$. Under a mild local uniform integrability condition, these random gradients represent unbiased estimators for the deterministic gradients $\nabla_\theta \E_{\hat Z \sim \hat{\P}} [ f(\theta, \lambda, \hat Z) ]$ and $\nabla_\lambda \E_{\hat Z \sim \hat{\P}} [ f(\theta, \lambda, \hat Z) ]$, respectively. In principle, these estimators can thus be used in stochastic gradient descent algorithms. However, they may fail to be Lipschitz continuous in $(\theta,\lambda)$ unless the support of~$\hat \P$ is a compact subset of~$\R_{++}^d$. Hence, stochastic gradient descent may fail to converge. Even if~$\hat \P$ is discrete, the Lipschitz modulus of the gradient estimators depends on the atoms of~$\hat \P$ and may thus become arbitrarily large, in which case classical as well as stochastic gradient descent suffer from numerical instability and poor convergence rates. We therefore suggest to solve the equivalent stochastic program with the adaptive Golden Ratio Algorithm introduced in~\citep[\textsection~2]{malitsky2020golden}. If~$\hat \P$ is discrete, this algorithm finds a $\delta$-suboptimal solution in~$\cO(1 / \delta)$ iterations. 

In the last experiment we assess the out-of-sample performance of various log-optimal portfolios with $d=10$ assets. To this end, we assume that the unknown true asset return distribution~$\P$ is lognormal and that its mean and covariance matrix coincide with the empirical mean and the Ledoit-Wolf covariance shrinkage estimator~\cite{ledoit2004well} corresponding to the 600 most recent monthly returns in the `$10$~Industry Portfolios' dataset from the Fama-French online data library.\footnote{\url{http://mba.tuck.dartmouth.edu/pages/faculty/ken.french/data_library.html}} Our experiment consists of $1{,}000$ independent trials. In each trial we construct~$\hat \P$ as the empirical distribution on $100$ training samples from~$\P$. Different distributionally robust log-optimal portfolios are then obtained by solving the stochastic programming reformulation of~\eqref{eq:dro:portfolio} with the adaptive Golden Ratio Algorithm for different values of $\eps \in \{ a \cdot 10^b: a \in \{1, \dots, 10 \},\; b \in \{-4, \cdots, -1 \} \}$. The sample average approximation (SAA) portfolio is obtained similarly by setting~$\eps=0$. Finally, the out-of-sample performance~$\E_{Z\sim\P}[-\log(\inner{\theta}{Z}]$ of any fixed portfolio~$\theta\in\Theta$ is evaluated empirically by using $10^6$~test samples from~$\P$. Figure~\ref{figure:portfolio}\,(a) plots the out-of-sample performance of the distributionally robust portfolios as a function of~$\eps$ averaged across all $1{,}000$ trials. We conclude that one can indeed improve on SAA by accounting for distributional ambiguity. Figure~\ref{figure:portfolio}\,(b) visualizes the distribution of the out-of-sample performance (which reflects the uncertainty of the $100$ training samples) of the SAA portfolio and the DRO portfolio corresponding to~$\eps=10^{-2}$. The picture suggests that accounting for distributional ambiguity not only reduces the expected logarithmic disutility but also (significantly) reduces the dispersion of the disutility. 
The average (monthly) return of the SAA portfolio equals $0.85\%$, while that of the DRO portfolio corresponding to~$\eps=10^{-2}$ is given by~$0.90\%$. In addition, the average Sharpe ratio of the SAA portfolio amounts to $0.20$, while that of the DRO portfolio corresponding to~$\eps=10^{-2}$ evaluates to~$0.25$. 

\begin{figure}[t]
    \center
    \subfigure{\includegraphics[width=0.505\columnwidth]{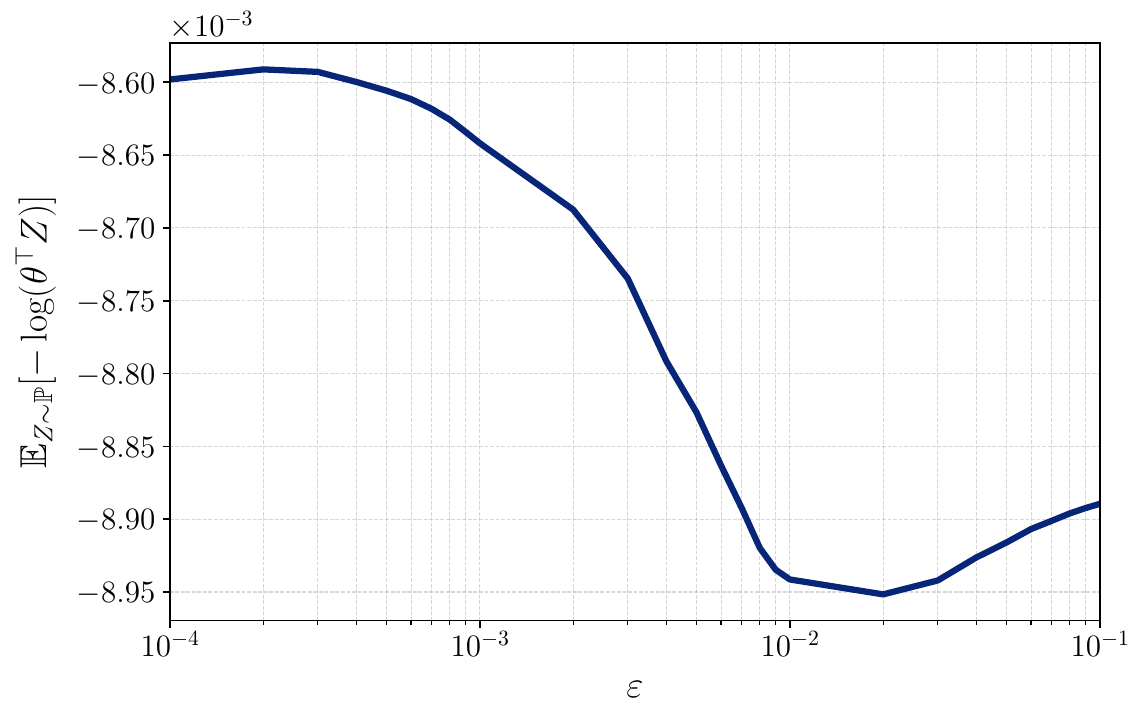}} ~
    \subfigure{\includegraphics[width=0.46\columnwidth]{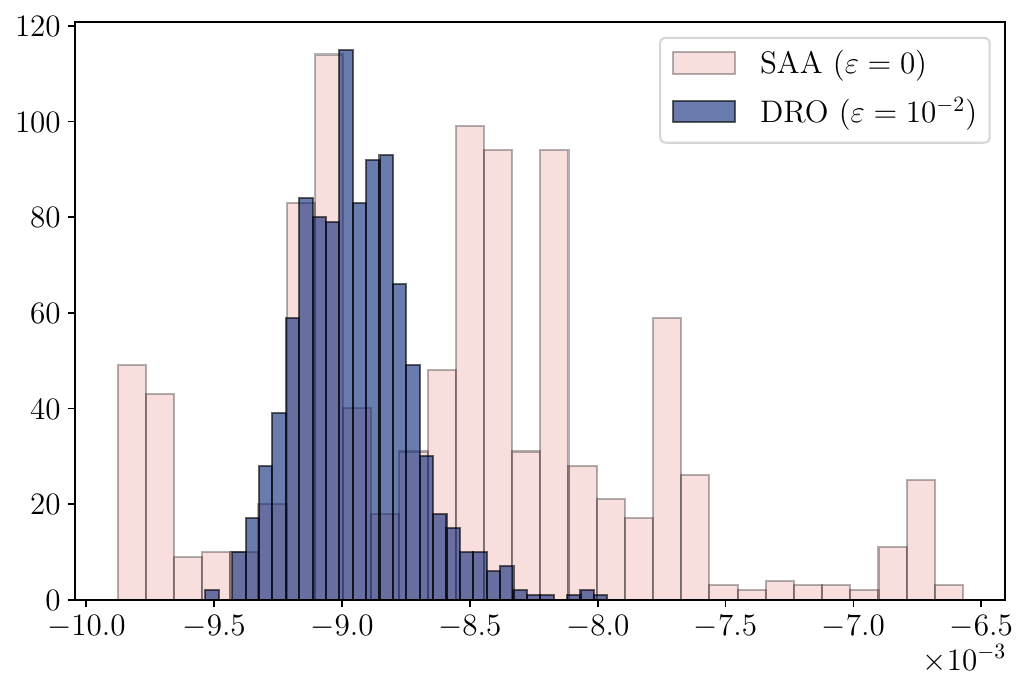}} 
    \caption{(Left) Average out-of-sample performance of the distributionally robust log-optimal portfolios as a function of $\eps$. (Right) Distribution of the out-of-sample performance for $\eps=0$ and $\eps =10^{-2}$.}
    \label{figure:portfolio}
\end{figure}

\textbf{Acknowledgements.} This research was supported by the Swiss National Science Foundation under the NCCR Automation (grant agreement~51NF40\_180545), and under Early Postdoc.Mobility Fellowships awarded to the first and second authors (grant agreements P2ELP2\_195149 and P500PT\_222215, respectively).
  \bibliographystyle{myabbrvnat} 
  \bibliography{bibfile.bib}

\begin{thebibliography}{105}
\providecommand{\natexlab}[1]{#1}
\providecommand{\url}[1]{\texttt{#1}}
\expandafter\ifx\csname urlstyle\endcsname\relax
  \providecommand{\doi}[1]{doi: #1}\else
  \providecommand{\doi}{doi: \begingroup \urlstyle{rm}\Url}\fi

\bibitem[Aolaritei et~al.(2023{\natexlab{a}})Aolaritei, Fochesato, Lygeros, and
  D{\"o}rfler]{aolaritei2023awasserstein}
L.~Aolaritei, M.~Fochesato, J.~Lygeros, and F.~D{\"o}rfler.
\newblock Wasserstein tube {MPC} with exact uncertainty propagation.
\newblock In \emph{Conference on Decision and Control}, pages 2036--2041,
  2023{\natexlab{a}}.

\bibitem[Aolaritei et~al.(2025)Aolaritei, Lanzetti, Chen, and
  D{\"o}rfler]{aolaritei2022distributional}
L.~Aolaritei, N.~Lanzetti, H.~Chen, and F.~D{\"o}rfler.
\newblock Distributional uncertainty propagation via optimal transport.
\newblock \emph{IEEE Transactions on Automatic Control}, 70\penalty0
  (9):\penalty0 6080--6095, 2025.

\bibitem[Aolaritei et~al.(2023{\natexlab{b}})Aolaritei, Shafiee, and
  Dörfler]{aolaritei2023wasserstein}
L.~Aolaritei, S.~Shafiee, and F.~Dörfler.
\newblock Wasserstein distributionally robust estimation in high dimensions:
  Performance analysis and optimal hyperparameter tuning.
\newblock \emph{arXiv preprint arXiv:2206.13269}, 2023{\natexlab{b}}.

\bibitem[Attouch and Az\'e(1993)]{attouch1993approximation}
H.~Attouch and D.~Az\'e.
\newblock Approximation and regularization of arbitrary functions in {H}ilbert
  spaces by the {L}asry-{L}ions method.
\newblock \emph{Annales de l'Institut Henri Poincar\'e}, 10\penalty0
  (3):\penalty0 289--312, 1993.

\bibitem[Attouch and Wets(1989)]{attouch1989epigraphical}
H.~Attouch and R.~J.-B. Wets.
\newblock Epigraphical analysis.
\newblock \emph{Annales de l'Institut Henri Poincar\'e}, 6:\penalty0 73--100,
  1989.

\bibitem[Banach(1938)]{Banach1938}
S.~Banach.
\newblock {\"Uber homogene Polynome in ($L^2$)}.
\newblock \emph{Studia Mathematica}, 7\penalty0 (1):\penalty0 36--44, 1938.

\bibitem[Bartl et~al.(2021)Bartl, Drapeau, Ob\l\'oj, and
  Wiesel]{bartl2020robust}
D.~Bartl, S.~Drapeau, J.~Ob\l\'oj, and J.~Wiesel.
\newblock Sensitivity analysis of {W}asserstein distributionally robust
  optimization problems.
\newblock \emph{Proceedings of the Royal Society A}, 477:\penalty0 20210176,
  2021.

\bibitem[Bauschke et~al.(2011)Bauschke, Combettes, et~al.]{bauschke2011convex}
H.~H. Bauschke, P.~L. Combettes, et~al.
\newblock \emph{Convex Analysis and Monotone Operator Theory in Hilbert
  Spaces}.
\newblock Springer, 2011.

\bibitem[Ben-Tal and Nemirovski(2001)]{ben2001lectures}
A.~Ben-Tal and A.~Nemirovski.
\newblock \emph{Lectures on Modern Convex Optimization: Analysis, Algorithms,
  and Engineering Applications}.
\newblock SIAM, 2001.

\bibitem[Bergemann and Schlag(2008)]{bergemann2008pricing}
D.~Bergemann and K.~H. Schlag.
\newblock Pricing without priors.
\newblock \emph{Journal of the European Economic Association}, 6\penalty0
  (2-3):\penalty0 560--569, 2008.

\bibitem[Bertsekas(1973)]{bertsekas1973stochastic}
D.~P. Bertsekas.
\newblock Stochastic optimization problems with nondifferentiable cost
  functionals.
\newblock \emph{Journal of Optimization Theory and Applications}, 12\penalty0
  (2):\penalty0 218--231, 1973.

\bibitem[Bertsekas(2009)]{bertsekas2009convex}
D.~P. Bertsekas.
\newblock \emph{Convex Optimization Theory}.
\newblock Athena Scientific, 2009.

\bibitem[Billingsley(2013)]{billingsley2013convergence}
P.~Billingsley.
\newblock \emph{Convergence of Probability Measures}.
\newblock Wiley, 2013.

\bibitem[Blanchet et~al.(2019{\natexlab{a}})Blanchet, Glynn, Yan, and
  Zhou]{blanchet2019multivariate}
J.~Blanchet, P.~W. Glynn, J.~Yan, and Z.~Zhou.
\newblock Multivariate distributionally robust convex regression under absolute
  error loss.
\newblock In \emph{Advances in Neural Information Processing Systems}, pages
  11817--11826, 2019{\natexlab{a}}.

\bibitem[Blanchet et~al.(2019{\natexlab{b}})Blanchet, Kang, and
  Murthy]{blanchet2019robust}
J.~Blanchet, Y.~Kang, and K.~Murthy.
\newblock Robust {W}asserstein profile inference and applications to machine
  learning.
\newblock \emph{Journal of Applied Probability}, 56\penalty0 (3):\penalty0
  830--857, 2019{\natexlab{b}}.

\bibitem[Blanchet and Murthy(2019)]{blanchet2019quantifying}
J.~Blanchet and K.~Murthy.
\newblock Quantifying distributional model risk via optimal transport.
\newblock \emph{Mathematics of Operations Research}, 44\penalty0 (2):\penalty0
  565--600, 2019.

\bibitem[Blanchet et~al.(2022{\natexlab{a}})Blanchet, Murthy, and
  Si]{blanchet2019confidence}
J.~Blanchet, K.~Murthy, and N.~Si.
\newblock Confidence regions in {W}asserstein distributionally robust
  estimation.
\newblock \emph{Biometrika}, 109\penalty0 (2):\penalty0 295--315,
  2022{\natexlab{a}}.

\bibitem[Blanchet et~al.(2022{\natexlab{b}})Blanchet, Murthy, and
  Zhang]{blanchet2018optimal}
J.~Blanchet, K.~Murthy, and F.~Zhang.
\newblock Optimal transport-based distributionally robust optimization:
  Structural properties and iterative schemes.
\newblock \emph{Mathematics of Operations Research}, 47\penalty0 (2):\penalty0
  1500--1529, 2022{\natexlab{b}}.

\bibitem[Boskos et~al.(2021)Boskos, Cortes, and
  Martinez]{Boskos2021Data-drivenProcesses}
D.~Boskos, J.~Cortes, and S.~Martinez.
\newblock Data-driven ambiguity sets with probabilistic guarantees for dynamic
  processes.
\newblock \emph{IEEE Transactions on Automatic Control}, 66\penalty0
  (7):\penalty0 2991--3006, 2021.

\bibitem[Bougeard et~al.(1991)Bougeard, Penot, and
  Pommellet]{bougeard1991towards}
M.~Bougeard, J.-P. Penot, and A.~Pommellet.
\newblock Towards minimal assumptions for the infimal convolution
  regularization.
\newblock \emph{Journal of Approximation Theory}, 64\penalty0 (3):\penalty0
  245--270, 1991.

\bibitem[Bredies et~al.(2010)Bredies, Kunisch, and Pock]{bredies2010total}
K.~Bredies, K.~Kunisch, and T.~Pock.
\newblock Total generalized variation.
\newblock \emph{SIAM Journal on Imaging Sciences}, 3\penalty0 (3):\penalty0
  492--526, 2010.

\bibitem[Chen and Paschalidis(2019)]{chen2019selecting}
R.~Chen and I.~Paschalidis.
\newblock Selecting optimal decisions via distributionally robust
  nearest-neighbor regression.
\newblock In \emph{Advances in Neural Information Processing Systems}, pages
  749--759, 2019.

\bibitem[Chen and Paschalidis(2018)]{chen2018robust}
R.~Chen and I.~C. Paschalidis.
\newblock A robust learning approach for regression models based on
  distributionally robust optimization.
\newblock \emph{Journal of Machine Learning Research}, 19\penalty0
  (1):\penalty0 517--564, 2018.

\bibitem[Chen et~al.(2024)Chen, Kuhn, and Wiesemann]{chen2024data}
Z.~Chen, D.~Kuhn, and W.~Wiesemann.
\newblock Data-driven chance constrained programs over {W}asserstein balls.
\newblock \emph{Operations Research}, 72\penalty0 (1):\penalty0 410--424, 2024.

\bibitem[Cl{\'e}ment and Desch(2008)]{clement2008wasserstein}
P.~Cl{\'e}ment and W.~Desch.
\newblock Wasserstein metric and subordination.
\newblock \emph{Studia Mathematica}, 1\penalty0 (189):\penalty0 35--52, 2008.

\bibitem[Coulson et~al.(2021)Coulson, Lygeros, and
  D{\"o}rfler]{coulson2021distributionally}
J.~Coulson, J.~Lygeros, and F.~D{\"o}rfler.
\newblock Distributionally robust chance constrained data-enabled predictive
  control.
\newblock \emph{IEEE Transactions on Automatic Control}, 67\penalty0
  (7):\penalty0 3289--3304, 2021.

\bibitem[Cover and Thomas(2006)]{cover2006elements}
T.~Cover and J.~Thomas.
\newblock \emph{Elements of Information Theory}.
\newblock Wiley, 2006.

\bibitem[De~la Fuente(2000)]{de2000mathematical}
A.~De~la Fuente.
\newblock \emph{Mathematical Methods and Models for Economists}.
\newblock Cambridge University Press, 2000.

\bibitem[Duchi et~al.(2021)Duchi, Glynn, and Namkoong]{duchi2021statistics}
J.~C. Duchi, P.~W. Glynn, and H.~Namkoong.
\newblock Statistics of robust optimization: {A} generalized empirical
  likelihood approach.
\newblock \emph{Mathematics of Operations Research}, 46\penalty0 (3):\penalty0
  946--969, 2021.

\bibitem[Duchi and Namkoong(2021)]{duchi2020learning}
J.~C. Duchi and H.~Namkoong.
\newblock Learning models with uniform performance via distributionally robust
  optimization.
\newblock \emph{Annals of Statistics}, 49\penalty0 (3):\penalty0 1378--1406,
  2021.

\bibitem[Farnia and Tse(2016)]{farnia2016minimax}
F.~Farnia and D.~Tse.
\newblock A minimax approach to supervised learning.
\newblock In \emph{Advances in Neural Information Processing Systems}, pages
  4240--4248, 2016.

\bibitem[Finlay and Oberman(2021)]{finlay2021scaleable}
C.~Finlay and A.~M. Oberman.
\newblock Scaleable input gradient regularization for adversarial robustness.
\newblock \emph{Machine Learning with Applications}, 3:\penalty0 100017, 2021.

\bibitem[Gao(2023)]{gao2023finite}
R.~Gao.
\newblock Finite-sample guarantees for {W}asserstein distributionally robust
  optimization: Breaking the curse of dimensionality.
\newblock \emph{Operations Research}, 71\penalty0 (6):\penalty0 2291--2306,
  2023.

\bibitem[Gao et~al.(2024)Gao, Chen, and Kleywegt]{gao2024wasserstein}
R.~Gao, X.~Chen, and A.~J. Kleywegt.
\newblock Wasserstein distributionally robust optimization and variation
  regularization.
\newblock \emph{Operations Research}, 72\penalty0 (3):\penalty0 1177--1191,
  2024.

\bibitem[Gao and Kleywegt(2023)]{gao2023distributionally}
R.~Gao and A.~Kleywegt.
\newblock Distributionally robust stochastic optimization with {W}asserstein
  distance.
\newblock \emph{Mathematics of Operations Research}, 48\penalty0 (2):\penalty0
  603--655, 2023.

\bibitem[Gao et~al.(2018)Gao, Xie, Xie, and Xu]{gao2018robust}
R.~Gao, L.~Xie, Y.~Xie, and H.~Xu.
\newblock Robust hypothesis testing using {W}asserstein uncertainty sets.
\newblock In \emph{Advances in Neural Information Processing Systems}, pages
  7902--7912, 2018.

\bibitem[Goodfellow et~al.(2015)Goodfellow, Shlens, and
  Szegedy]{ian15adversarial}
I.~J. Goodfellow, J.~Shlens, and C.~Szegedy.
\newblock Explaining and harnessing adversarial examples.
\newblock In \emph{International Conference on Learning Representations}, 2015.

\bibitem[Gulrajani et~al.(2017)Gulrajani, Ahmed, Arjovsky, Dumoulin, and
  Courville]{gulrajani2017improved}
I.~Gulrajani, F.~Ahmed, M.~Arjovsky, V.~Dumoulin, and A.~Courville.
\newblock Improved training of {W}asserstein {GAN}s.
\newblock In \emph{Advances in Neural Information Processing Systems}, pages
  5769--5779, 2017.

\bibitem[Hein and Andriushchenko(2017)]{hein2017formal}
M.~Hein and M.~Andriushchenko.
\newblock Formal guarantees on the robustness of a classifier against
  adversarial manipulation.
\newblock In \emph{Advances in Neural Information Processing Systems}, pages
  2263--2273, 2017.

\bibitem[Hiriart-Urruty(1980)]{hiriart1980extension}
J.-B. Hiriart-Urruty.
\newblock Extension of {L}ipschitz functions.
\newblock \emph{Journal of Mathematical Analysis and Applications}, 77\penalty0
  (2):\penalty0 539--554, 1980.

\bibitem[Ho-Nguyen et~al.(2022)Ho-Nguyen, K{\i}l{\i}n{\c{c}}-Karzan,
  K{\"u}{\c{c}}{\"u}kyavuz, and Lee]{ho2020distributionally}
N.~Ho-Nguyen, F.~K{\i}l{\i}n{\c{c}}-Karzan, S.~K{\"u}{\c{c}}{\"u}kyavuz, and
  D.~Lee.
\newblock Distributionally robust chance-constrained programs with right-hand
  side uncertainty under {W}asserstein ambiguity.
\newblock \emph{Mathematical Programming}, 196\penalty0 (1--2):\penalty0
  641--672, 2022.

\bibitem[Ho-Nguyen et~al.(2023)Ho-Nguyen, Kilin{\c{c}}-Karzan,
  K{\"u}{\c{c}}{\"u}kyavuz, and Lee]{ho2023strong}
N.~Ho-Nguyen, F.~Kilin{\c{c}}-Karzan, S.~K{\"u}{\c{c}}{\"u}kyavuz, and D.~Lee.
\newblock Strong formulations for distributionally robust chance-constrained
  programs with left-hand side uncertainty under {W}asserstein ambiguity.
\newblock \emph{INFORMS Journal on Optimization}, 5\penalty0 (2):\penalty0
  211--232, 2023.

\bibitem[Ho-Nguyen and Wright(2023)]{ho2023adversarial}
N.~Ho-Nguyen and S.~J. Wright.
\newblock Adversarial classification via distributional robustness with
  {W}asserstein ambiguity.
\newblock \emph{Mathematical Programming}, 198\penalty0 (2):\penalty0
  1411--1447, 2023.

\bibitem[Hu et~al.(2014)Hu, Ongie, Ramani, and Jacob]{hu2014generalized}
Y.~Hu, G.~Ongie, S.~Ramani, and M.~Jacob.
\newblock {Generalized higher degree total variation (HDTV) regularization}.
\newblock \emph{IEEE Transactions on Image Processing}, 23\penalty0
  (6):\penalty0 2423--2435, 2014.

\bibitem[Jakubovitz and Giryes(2018)]{jakubovitz2018improving}
D.~Jakubovitz and R.~Giryes.
\newblock Improving {DNN} robustness to adversarial attacks using {J}acobian
  regularization.
\newblock In \emph{European Conference on Computer Vision}, pages 514--529,
  2018.

\bibitem[Ko{\c{c}}yi{\u{g}}it et~al.(2020)Ko{\c{c}}yi{\u{g}}it, Iyengar, Kuhn,
  and Wiesemann]{koccyiugit2020distributionally}
{\c{C}}.~Ko{\c{c}}yi{\u{g}}it, G.~Iyengar, D.~Kuhn, and W.~Wiesemann.
\newblock Distributionally robust mechanism design.
\newblock \emph{Management Science}, 66\penalty0 (1):\penalty0 159--189, 2020.

\bibitem[Ko{\c{c}}yi{\u{g}}it et~al.(2022)Ko{\c{c}}yi{\u{g}}it, Rujeerapaiboon,
  and Kuhn]{koccyiugit2021robust}
{\c{C}}.~Ko{\c{c}}yi{\u{g}}it, N.~Rujeerapaiboon, and D.~Kuhn.
\newblock Robust multidimensional pricing: {S}eparation without regret.
\newblock \emph{Mathematical Programming}, 196\penalty0 (1--2):\penalty0
  841--874, 2022.

\bibitem[Krantz and Parks(2002)]{krantz2002primer}
S.~G. Krantz and H.~R. Parks.
\newblock \emph{A Primer of Real Analytic Functions}.
\newblock Springer, 2002.

\bibitem[Kuhn et~al.(2019)Kuhn, Mohajerin~Esfahani, Nguyen, and
  Shafieezadeh-Abadeh]{kuhn2019wasserstein}
D.~Kuhn, P.~Mohajerin~Esfahani, V.~A. Nguyen, and S.~Shafieezadeh-Abadeh.
\newblock {Wasserstein distributionally robust optimization: Theory and
  applications in machine learning}.
\newblock In \emph{Operations Research \& Management Science in the Age of
  Analytics}, pages 130--166. INFORMS, 2019.

\bibitem[Kuhn et~al.(2025)Kuhn, Shafiee, and
  Wiesemann]{kuhn2024distributionally}
D.~Kuhn, S.~Shafiee, and W.~Wiesemann.
\newblock Distributionally robust optimization.
\newblock \emph{Acta Numerica}, 34:\penalty0 579--804, 2025.

\bibitem[Kwon et~al.(2020)Kwon, Kim, Won, and Paik]{kwon2020principled}
Y.~Kwon, W.~Kim, J.-H. Won, and M.~C. Paik.
\newblock Principled learning method for {W}asserstein distributionally robust
  optimization with local perturbations.
\newblock In \emph{International Conference on Machine Learning}, pages
  5567--5576, 2020.

\bibitem[Lam(2019)]{lam2019recovering}
H.~Lam.
\newblock Recovering best statistical guarantees via the empirical
  divergence-based distributionally robust optimization.
\newblock \emph{Operations Research}, 67\penalty0 (4):\penalty0 1090--1105,
  2019.

\bibitem[LeCun et~al.(1998)LeCun, Bottou, Bengio, and
  Haffner]{lecun1998gradient}
Y.~LeCun, L.~Bottou, Y.~Bengio, and P.~Haffner.
\newblock Gradient-based learning applied to document recognition.
\newblock \emph{Proceedings of the IEEE}, 86\penalty0 (11):\penalty0
  2278--2324, 1998.

\bibitem[Ledoit and Wolf(2004)]{ledoit2004well}
O.~Ledoit and M.~Wolf.
\newblock A well-conditioned estimator for large-dimensional covariance
  matrices.
\newblock \emph{Journal of Multivariate Analysis}, 88\penalty0 (2):\penalty0
  365--411, 2004.

\bibitem[Lee and Raginsky(2018)]{lee2018minimax}
J.~Lee and M.~Raginsky.
\newblock Minimax statistical learning with {W}asserstein distances.
\newblock In \emph{Advances in Neural Information Processing Systems}, pages
  2687--2696, 2018.

\bibitem[Lefkimmiatis and Unser(2013)]{lefkimmiatis2013poisson}
S.~Lefkimmiatis and M.~Unser.
\newblock {Poisson image reconstruction with Hessian Schatten-norm
  regularization}.
\newblock \emph{IEEE Transactions on Image Processing}, 22\penalty0
  (11):\penalty0 4314--4327, 2013.

\bibitem[Lefkimmiatis et~al.(2013)Lefkimmiatis, Ward, and
  Unser]{lefkimmiatis2013hessian}
S.~Lefkimmiatis, J.~P. Ward, and M.~Unser.
\newblock {Hessian Schatten-norm regularization for linear inverse problems}.
\newblock \emph{IEEE Transactions on Image Processing}, 22\penalty0
  (5):\penalty0 1873--1888, 2013.

\bibitem[Lehmann and Casella(2006)]{lehmann2006theory}
E.~L. Lehmann and G.~Casella.
\newblock \emph{Theory of Point Estimation}.
\newblock Springer, 2006.

\bibitem[Levy and Nikoukhah(2004)]{levy2004robust}
B.~C. Levy and R.~Nikoukhah.
\newblock Robust least-squares estimation with a relative entropy constraint.
\newblock \emph{IEEE Transactions on Information Theory}, 50\penalty0
  (1):\penalty0 89--104, 2004.

\bibitem[Levy and Nikoukhah(2012)]{levy2012robust}
B.~C. Levy and R.~Nikoukhah.
\newblock Robust state space filtering under incremental model perturbations
  subject to a relative entropy tolerance.
\newblock \emph{IEEE Transactions on Automatic Control}, 58\penalty0
  (3):\penalty0 682--695, 2012.

\bibitem[Lyu et~al.(2015)Lyu, Huang, and Liang]{lyu2015unified}
C.~Lyu, K.~Huang, and H.-N. Liang.
\newblock A unified gradient regularization family for adversarial examples.
\newblock In \emph{International Conference on Data Mining}, pages 301--309,
  2015.

\bibitem[M{\k{a}}dry et~al.(2018)M{\k{a}}dry, Makelov, Schmidt, Tsipras, and
  Vladu]{madry2018towards}
A.~M{\k{a}}dry, A.~Makelov, L.~Schmidt, D.~Tsipras, and A.~Vladu.
\newblock Towards deep learning models resistant to adversarial attacks.
\newblock In \emph{International Conference on Learning Representations}, 2018.

\bibitem[Malitsky(2020)]{malitsky2020golden}
Y.~Malitsky.
\newblock Golden ratio algorithms for variational inequalities.
\newblock \emph{Mathematical Programming}, 184\penalty0 (1):\penalty0 383--410,
  2020.

\bibitem[Milgrom and Segal(2002)]{milgrom2002envelope}
P.~Milgrom and I.~Segal.
\newblock Envelope theorems for arbitrary choice sets.
\newblock \emph{Econometrica}, 70\penalty0 (2):\penalty0 583--601, 2002.

\bibitem[Mohajerin~Esfahani and Kuhn(2018)]{esfahani2018data}
P.~Mohajerin~Esfahani and D.~Kuhn.
\newblock {Data-driven distributionally robust optimization using the
  Wasserstein metric: Performance guarantees and tractable reformulations}.
\newblock \emph{Mathematical Programming}, 171\penalty0 (1--2):\penalty0
  115--166, 2018.

\bibitem[Mohajerin~Esfahani et~al.(2018)Mohajerin~Esfahani,
  Shafieezadeh-Abadeh, Hanasusanto, and Kuhn]{esfahani2018inverse}
P.~Mohajerin~Esfahani, S.~Shafieezadeh-Abadeh, G.~A. Hanasusanto, and D.~Kuhn.
\newblock Data-driven inverse optimization with imperfect information.
\newblock \emph{Mathematical Programming}, 167\penalty0 (1):\penalty0 191--234,
  2018.

\bibitem[Munkres(2000)]{munkres2000topology}
J.~R. Munkres.
\newblock \emph{Topology}.
\newblock Prentice Hall, 2000.

\bibitem[Nagarajan and Kolter(2017)]{nagarajan2017gradient}
V.~Nagarajan and J.~Z. Kolter.
\newblock {Gradient descent GAN optimization is locally stable}.
\newblock In \emph{Advances in Neural Information Processing Systems}, pages
  5591--5600, 2017.

\bibitem[Nguyen et~al.(2022)Nguyen, Kuhn, and
  Mohajerin~Esfahani]{nguyen2020distributionally}
V.~A. Nguyen, D.~Kuhn, and P.~Mohajerin~Esfahani.
\newblock Distributionally robust inverse covariance estimation: The
  {W}asserstein shrinkage estimator.
\newblock \emph{Operations Research}, 70\penalty0 (1):\penalty0 490--515, 2022.

\bibitem[Nguyen et~al.(2023)Nguyen, Shafieezadeh-Abadeh, Kuhn, and
  Mohajerin~Esfahani]{nguyen2023bridging}
V.~A. Nguyen, S.~Shafieezadeh-Abadeh, D.~Kuhn, and P.~Mohajerin~Esfahani.
\newblock Bridging {B}ayesian and minimax mean square error estimation via
  {W}asserstein distributionally robust optimization.
\newblock \emph{Mathematics of Operations Research}, 48\penalty0 (1):\penalty0
  1--37, 2023.

\bibitem[Ororbia~II et~al.(2017)Ororbia~II, Kifer, and
  Giles]{ororbia2017unifying}
A.~G. Ororbia~II, D.~Kifer, and C.~L. Giles.
\newblock Unifying adversarial training algorithms with data gradient
  regularization.
\newblock \emph{Neural Computation}, 29\penalty0 (4):\penalty0 867--887, 2017.

\bibitem[Papernot et~al.(2016{\natexlab{a}})Papernot, McDaniel, and
  Goodfellow]{papernot2016transferability}
N.~Papernot, P.~McDaniel, and I.~Goodfellow.
\newblock Transferability in machine learning: From phenomena to black-box
  attacks using adversarial samples.
\newblock \emph{arXiv:1605.07277}, 2016{\natexlab{a}}.

\bibitem[Papernot et~al.(2016{\natexlab{b}})Papernot, McDaniel, Wu, Jha, and
  Swami]{papernot2016distillation}
N.~Papernot, P.~McDaniel, X.~Wu, S.~Jha, and A.~Swami.
\newblock Distillation as a defense to adversarial perturbations against deep
  neural networks.
\newblock In \emph{IEEE Symposium on Security and Privacy}, pages 582--597,
  2016{\natexlab{b}}.

\bibitem[Parikh and Boyd(2014)]{parikh2014proximal}
N.~Parikh and S.~Boyd.
\newblock Proximal algorithms.
\newblock \emph{Foundations and Trends in Optimization}, 1\penalty0
  (3):\penalty0 127--239, 2014.

\bibitem[Penot(1998)]{penot1998proximal}
J.-P. Penot.
\newblock Proximal mappings.
\newblock \emph{Journal of Approximation Theory}, 94\penalty0 (2):\penalty0
  203--221, 1998.

\bibitem[Polyak(1987)]{polyak1987intro}
B.~Polyak.
\newblock \emph{Introduction to Optimization}.
\newblock Optimization Software, Inc., 1987.

\bibitem[Rockafellar(1970)]{rockafellar1970convex}
R.~T. Rockafellar.
\newblock \emph{Convex Analysis}.
\newblock Princeton University Press, 1970.

\bibitem[Rockafellar(1974)]{rockafellar1974conjugate}
R.~T. Rockafellar.
\newblock \emph{Conjugate Duality and Optimization}.
\newblock SIAM, 1974.

\bibitem[Rockafellar and Wets(2009)]{rockafellar2009variational}
R.~T. Rockafellar and R.~J.-B. Wets.
\newblock \emph{Variational Analysis}.
\newblock Springer, 2009.

\bibitem[Roth et~al.(2017)Roth, Lucchi, Nowozin, and
  Hofmann]{roth2017stabilizing}
K.~Roth, A.~Lucchi, S.~Nowozin, and T.~Hofmann.
\newblock Stabilizing training of generative adversarial networks through
  regularization.
\newblock In \emph{Advances in Neural Information Processing Systems}, pages
  2018--2028, 2017.

\bibitem[Rujeerapaiboon et~al.(2016)Rujeerapaiboon, Kuhn, and
  Wiesemann]{rujeerapaiboon2016robust}
N.~Rujeerapaiboon, D.~Kuhn, and W.~Wiesemann.
\newblock Robust growth-optimal portfolios.
\newblock \emph{Management Science}, 62\penalty0 (7):\penalty0 2090--2109,
  2016.

\bibitem[Shafieezadeh-Abadeh et~al.(2019)Shafieezadeh-Abadeh, Kuhn, and
  Mohajerin~Esfahani]{shafieezadeh2019regularization}
S.~Shafieezadeh-Abadeh, D.~Kuhn, and P.~Mohajerin~Esfahani.
\newblock Regularization via mass transportation.
\newblock \emph{Journal of Machine Learning Research}, 20\penalty0
  (103):\penalty0 1--68, 2019.

\bibitem[Shafieezadeh-Abadeh et~al.(2015)Shafieezadeh-Abadeh,
  Mohajerin~Esfahani, and Kuhn]{shafieezadeh2015distributionally}
S.~Shafieezadeh-Abadeh, P.~Mohajerin~Esfahani, and D.~Kuhn.
\newblock Distributionally robust logistic regression.
\newblock In \emph{Advances in Neural Information Processing Systems}, pages
  1576--1584, 2015.

\bibitem[Shafieezadeh-Abadeh et~al.(2018)Shafieezadeh-Abadeh, Nguyen, Kuhn, and
  Mohajerin~Esfahani]{shafieezadeh2018wasserstein}
S.~Shafieezadeh-Abadeh, V.~A. Nguyen, D.~Kuhn, and P.~Mohajerin~Esfahani.
\newblock Wasserstein distributionally robust {K}alman filtering.
\newblock In \emph{Advances in Neural Information Processing Systems}, pages
  8474--8483, 2018.

\bibitem[Shapiro et~al.(2014)Shapiro, Dentcheva, and
  Ruszczy{\'n}ski]{shapiro2014lectures}
A.~Shapiro, D.~Dentcheva, and A.~Ruszczy{\'n}ski.
\newblock \emph{Lectures on Stochastic Programming: Modeling and Theory}.
\newblock SIAM, 2014.

\bibitem[Shen and Jiang(2022)]{shen2020chance}
H.~Shen and R.~Jiang.
\newblock Chance-constrained set covering with {W}asserstein ambiguity.
\newblock \emph{Mathematical Programming (Forthcoming)}, 2022.

\bibitem[Sinha et~al.(2018)Sinha, Namkoong, and Duchi]{sinha2018certifying}
A.~Sinha, H.~Namkoong, and J.~Duchi.
\newblock Certifying some distributional robustness with principled adversarial
  training.
\newblock In \emph{International Conference on Learning Representations}, 2018.

\bibitem[Sion(1958)]{sion1958general}
M.~Sion.
\newblock On general minimax theorems.
\newblock \emph{Pacific Journal of Mathematics}, 8\penalty0 (1):\penalty0
  171--176, 1958.

\bibitem[Szegedy et~al.(2014)Szegedy, Zaremba, Sutskever, Bruna, Erhan,
  Goodfellow, and Fergus]{szegedy2014intriguing}
C.~Szegedy, W.~Zaremba, I.~Sutskever, J.~Bruna, D.~Erhan, I.~J. Goodfellow, and
  R.~Fergus.
\newblock Intriguing properties of neural networks.
\newblock In \emph{International Conference on Learning Representations}, 2014.

\bibitem[Ta{\c{s}}kesen et~al.(2023)Ta{\c{s}}kesen, Shafieezadeh-Abadeh, and
  Kuhn]{tacskesen2023semi}
B.~Ta{\c{s}}kesen, S.~Shafieezadeh-Abadeh, and D.~Kuhn.
\newblock Semi-discrete optimal transport: Hardness, regularization and
  numerical solution.
\newblock \emph{Mathematical Programming}, 199\penalty0 (1):\penalty0
  1033--1106, 2023.

\bibitem[Toland(1978)]{toland1978duality}
J.~F. Toland.
\newblock Duality in nonconvex optimization.
\newblock \emph{Journal of Mathematical Analysis and Applications}, 66\penalty0
  (2):\penalty0 399--415, 1978.

\bibitem[Tu et~al.(2019)Tu, Zhang, and Tao]{tu2019theoretical}
Z.~Tu, J.~Zhang, and D.~Tao.
\newblock Theoretical analysis of adversarial learning: {A} minimax approach.
\newblock In \emph{Advances in Neural Information Processing Systems}, pages
  12280--12290, 2019.

\bibitem[Van~Parys et~al.(2021)Van~Parys, Mohajerin~Esfahani, and
  Kuhn]{van2020data}
B.~P. Van~Parys, P.~Mohajerin~Esfahani, and D.~Kuhn.
\newblock From data to decisions: {D}istributionally robust optimization is
  optimal.
\newblock \emph{Management Science}, 67\penalty0 (6):\penalty0 3387--3402,
  2021.

\bibitem[Varga et~al.(2018)Varga, Csisz{\'a}rik, and
  Zombori]{varga2017gradient}
D.~Varga, A.~Csisz{\'a}rik, and Z.~Zombori.
\newblock Gradient regularization improves accuracy of discriminative models.
\newblock \emph{Schedae Informaticae}, 27:\penalty0 31--45, 2018.

\bibitem[Villani(2008)]{villani2008optimal}
C.~Villani.
\newblock \emph{Optimal Transport: Old and New}.
\newblock Springer, 2008.

\bibitem[Volpi et~al.(2018)Volpi, Namkoong, Sener, Duchi, Murino, and
  Savarese]{volpi2018generalizing}
R.~Volpi, H.~Namkoong, O.~Sener, J.~Duchi, V.~Murino, and S.~Savarese.
\newblock Generalizing to unseen domains via adversarial data augmentation.
\newblock In \emph{Advances in Neural Information Processing Systems}, pages
  5339--5349, 2018.

\bibitem[Wang et~al.(2019)Wang, Ma, Bailey, Yi, Zhou, and
  Gu]{wang2019convergence}
Y.~Wang, X.~Ma, J.~Bailey, J.~Yi, B.~Zhou, and Q.~Gu.
\newblock On the convergence and robustness of adversarial training.
\newblock In \emph{International Conference on Machine Learning}, pages
  6586--6595, 2019.

\bibitem[Xie(2021)]{xie2019distributionally}
W.~Xie.
\newblock On distributionally robust chance constrained programs with
  {W}asserstein distance.
\newblock \emph{Mathematical Programming}, 186\penalty0 (1):\penalty0 115--155,
  2021.

\bibitem[Yang(2020)]{Yang2020b}
I.~Yang.
\newblock Wasserstein distributionally robust stochastic control: A data-driven
  approach.
\newblock \emph{IEEE Transactions on Automatic Control}, 66\penalty0
  (8):\penalty0 3863--3870, 2020.

\bibitem[Yue et~al.(2022)Yue, Kuhn, and Wiesemann]{yue2020linear}
M.-C. Yue, D.~Kuhn, and W.~Wiesemann.
\newblock On linear optimization over {W}asserstein balls.
\newblock \emph{Mathematical Programming}, 195\penalty0 (1--2):\penalty0
  1107--1122, 2022.

\bibitem[Zhang et~al.(2024)Zhang, Yang, and Gao]{zhang2022simple}
L.~Zhang, J.~Yang, and R.~Gao.
\newblock A short and general duality proof for {W}asserstein distributionally
  robust optimization.
\newblock \emph{Operations Research (Forthcoming)}, 2024.

\bibitem[Zhao and Guan(2018)]{zhao2018data}
C.~Zhao and Y.~Guan.
\newblock Data-driven risk-averse stochastic optimization with {W}asserstein
  metric.
\newblock \emph{Operations Research Letters}, 46\penalty0 (2):\penalty0
  262--267, 2018.

\bibitem[Zhen et~al.(2023)Zhen, Kuhn, and Wiesemann]{zhen2023unified}
J.~Zhen, D.~Kuhn, and W.~Wiesemann.
\newblock A unified theory of robust and distributionally robust optimization
  via the primal-worst-equals-dual-best principle.
\newblock \emph{Operations Research (Forthcoming)}, 2023.

\bibitem[Zorzi(2016)]{zorzi2016robust}
M.~Zorzi.
\newblock Robust {K}alman filtering under model perturbations.
\newblock \emph{IEEE Transactions on Automatic Control}, 62\penalty0
  (6):\penalty0 2902--2907, 2016.

\bibitem[Zorzi(2017)]{zorzi2017robustness}
M.~Zorzi.
\newblock On the robustness of the {B}ayes and {W}iener estimators under model
  uncertainty.
\newblock \emph{Automatica}, 83:\penalty0 133--140, 2017.

\end{thebibliography}
  
\fi

\end{document}